\documentclass[reqno]{amsart}%
\usepackage{amssymb}
\usepackage{amsfonts}
\usepackage{amsmath}
\usepackage{cancel}
\usepackage{xcolor}
\usepackage{verbatim}
\usepackage{mathtools}
\usepackage{graphicx}%
\usepackage[colorlinks=true, pdfstartview=FitV, linkcolor=blue, citecolor=blue, urlcolor=blue]{hyperref}
\usepackage{tikz}
\usepackage{float}
\usepackage{enumitem}

\allowdisplaybreaks[4]

\newtheorem{theorem}{Theorem}[section]
\theoremstyle{plain}

\newtheorem{conjecture}[theorem]{Conjecture}
\newtheorem{corollary}[theorem]{Corollary}

\newtheorem{definition}[theorem]{Definition}
\newtheorem{example}[theorem]{Example}

\newtheorem{lemma}[theorem]{Lemma}

\newtheorem{proposition}[theorem]{Proposition}
\newtheorem{remark}[theorem]{Remark}

\numberwithin{equation}{section}

\newcommand{\Sph}{\mathbb{S}}
\newcommand{\mina}{{\rm MinA}}
\newcommand{\vol}{{\rm vol}}
\newcommand{\Vol}{{\rm Vol}}
\newcommand{\Ric}{{\rm Ric}}
\newcommand{\Scal}{{\rm Scalar}}
\newcommand{\Diam}{{\rm Diam}}
\newcommand{\MinA}{{\rm MinA}}
\newcommand{\Area}{{\rm Area}}

\begin{document}
\title[ ]{Scalar Curvature Compactness for Warped Products on \texorpdfstring{$\Sph^2\times\Sph^1$}{S2 x S1} with Varying Base Metrics}

\author{Changliang Wang}
\address{
School of Mathematical Sciences and Institute for Advanced Study, Key Laboratory of Intelligent Computing and Applications(Ministry of Education), Tongji University, Shanghai 200092, China}
\email{wangchl@tongji.edu.cn}

\author{Zhixin Wang}
\address{Department of Mathematics, Shanghai Jiao Tong University, Shanghai, 201100}
\email{jhin@sjtu.edu.cn}

\begin{abstract}
We study the Gromov--Sormani MinA scalar curvature compactness conjecture for warped
product metrics on $\Sph^2\times\Sph^1$ of the form introduced by Kazaras-Xu in \cite{KazarasXu2023} as follows:
\[
g_i=\varphi_i^{-2}h_i+\varphi_i^2d\xi^2,
\qquad
h_i=dr^2+u_i^2(r)d\theta^2.
\]
Assuming nonnegative
scalar curvature, a uniform volume upper bound, and a positive lower bound for
the areas of closed minimal surfaces, we prove a uniform diameter bound for the
base surfaces $(\Sph^2,h_i)$.  Based on this key estimate, we further obtain compactness of the
base warping functions $u_i$ and local and global estimates for the fiber
warping functions $\varphi_i$.

After passing to a subsequence, the metrics converge in $L^p$, for every
finite $p$, to a limit metric $g_\infty$. 
We also
obtain Gromov--Hausdorff and Sormani--Wenger intrinsic flat subconvergence, and
prove that $g_\infty$ has nonnegative scalar curvature in the distributional
sense of Lee--LeFloch.  Thus the Gromov--Sormani scalar curvature compactness
conjecture is verified for this warped product class. Finally, we construct a
$C^{1,\alpha}$ example illustrating the subtlety of volume-limit tests for
nonnegative scalar curvature in low regularity.
\end{abstract}
\maketitle

\tableofcontents

\section{Introduction}

Scalar curvature lower bounds impose strong global restrictions on smooth
Riemannian manifolds, but their compactness theory is much less understood
than the corresponding theories for sectional or Ricci curvature.  In \cite{Gromov-Plateau}, Gromov proposed a three-dimensional scalar curvature compactness conjecture using intrinsic flat distance introduced by Sormani and Wenger in \cite{SormaniWenger2011}. Motivated by the pulled-string example of  Basilio, Dodziuk and Sormani, constructed by sewing with tunnels \cite{BasilioDodziukSormani-sewing}, the MinA condition in \eqref{defn-MinA} below was added to rule out certain collapsing phenomena. This led to the following MinA scalar compactness conjecture, formulated at an IAS Emerging Topics Workshop co-organized by Gromov and Sormani \cite{Sormani-conjectures}, \cite[Section 4]{Sormani-conjectures}.

\begin{conjecture}[Gromov--Sormani MinA scalar compactness conjecture ]\label{conj: Scalar-Compactness}
Let $\{M_i^3\}$ be a sequence of closed oriented three-dimensional
Riemannian manifolds without boundary satisfying
\begin{equation}
 \Scal_{g_i} \geq 0, \qquad
 \Vol_{g_i}(M_i) \le V, \qquad
 \Diam_{g_i}(M_i) \le D,
\end{equation}
and
\begin{equation}\label{defn-MinA}
\MinA(M_i^3,g_i)
:=
\inf\{\Area(\Sigma) \mid \Sigma \text{ is a closed min. surf. in } M_i^3\}
\ge A_0>0.
\end{equation}
Then a subsequence converges in the volume-preserving intrinsic flat sense to a
three-dimensional rectifiable limit space $M_\infty$.  Furthermore,
$M_\infty$ is expected to be a connected geodesic metric space with Euclidean
tangent cones almost everywhere and with a generalized notion of nonnegative
scalar curvature.
\end{conjecture}

Several special cases and examples have clarified both the strength and the
subtlety of this conjecture.  The conjecture was verified in the rotationally symmetric setting by Park, Tian and Wang
\cite{ParkTianWang2018}.  Warped product circles over the round sphere were
studied by Tian and Wang \cite{TianWang2024}, who proved $L^q$-compactness of the metric
tensors, and Sobolev regularity and nonnegative distributional scalar curvature of
the limit.  A related extreme example constructed by Sormani, Tian and Wang \cite{SormaniTianWang2024} shows that one cannot expect uniform upper
bounds for the warping factors, and that the Sobolev regularity obtained in \cite{TianWang2024} is sharp. The extreme example was further generalized by Tian \cite{Tian2024}, and the corresponding geometric convergence was established by Sormani, Tian and Yeung \cite{SormaniTianYeung2025}. 

Allowing the base metric to vary, and partially inspired by the  higher dimensional examples of Lee-Naber-Neumayer \cite{LeeNaberNeumayer2023}, Kazaras and Xu \cite{KazarasXu2023} constructed drawstring examples on $\Sph^2 \times \Sph^1$ satisfying  all hypotheses in Conjecture \ref{conj: Scalar-Compactness}, whose limit exhibits a failure of a
volume-limit notion of nonnegative scalar curvature at a collapsing point.  The nonnegative scalar curvature compactness problem has also been studied in conformal classes by Allen, Tian and Wang \cite{AllenTianWang2026}.  These works indicate
that, even in highly structured classes, the correct notion of convergence and the correct
weak notion of scalar curvature are both delicate.

The purpose of this paper is to study Conjecture
\ref{conj: Scalar-Compactness} for a broader class of warped product metrics on
$\Sph^2\times\Sph^1$.  We consider smooth Riemannian metrics of the form
\begin{equation}\label{metric-form}
 g_i = \frac{1}{\varphi^2_i} h_i + \varphi^2_i d\xi^2, \quad \xi \in [0, 2\pi],
\end{equation}
where $\varphi_i \in C^{\infty}(\Sph^2)$ and $h_i$ are rotationally symmetric metrics on $S^2$ as 
\begin{equation}\label{eqn: metric-on-base}
h_i = dr^2 + u^2_i d\theta^2, \quad r\in [0, a_i], \ \ \theta\in[0, 2\pi].
\end{equation}
Here $u_i\in C^\infty([0,a_i])$ satisfies the usual smoothness conditions at the
poles,
\begin{equation}
 u_i(0)=u_i(a_i)=0, \qquad u_i'(0)=1, \qquad u_i'(a_i)=-1,
\end{equation}
and $u_i>0$ on $(0,a_i)$.  Throughout the paper we assume that the sequence
$M_i=(\Sph^2\times\Sph^1,g_i)$ satisfies the following uniform hypothesis,
which we denote by $\Lambda$:
\begin{equation}\label{eqn: condition-lambda-intro}
\begin{cases}
\operatorname{Scalar}_{g_i} \geq 0, \\[1ex]
\Vol_{g_i}(M_i) \leq V_0, \\[1ex]
\MinA \left(M_i\right) =\inf \left\{\Area(\Sigma) \mid \Sigma \text{ closed min. surf. in } M_i\right\} \geq A_0>0.
\end{cases}
\end{equation}

The model \eqref{metric-form} is particularly interesting because it already
exhibits subtle limiting phenomena. Building on this ansatz, Kazaras and Xu
\cite{KazarasXu2023} constructed drawstring sequences satisfying the
hypotheses above.  These sequences converge, both in the Gromov--Hausdorff
sense and in the volume-preserving intrinsic flat sense, to pulled-string
spaces obtained from $\Sph^2\times\Sph^1$ by pulling along a circle
$\{p\}\times\Sph^1$.  In particular, the resulting limit spaces  fail to satisfy a volume-limit notion of
nonnegative scalar curvature. Their example
therefore illustrates the importance of identifying a suitable weak formulation
of nonnegative scalar curvature for limits of such sequences. 
The volume-limit notion of nonnegative scalar curvature will be discussed in \ref{sect: c1alpha-example}.

Our results show that, although volume-limit notion of nonnegative scalar
curvature may fail to be stable under such limits, the limit nevertheless
retains nonnegative scalar curvature in the distributional sense; see Theorem
\ref{thm2-NNSC} below.  
{This provides further evidence that distributional
scalar curvature might be a natural and stable weak formulation in the study of scalar curvature
compactness.}

\medskip

A useful feature of the ansatz \eqref{metric-form} is that the scalar curvature
has the simple form
\begin{equation}\label{eqn:intro-scalar-formula}
\Scal_{g_i}
=
\varphi_i^2
\left(
\Scal_{h_i}
-
2\frac{|\nabla^{h_i}\varphi_i|^2}{\varphi_i^2}
\right).
\end{equation}
Thus the scalar curvature inequality controls first derivatives of
$\varphi_i$ through the scalar curvature of the base.  This first-order feature
is essential throughout the paper.  It allows us to extract geometric
information about $(\Sph^2,h_i)$ and analytic information about both warping
functions $u_i$ and $\varphi_i$.

\medskip
There are several difficulties which do not appear in the previously studied
case where the base metric is fixed.   The first and most important one is that the diameters of the base
surfaces $(\Sph^2,h_i)$ are not assumed to be uniformly bounded.  Since
\begin{equation}
\Diam_{h_i}(\Sph^2)=a_i,
\end{equation}
compactness of the functions $u_i$ cannot even be formulated on a fixed domain
unless the intervals $[0,a_i]$ are uniformly controlled.  We overcome this
difficulty by proving the following uniform two-sided diameter bound.

\begin{theorem}[Uniform diameter bound]\label{thm-main1}
  Let $\{g_i\}$ be a sequence of Riemannian metrics of the form \eqref{metric-form} on $\Sph^2 \times \Sph^1$ satisfying the condition $\Lambda$ in \eqref{eqn: condition-lambda-intro}. Then there exists a positive constant $C(\Lambda)$, depending only on $V_0$ and $A_0$, such that 
       \begin{equation}
       C(\Lambda)^{-1} \leq {\rm Diam}_{h_i}(\Sph^2) = a_i \leq C(\Lambda), \ \ \forall i \in \mathbb{N}.
       \end{equation} 
\end{theorem}

Theorem \ref{thm-main1} is proven in Section \ref{subset: diameter-bound}. The upper
bound is the delicate part.  The argument is by contradiction: if
$a_i\to\infty$, then the volume and $\MinA$ bounds force the average of
$\ln(1/\varphi_i)$ to tend to $-\infty$.  After subtracting this average and
rescaling the base interval to $[0,2]$, one obtains functions $f_i$ satisfying a
universal differential inequality of the form
\begin{equation}
 |\nabla f_i|^2\le -\frac{w_i''}{w_i}.
\end{equation}
A key integral ratio estimate for $e^{f_i}$ and $e^{2f_i}$ gives a uniform
upper bound, contradicting the divergence forced by the volume and $\MinA$
assumptions.

In view of Theorem \ref{thm-main1}, we may normalize the radial parameter so
that $a_i=2$ for all $i$.  Under this normalization all functions $u_i$ are
defined on the common interval $[0,2]$, and the constants in the estimates
remain controlled by $\Lambda$.  We use this normalization in the
statements below.

\begin{theorem}[Convergence of the warping functions]\label{thm-main2}
   Let $\{g_i\}$ be a sequence of Riemannian metrics of the form \eqref{metric-form} on $\Sph^2 \times \Sph^1$ satisfying the condition $\Lambda$ in \eqref{eqn: condition-lambda-intro}, and assume that $a_i=2$ for all $i$.
Then, after passing to a subsequence, the following hold.
\begin{enumerate}[label=\textup{(\arabic*)}, leftmargin=*, itemsep=2pt]
\item   The base warping functions converge uniformly:
            \begin{equation}
            u_i \to u_\infty \quad \text{uniformly on} \ \ [0, 2].
            \end{equation}
            Moreover, $u_i'\to u_\infty'$ almost everywhere on $(0,2)$, and
$u_\infty$ is concave and $1$-Lipschitz, with
\begin{equation}
 u_\infty(0)=u_\infty(2)=0,
 \qquad
 \max_{[0,2]}u_\infty\ge \frac{A_0}{4\pi^2}>0.
\end{equation}

\item    For every $0 < \epsilon <1$, let 
             \begin{equation}
           \Omega_\epsilon
:=
\{(r,\theta)\mid \epsilon\le r\le 2-\epsilon,\ 0\le\theta\le 2\pi\}
\subset \Sph^2 .
             \end{equation} 
             
Then, for every $\alpha<1/2$,
\begin{equation}\label{NNSC-convergence-local}
 \varphi_i\to\varphi_\infty
 \quad\text{in } C^{0,\alpha}(\Omega_\epsilon),
\end{equation}
and
\begin{equation}
 \varphi_i\rightharpoonup\varphi_\infty
 \quad\text{weakly in } W^{1,2}(\Omega_\epsilon).
\end{equation}
Furthermore, after passing to a further subsequence if necessary,
\begin{equation}\label{eqn: measure-convergence}
 |\nabla\varphi_i|^2\,dr\,d\theta
 \rightharpoonup
 |\nabla\varphi_\infty|^2\,dr\,d\theta
 +
 \sum_{j\in J} c_j\delta_{x_j}
\end{equation}
as Radon measures on $\Omega_\epsilon$, where $J$ is at most countable and
$c_j\ge0$.

\item Globally on $(\Sph^2,g_{\Sph^2})$, for every $1\le p<\infty$ and every
$1\le q<2$,
\begin{equation}\label{NNSC-convergence-global}
 \varphi_i\to\varphi_\infty
 \quad\text{strongly in } L^p(\Sph^2,g_{\Sph^2}),
 \qquad
 \varphi_\infty\in W^{1,q}(\Sph^2,g_{\Sph^2}).
\end{equation}

\item For every $k \in \mathbb{N}$, there exists a constant $C=C(\Lambda, k)$ such that
\begin{equation}
C^{-1} r^{\frac{1}{k}} 
    \leq
    \varphi_i(r, \theta)
    \leq 
    C r^{-\frac{1}{k}}, \quad \forall (r, \theta) \in (0, 1] \times [0, 2\pi],
\end{equation}
and
\begin{equation}
    C^{-1} (2 - r)^{\frac{1}{k}} 
    \leq
    \varphi_i(r, \theta)
    \leq 
    C (2 -r)^{-\frac{1}{k}}, \ \ \forall (r, \theta) \in [1, 2) \times [0, 2\pi].
\end{equation}

 The same bounds hold for $\varphi_\infty$.  In particular,
\begin{equation}
 0<\varphi_\infty(r,\theta)<\infty
 \qquad\text{for every } (r,\theta)\in(0,2)\times[0,2\pi].
\end{equation}
\end{enumerate}
\end{theorem}

Section \ref{sect: warping-function-convergence} is devoted to the proof of Theorem \ref{thm-main2}.  Away from the poles $r=0$ and $r=2$, the scalar curvature inequality gives uniform two-sided bounds and $C^{0,1/2}$ control for $\varphi_i$, together with local $W^{1,2}$ estimates. {The convergence in \eqref{eqn: measure-convergence} then follows from the standard concentration-compactness lemma; see, Lions \cite[Lemma I.1]{Lions1985} and also the book of Hebey \cite[Chapters 4 and 7]{Hebey1999book}.}  Near the two poles, the warping functions $\varphi_i$ may grow or decay, as demonstrated by the extremal example of Sormani-Tian-Wang \cite{SormaniTianWang2024} and the drawstring example Kazaras-Xu \cite{KazarasXu2023}.  We prove that this possible degeneration is nevertheless very mild: for each integer $k$ one has sub-polynomial bounds of the form 
\begin{equation}
r^{1/k}\lesssim \varphi_i\lesssim r^{-1/k}
\end{equation} 
near the pole $r=0$, and analogous bounds near the other; see item (4) in Theorem \ref{thm-main2}. These endpoint bounds, together with local estimates away from the poles, yield global $W^{1,p}$ control for every $p<2$, as stated in item (3) in Theorem \ref{thm-main2}.

We use the limiting warping functions obtained in Theorem \ref{thm-main2} to define a pointwise Riemannian metric on the regular part
\begin{equation}
 \mathring M
 :=
 \{0<r<2\}\times\Sph^1_\theta\times\Sph^1_\xi
 \subset \Sph^2\times\Sph^1,
\end{equation}
by
\begin{equation}\label{eqn: limit-metric-g-infty}
 g_\infty
 =
 \varphi_\infty^{-2}
 \bigl(dr^2+u_\infty^2d\theta^2\bigr)
 +
 \varphi_\infty^2d\xi^2.
\end{equation}
The preceding estimates imply the convergence of the metric tensors.

\begin{theorem}[Analytic convergence of the metrics]\label{thm-analytic-convergence}
Let $\{g_i\}$ be as in Theorem \ref{thm-main2}.  Then, for every
$1\le p<\infty$,
\begin{equation}
 g_i\to g_\infty
 \quad\text{in } L^p(M,g_0)
\end{equation}
as symmetric $2$-tensors, where $g_0$ is the standard product metric on
$\Sph^2\times\Sph^1$.  Moreover, on every interior region
$M_\epsilon:=\Omega_\epsilon\times\Sph^1$, one has local uniform convergence of
the coefficients of $g_i$ to those of $g_\infty$.
\end{theorem}

The next result gives the metric and intrinsic flat compactness conclusion.
It is formulated in terms of the limiting local distance $d_*$, because the
global limiting distance may be shorter than the intrinsic distance of
$g_\infty$.

\begin{theorem}[Gromov--Hausdorff and Sormani--Wenger intrinsic flat convergence]\label{thm-main3}
Let $\{g_i\}$ be a sequence of Riemannian metrics of the form \eqref{metric-form} on $\Sph^2 \times \Sph^1$. Assume that $(M_i, g_i) = (\Sph^2\times \Sph^1, g_i)$ satisfy the condition $\Lambda$ in \eqref{eqn: condition-lambda-intro}.  Then, after passing to a subsequence, the following hold.
\begin{enumerate}[label=\textup{(\arabic*)}, leftmargin=*, itemsep=2pt]
\item There exists a constant $C(\Lambda)$ such that
          \begin{equation}
          \Diam_{g_i}(\Sph^2 \times \Sph^1) \leq C(\Lambda), \quad \forall i \in \mathbb{N}.
          \end{equation}
\item The volumes converge:
         \begin{equation}
          \Vol_{g_i}(\mathring{M}) \to \Vol_{g_\infty}(\mathring{M}), \quad \text{as} \ \ i \to \infty.
         \end{equation}
\item There exists a metric $d_*$
         on $\mathring{M}$ such that, 
          \begin{equation}
         d_{g_i}\big|_{K\times K}\to d_*
         \quad\text{uniformly on }K\times K
         \end{equation}
         for every compact set $ K\Subset \mathring{M}$. Moreover,
         \begin{equation}
         d_* \leq d_{g_\infty}
        \quad\text{on }\mathring{M}\times \mathring{M},
       \end{equation}
        and \(d_*\) agrees locally with \(d_{g_\infty}\). In particular,
        \(d_*\) induces the same local topology and the same local Riemannian current
         structure as \(g_\infty\).
 \item Let $\overline{(\mathring{M}, d_*)}$ denote the metric completion of $(\mathring{M}, d_*)$. Then
            \begin{equation}
               (M_i, d_{g_i}) \longrightarrow \overline{(\mathring{M}, d_*)}
           \end{equation}
           in the Gromov--Hausdorff sense.
 \item  Let
          $
         (X,d_X,T)
        $
         be the settled completion of the integral current space associated to
          $(\mathring{M}, d_*)$, with current $T$ induced locally by the oriented
        Riemannian metric \(g_\infty\). Then
          \begin{equation}
          (M_i,d_{g_i},[\![M_i]\!])
          \longrightarrow
          (X,d_X,T)
           \end{equation}
          in the Sormani--Wenger intrinsic flat sense.
\end{enumerate}
\end{theorem}

Theorem \ref{thm-main3} is proved in Section \ref{sect: gh-if-convergence}.
In general, the inequality $d_*\le d_{g_\infty}$ can be strict.  This may happen
when regions collapse or disappear in the limit, creating shortcuts for the
limiting global distance. To illustrate this phenomenon, in Appendix \ref{app:shortcut-example}, we provide an example of drawstring type of Kazaras-Xu \cite{KazarasXu2023}, in which $d_* < d_{g_\infty}$ because shortcuts form near the collapsing region; see Lemma \ref{lem: example-compare-distances}.  Consequently, the completion
$\overline{(\mathring M,d_*)}$ need not be isometric to the metric completion of
$(\mathring M,d_{g_\infty})$, and the intrinsic flat limit need not be the
integral current space directly associated to the metric completion of
$(\mathring M,g_\infty)$.

Our final compactness theorem concerns scalar curvature of the limit.  The limiting
metric $g_\infty$ generally has only weak regularity, so scalar curvature is not defined
pointwise.  We use the distributional scalar curvature framework of
Lee--LeFloch \cite{LeeLeFloch2015}, building on the broader distributional
curvature formalism of LeFloch--Mardare \cite{LeFlochMardare2007}.  This belongs to
the same circle of ideas as the low regularity scalar curvature literature on
metrics with corners, studied by Miao \cite{Miao2002}, McFeron--Szekelyhidi \cite{McFeronSzekelyhidi2012} and others, and Ricci flow based weak scalar curvature bounds, studied by Burkhardt--Guim \cite{BurkhardtGuim2019}, Jiang--Sheng--Zhang \cite{JiangShengZhang2023} and others.

In our setting we compute explicitly the Lee--LeFloch vector field $V$ and the
lower order term $F$ for the ansatz \eqref{metric-form}.  On compact interior
regions, weak lower semicontinuity gives nonnegative distributional scalar
curvature.  Passing from interior regions to all of $M$ requires a careful
boundary analysis near the poles; the sub-polynomial estimates for
$\varphi_i$ obtained in Theorem \ref{thm-main2} make the boundary terms vanish. Combining these ingredients, we obtain the following theorem.

\begin{theorem}[Nonnegative distributional scalar curvature]\label{thm2-NNSC}
The limit metric $g_\infty$ has nonnegative scalar curvature in the
distributional sense of Lee--LeFloch on $M=\Sph^2\times\Sph^1$.
\end{theorem}

Theorems \ref{thm-main3} and \ref{thm2-NNSC} together confirm Conjecture \ref{conj: Scalar-Compactness} for the warped product class of form as in \eqref{metric-form}.  More precisely, in this setting the
diameter bound required in the conjecture follows from the volume and
\(\MinA\) assumptions, the sequence subconverges in both the
Gromov--Hausdorff and Sormani--Wenger intrinsic flat senses, and the limit
retains nonnegative scalar curvature in the distributional sense of
Lee--LeFloch.

The last section of the paper discusses a complementary issue: how one should test
nonnegative scalar curvature for low regularity metrics using volume
comparison.  In the smooth setting, the small geodesic ball volume expansion
\begin{equation}
 \Vol_g(B_g(p,t))
 =
 \frac{4\pi}{3}t^3
 \left(
 1- \frac{\Scal_g(p)}{30}t^2+O(t^4)
 \right)
\end{equation}
shows that nonnegative scalar curvature is equivalent to an infinitesimal
volume deficit relative to Euclidean balls.  This is closely related to
Gromov's volumic viewpoint on scalar curvature and its subsequent developments
\cite{Gromov2017Questions,Deng2021Volumic}.

However, for low regularity metrics one must be careful about replacing
geodesic balls by exponential images of prescribed tangent space regions.  We
construct a family of $C^{1,\alpha}$ metrics with nonnegative scalar curvature
on the smooth part, for which a symmetric fixed region in the tangent space can
detect sectorwise off-diagonal Ricci contributions.  The key point is that the
metric is modeled near the singular set on different smooth sectors, for
example by warping factors of the form $1\pm kr\theta$.  Thus the usual
cancellation of off-diagonal Ricci terms over a symmetric region may fail for
exponential images of that fixed region.

This does not contradict the geodesic ball test.  A geodesic ball is defined by
the intrinsic distance condition $d_g(p,\cdot)<t$, and its preimage under the
exponential map is not a fixed Euclidean ball or sector.  The deformation of
this preimage contributes an additional boundary term, which cancels the
artificial off-diagonal contribution.   This also
illustrates the subtle difference between various possible weak notions of
scalar curvature lower bounds for nonsmooth metrics.

\medskip
{\bf Organization of the paper.}
In Section \ref{sect: diameter-bound}, we derive the basic consequences
of the hypotheses and prove the uniform diameter bound for the base spaces.
This is the main compactness input for the rest of the paper.  In Section
\ref{sect: warping-function-convergence}, we prove convergence of the base
warping functions $u_i$ and obtain local and global estimates for the fiber
warping functions $\varphi_i$.  In Section \ref{sect: gh-if-convergence}, we prove the Gromov--Hausdorff and
intrinsic flat convergence statements.  In Section \ref{sect: distributional-scalar},
we prove that the limit metric has nonnegative distributional scalar curvature.
Finally, in Section \ref{sect: c1alpha-example}, we construct the
$C^{1,\alpha}$ example and discuss the volume-limit test for nonnegative scalar
curvature in low regularity test for nonnegative scalar
curvature in low regularity.


\medskip
{\bf Acknowledgements.}
 The authors are grateful to Professor Christina Sormani for her helpful comments and suggestions on an earlier version of this paper. They also thank Demetre Kazaras for his interest and discussions on this work.
 CW is partially supported by the Natural Science Foundation of Shanghai (Grant No. 25ZR1401357) and the Fundamental Research Funds for the Central Universities. 
 ZW is supported by the China Postdoctoral Science Foundation (Grant No. 2025T180838).


\section{Uniform diameter bounds for the base spaces}\label{sect: diameter-bound}

This section establishes the preliminary compactness estimates needed for the warped product metrics in \eqref{metric-form}.  In Section \ref{subset: basic-consequences}, we first derive several basic consequences of the scalar curvature, volume, and $\MinA$ hypotheses.  In Section \ref{subset: diameter-bound}, we prove the main and the most delicate estimate of this section: a uniform diameter bound for the base surfaces $(\Sph^2,h_i)$. This bound prevents the rotational parameter interval $[0,a_i]$ from degenerating or escaping to infinity. It is therefore essential for the convergence theory of the base warping functions $u_i$, and it also plays an important role in deriving estimates for the fiber warping functions $\varphi_i$ in Section \ref{sect: warping-function-convergence}.

\subsection{Basic consequences of the geometric hypotheses}\label{subset: basic-consequences}
We begin by deriving several elementary consequences of the uniform geometric bounds in Theorem \ref{thm-main1}.  These estimates convert the scalar curvature condition into a differential inequality for the base warping function $u_i$ and the fiber factor $\varphi_i$, and they relate the $\MinA$ condition to a noncollapsing estimate for $u_i$.

The first ingredient is the following scalar curvature formula, taken from \cite{KazarasXu2023} and included here for completeness.

\begin{lemma}[Nonnegative scalar curvature condition]\label{lem: scalar-curvature-formula}
The scalar curvature of the warped product metric $g_i$ on $\Sph^2 \times \Sph^1$ as in \eqref{metric-form} is given by 
    \begin{equation}\label{eqn: scalar-curvature}
        \Scal_{g_i} =\varphi^2_i\left( \Scal_{h_i}-2\tfrac{\left|\nabla^{h_i} \varphi_i\right|^2}{\varphi^2_i} \right)
    \end{equation}
   Thus $g_j$ has nonnegative scalar curvature if and only if
   \begin{equation}
   2\tfrac{\left|\nabla^{h_i} \varphi_i\right|^2}{\varphi^2_i} \leq \Scal_{h_i} = -2 \frac{u^{\prime\prime}_i}{u_i}.
   \end{equation}
\end{lemma}
    \begin{proof}
  Let $g^\prime_i=\frac{1}{\varphi^4_i}h_i+d\xi^2$. Then $g_i = \varphi^2_i g^\prime_i$. By using the scalar curvature formula under conformal changes twice, we have
\begin{equation}
\Scal_{g^{\prime}_i} = \Scal_{1/\varphi^4_i h_i}=\varphi^4_i\left( \Scal_{h_i}+4 \Delta_{h_i} \ln \varphi_i\right) 
\end{equation}
\begin{equation}
\begin{aligned}
\Scal_{g_i}&= \frac{1}{\varphi^2_i}\left(\Scal_{g^{\prime}_i}-4 \Delta_{g^{\prime}_i}\ln \varphi_i-2\left|\nabla^{g^{\prime}_i} \ln \varphi_i\right|^2\right) \\
&= \frac{1}{\varphi^2_i}\left(\varphi^4_i R_{h_i} + 4 \varphi^2_i \Delta_{h_i}\ln \varphi_i -4 \varphi^4 \Delta_h(\ln \varphi_i)-2 |\nabla^{g_i} \ln \varphi|^2\right) \\
&= \varphi^2_i\left(R_{h_i}-2 \tfrac{\left|\nabla^{h_i} \varphi_i\right|^2}{\varphi^2_i}\right).
\end{aligned}
\end{equation}
So (\ref{eqn: scalar-curvature}) is proved.
\end{proof}

We next derive several consequences of the $\mina$ condition. 

\begin{lemma}\label{lem: torus-mean-curvature}
Let $\Gamma$ be a smooth embedded closed curve in $\Sph^2$. The mean curvature of of $\Gamma\times \Sph^1$ in $(\Sph^2 \times \Sph^1, g_i)$ as in \eqref{metric-form} is given by
\begin{equation} 
   H_i =\varphi_i \kappa_i
 \end{equation}
  where $\kappa_i$ is the geodesic curvature of $\Gamma$ in $(\Sph^2, h_i)$. 
As a result, every closed geodesic $\Gamma\subset (S^2,h_i)$ gives a minimal surface $\Gamma\times \Sph^1$ in $(\Sph^2 \times \Sph^1, g_i)$. 
\begin{equation}
\Area_{g_i}(\Gamma\times \mathbb{S}^1) = 2 \pi L_{h_i}(\Gamma_i)
\end{equation}
 \end{lemma}
 \begin{proof}
Let $g^\prime_i=\frac{1}{\varphi^4_i}h_i+d\xi^2$. Then $g_i = \varphi^2_i g^\prime_i$.  Let $H^\prime_i$ be the mean curvature of $\Gamma\times S^1$ in $(\Sph^2 \times \Sph^1, g^\prime_i)$ and $\nu_i$ the normal vector for $\Gamma$ in $(S^2,h_i)$, then
\begin{equation}
    H^\prime_i =\kappa_{1/\varphi^4h_i}= \varphi^2_i(\kappa_i -2\nu_i(\ln\varphi_i)),
\end{equation}
and so
\begin{equation}
    H_i =\frac{1}{\varphi}(H^\prime_i+2\varphi^2_i \nu_i(\ln\varphi_i))=\varphi_i \kappa_i.
\end{equation}
\end{proof}

The scalar curvature formula in Lemma \ref{lem: scalar-curvature-formula} and the minimal surface construction in Lemma \ref{lem: torus-mean-curvature} give  the following quantitative lower bound for the maximum of the base warping function.

\begin{lemma}\label{lem: mina-consequence-1}
Let $g_i$ be the warped product Riemannian metrics on $ \Sph^2 \times \Sph^1$ as in \eqref{metric-form} satisfying 
\begin{equation}
\Scal_{g_i} \geq 0, \ \ \text{and} \ \ \mina(\Sph^2 \times \Sph^1, g_i) \geq A_0 >0.
\end{equation}
Then $u_i$ are convave functions on $[0, a_i]$ with $u_i(0) = u_i(a_0) =0$, and so each $u_i$ has a unique maximal point $r_i \in (0, a_i)$. Moreover, 
\begin{equation}
\max\limits_{[0, a_i]}u_i = u_i(r_i)\geq \frac{A_0}{4\pi^2}.
\end{equation}
\end{lemma}
\begin{proof}
By Lemma \ref{lem: scalar-curvature-formula}, $\Scal_{g_i} \geq 0$ implies
\begin{equation}
-\frac{2u^{\prime\prime}_i(r)}{u_i(r)} = \Scal_{h_i} \geq \frac{|\nabla^{h_i} \varphi_i|}{\varphi_i} \geq 0 \quad \text{on} \ \ (0, a_i).
\end{equation}
Moreover, smoothness of the metric $h_j$ as in (\ref{eqn: metric-on-base}) implies $u_i(0) = u_0(a_i) =0$ and $u_{i}(r)>0$ on $(0, a_i)$. Thus $u^{\prime\prime}(r)_i \leq 0$ and $u_i$ has a unique critical point $r_i$ that is the maximal point on $(0, a_i)$. The curve $\Gamma_i := \{(r, \theta) \mid r= r_i\} \subset \Sph^2$ is a geodesic in $(\Sph^2, h_i)$. Then by Lemma \ref{lem: torus-mean-curvature}, $\Gamma_i \times \Sph^1$ is a minimal surface in $(\Sph^2 \times \Sph^1, g_i)$. Therefore,
\begin{equation}
2\pi u_i(r_i) = L_{h_i}(\Gamma_i) = \frac{{\rm Area}_{g_i}(\Gamma_i \times \Sph^1)}{2\pi} \geq \frac{A_0}{2\pi}.
\end{equation}
\end{proof}

In addition, since $\varphi_i$ is independent of $\xi$, each slice $S_{\xi_0} := \{(r, \theta, \xi) \mid \xi = \xi_0\}$ is a minimal surface in $(\Sph^2 \times \Sph^1, g_i)$ for every fixed $\xi_0 \in [0, 2\pi]$. As a result, we have
\begin{lemma}\label{lem: mina-consequence-2}
Let $g_i$ be the warped product Riemannian metrics on $M_i = \Sph^2 \times \Sph^1$ as in \eqref{metric-form} satisfying 
$
\mina(M_i, g_i) \geq A_0 >0.
$
Then 
\begin{equation}
\int_{\Sph^2} \frac{1}{\varphi^2_i} d\vol_{h_i} \geq A_0.
\end{equation}
\end{lemma}

The volume bound gives the complementary $L^1$ control for $\varphi_i^{-1}$.  Since
\begin{equation}
\Vol_{g_i}(\Sph^2 \times \Sph^1)=2\pi\int_{\Sph^2}\frac{1}{\varphi_i}\,d\vol_{h_i},
\end{equation}
we have
\begin{lemma}\label{lem: volume-bound-consequence}
Let $g_i$ be the warped product Riemannian metrics on $ \Sph^2 \times \Sph^1$ as in \eqref{metric-form} satisfying 
\begin{equation}
{\rm Vol}_{g_i} (\Sph^2 \times \Sph^1) \leq V_0.
\end{equation}
Then
\begin{equation}
\int_{\Sph^2} \frac{1}{\varphi_i} d\vol_{h_i} \leq \frac{V_0}{2\pi}.
\end{equation}
\end{lemma}


\subsection{Uniform boundedness of the diameter of the base surfaces}\label{subset: diameter-bound}
We now prove the key estimate of this section: the base surfaces $(\Sph^2,h_i)$ have uniformly bounded diameters. 

We first recall a basic fact about rotationally symmetric metrics.

\begin{lemma}\label{lem: diameter-of-base}
Let 
\begin{equation}
h:= dr^2 + u^2 d\theta^2, \quad r \in [0, a], \ \ \theta \in [0, 2\pi], \ \ 0 \leq u \in C^\infty([0, a])
\end{equation}
 be a rotationally symmetric smooth Riemannian metric on $\Sph^2$. Then the diameter of $(\Sph^2, h)$ is 
 \begin{equation}
 {\rm Diam}_h(\Sph^2) = a.
 \end{equation}
\end{lemma}

\begin{proposition}\label{prop: diameter-bound}
Let $g_i$ be a sequence of Riemannian metrics as in \eqref{metric-form} satisfying condition $\Lambda$:
\begin{equation}
 \Scal_{g_i}\geq 0, \ \  {\rm Vol}_{g_i}(\Sph^2 \times \Sph^1)\leq V, \ \  \mina(\Sph^2 \times \Sph^1, g_i) \geq A_0>0.
\end{equation}
Then there exists a positive constant $C(\Lambda)$ depending only on $V$ and $A_0$ such that
\begin{equation}
C(\Lambda)^{-1} \leq {\rm Diam}_{h_i}(\Sph^2) = a_i \leq C(\Lambda).
\end{equation}
\end{proposition}

The proof uses the following elementary inequalities.
\begin{lemma}\label{lem: elementary-inequality}
\begin{equation}\label{claim-1-1}
2 \sqrt{x+y-y^2}-y \leq 2 x+ \frac{5}{8},
\end{equation}
and
\begin{equation}\label{claim-1-3}
 \sqrt{x+y-y^2}-y \leq  x+\frac{1}{4}
\end{equation}
for $x\geq y^2-y$. 
\end{lemma}
\begin{proof} 
By the elementary inequality $2\sqrt{t} \leq 2t + \frac{1}{2}$ for $t \geq 0$, we have
\begin{equation}
2 \sqrt{x+y-y^2}-y \leq 2(x + y - y^2) + \frac{1}{2} -y = 2x + \left(y - 2y^2 + \frac{1}{2} \right) \leq 2x + \frac{5}{8}.
\end{equation}
The second inequality can be derive simiarly.
For the third one, 
\begin{equation}
    x+y-\sqrt{x+y-y^2}\geq x+y-\sqrt{x+y}\geq -0.25
\end{equation}
\end{proof}

\begin{lemma}\label{lem: concave-function-properties}
Let $w(t)$ be a smooth concave function on $[0, 2]$ satisfying
\begin{equation}
w(0) = w(2) =0, \quad \dot{w}(0) = 1, \quad \text{and} \quad \dot{w}(2)=-1,
\end{equation}
where $\dot{w}$ denotes derivative of $w$.
Then
\begin{equation}\label{eqn: estimate-of-w}
0\leq w(t) \leq \min\{t, 2-t\}\leq 1, \ \  
\text{and} \ \ -1 \leq \dot{w}(t) \leq 1 \quad \forall t \in [0, 2].
\end{equation}
Moreover, for any $t_0 \in (0, 2)$, 
\begin{equation}
w(t) \leq 
  \begin{cases}
  \frac{2-t}{2-t_0} w(t_0), & 0 \leq t \leq t_0, \cr
  \frac{t}{t_0} w(t_0), & t_0 \leq t \leq 2,
  \end{cases}
\end{equation}
\begin{equation}
 w(t) \geq
 \begin{cases}
 \frac{t}{t_0}w(t_0), & 0 \leq t \leq t_0, \cr
 \frac{2 - t}{2-t_0}w(t_0), & t_0 \leq t \leq 2,
 \end{cases}
\end{equation}
and 
\begin{equation}
|\dot{w}|(t) \leq \frac{1}{\epsilon} \max\limits_{[0, 2]} w , \quad \text{and} \quad w(t) \geq \tfrac{\epsilon}{2}\max\limits_{[0, 2]} w, \quad \forall t \in \left[\epsilon, 2-\epsilon \right], \quad 0< \epsilon <2.
\end{equation}
If we further assume that $\max\limits_{[0, 2]}w(t) \geq A >0$, then we have that
\begin{equation}\label{eqn: concave-function-estimate-by-linear-1}
    w(t) \geq 
    \begin{cases}
        \frac{A}{2}t, & 0 \leq t \leq 1, \cr
        \frac{A}{2}(2-t), & 1 \leq t \leq 2.
    \end{cases}
\end{equation}
\end{lemma}
\begin{proof}
The proof is elementary, and all of these properties can be established in a similar way. We therefore prove only estimate (\ref{eqn: concave-function-estimate-by-linear-1}) for $0 \leq t \leq 1$.

For any $t \in [0, 1]$, since $t = t \cdot 1 + (1-t) \cdot 0$, the concavity of $w$ gives
\begin{equation}
    w(t) \geq t (1) + (1-t)w(0) = tw(1).
\end{equation}
Then it suffices show that $w(1) \geq \frac{A}{2}$. Let $s\in (0, 2)$ such that $w(s) = \max\limits_{[0, 2]}w \geq A$. 
If $0 \leq s \leq 1$, then since $1 = \frac{1}{2-s} \cdot s + \frac{1-s}{2-s} \cdot 2$, the concavity of $w$ gives 
\begin{equation}
    w(1) \geq \frac{1}{2-s} \cdot w(s) + \frac{1-s}{2-s} \cdot w(2)
    = \frac{1}{2-s} \cdot w(s)
    \geq \frac{A}{2}.
\end{equation}
If $ 1 \leq s \leq 2$, then similarly, $1= \frac{1}{s} \cdot s + \frac{s-1}{s} \cdot 0$ and the concavity of $w$ again imply $w(1) \geq \frac{A}{2}$. This proves $w(t) \geq \frac{A}{2}t$ for $0 \leq t \leq 1$.
\end{proof}

\begin{lemma}\label{lem: oscillation-f}
Let 
\begin{equation}
h = dt^2 + w^2(t) d\theta^2, \quad t \in [0, 2], \ \ \theta \in [0, 2\pi]
\end{equation}
be a smooth rotationally symmetric Riemannian metric on $\Sph^2$, and $f$ a smooth function on $\Sph^2$ satisfying
\begin{equation}\label{eqn: curvature-condition-w}
 |\nabla^h f|^2(t, \theta) \leq - \frac{\ddot{w}(t)}{w(t)}, \quad \forall t \in [0, 2] \ \ \text{and} \ \ \theta \in [0, 2\pi],
\end{equation}
Then for any fixed $0 < \epsilon <1$, we have that
\begin{equation}
|f(t_0, \theta_1) - f(t_1, \theta_2)| \leq \tfrac{4\sqrt{2(1-\epsilon)}}{\epsilon} + \tfrac{2\pi}{\sqrt{1-\epsilon}}
\end{equation}
holds for any $t_0,t_1 \in [\epsilon, 2-\epsilon]$ and $\theta_1, \theta_2 \in [0, 2\pi]$.

In particular, by letting $\epsilon = \frac{1}{2}$ and $t_0 =1$, we obtain
\begin{equation}
|f(1, \theta_1) - f(1, \theta_2)| \leq 8 + 2\sqrt{2}\pi, \quad \forall \theta_1, \theta_2 \in [0, 2\pi].
\end{equation}
\end{lemma}
\begin{proof}
By the smoothness of the Riemannian metric $h$, the warping function $w$ satisfies $w(0)= w(2)=0$, $\dot{w}(0)=1$ and $\dot{w}(2)=-1$. Moreover, (\ref{eqn: curvature-condition-w}) implies $w$ is a concave function on $[0, 2]$. Thus Lemma \ref{lem: concave-function-properties} applies.

By (\ref{eqn: curvature-condition-w}) and Lemma \ref{lem: concave-function-properties}, we have that $\forall a, b \in \left[ \epsilon, 2 - \epsilon \right]$ and $\theta \in [0, 2\pi]$, 
\begin{equation}
\begin{aligned}
  |f(b, \theta) - f(a, \theta)| 
  & = \left| \int^b_a \partial_t f(t, \theta) dt \right| \\
  & \leq \int^b_a \sqrt{- \frac{\ddot{w}(t)}{w(t)}} dt \\
  & \leq \sqrt{\int^b_a - \ddot{w}(t)dt \cdot \int^b_a \frac{1}{w(t)}dt} \\
  & \leq \sqrt{\tfrac{2}{\epsilon} \max\limits_{[0, 2]} w \cdot (2- 2\epsilon)\tfrac{2}{\epsilon\max\limits_{[0, 2]}w}} = \tfrac{2 \sqrt{2(1-\epsilon)}}{\epsilon}.
\end{aligned}
\end{equation}

Again, by Lemma \ref{lem: concave-function-properties}, we have $\int^{2-\epsilon}_{\epsilon} -\ddot{w}(t) dt \leq 2$. Thus, there exists $t_\epsilon \in [\epsilon, 2 -\epsilon]$ such that $-\ddot{w}(t_\epsilon) \leq \frac{1}{1-\epsilon}$. Then by (\ref{eqn: curvature-condition-w}), we have
\begin{equation}
|\partial_{\theta} f|(t_\epsilon, \theta) \leq \sqrt{-\ddot{w}(t_\epsilon) w(t_\epsilon)} \leq \frac{1}{\sqrt{1-\epsilon}}, \quad \forall \theta \in [0, 2\pi].
\end{equation}
Consequently, for any $\theta_1, \theta_2 \in [0, 2\pi]$, $|f(t_\epsilon, \theta_1) - f(t_\epsilon, \theta_2)| \leq \frac{2\pi}{\sqrt{1-\epsilon}}$. Thus, for any $t_0 \in [\epsilon, 2 - \epsilon]$,
\begin{eqnarray*}
& & |f(t_0, \theta_1) - f(t_0, \theta_2)| \\
& \leq & |f(t_0, \theta_1) - f(t_\epsilon, \theta_1)| + |f(t_\epsilon, \theta_1) - f(t_\epsilon, \theta_2)| + |f(t_\epsilon, \theta_2) - f(t_0, \theta_2)| \\
& \leq & \tfrac{2\sqrt{2(1-\epsilon)}}{\epsilon} + \tfrac{2\pi}{\sqrt{1-\epsilon}} +\tfrac{2\sqrt{2(1-\epsilon)}}{\epsilon}  = \tfrac{4\sqrt{2(1-\epsilon)}}{\epsilon} + \tfrac{2\pi}{\sqrt{1-\epsilon}} 
\end{eqnarray*}
holds for any $\theta_1, \theta_2 \in [0, 2\pi]$. This complete the proof.
\end{proof}

\begin{lemma}\label{lem: integral-f-estimate}
Let $h$ and $f$ be a Riemannian metric and a smooth function on $\Sph^2$ as in Lemma \ref{lem: oscillation-f}, and we further assume that
\begin{equation}\label{eqn: f-average-zero-1}
\int_{\Sph^2} f d\vol_{h} = 0.
\end{equation}
 Then for any fixed $0< \epsilon <1$, there exists a constant $C(\epsilon)$ such that
\begin{equation}
 -C(\epsilon) \leq  \int^{2\pi}_0 f(t_0, \theta) d\theta \leq C(\epsilon)
\end{equation}
holds for any $t_0 \in [\epsilon, 2 - \epsilon]$. The constant $C(\epsilon)$ is independent of $h$ and $f$, and goes to $+\infty$ as $\epsilon \to 0$.
\end{lemma}
\begin{proof}
We only prove the upper bound estimate, and the lower bound follows by applying the same argument to $-f$.

If there exists $\theta_0 \in [0, 2\pi]$ such that $f(t_0, \theta_0) \leq 0$, then by applying Lemma \ref{lem: oscillation-f}, we have
\begin{equation}
f(t_0, \theta) \leq \tfrac{4\sqrt{2(1-\epsilon)}}{\epsilon} + \tfrac{2\pi}{\sqrt{1-\epsilon}} + f(t_0, \theta_0) 
\leq 
\tfrac{4\sqrt{2(1-\epsilon)}}{\epsilon} + \tfrac{2\pi}{\sqrt{1-\epsilon}}, \quad \forall \theta \in [0, 2\pi].
\end{equation}
Consequently, the desired integral estimate directly follows. 

Thus we assume that $f(t_0, \theta) >0$ for all $\theta \in [0, 2\pi]$ in the rest of the proof of the upper bound estimate.
By (\ref{eqn: curvature-condition-w}), we have $|\partial_t f(t, \theta)|\leq \sqrt{-\frac{\ddot{w}(t)}{w(t)}}$. we set 
\begin{equation}
v(t): = t\frac{\dot{w}(t)}{w(t)}, \quad \text{for} \ \ t \in[0, t_0].
\end{equation} 
Using the concavity of $w$, one can easily obtain
\begin{equation}
- \frac{t}{2-t} \leq t \frac{\dot{w}}{w} \leq 1,  \ \  \text{for} \ \  t \in (0, 2).
\end{equation}
Thus
\begin{equation}\label{eqn: v-estimate}
\frac{\epsilon -2}{\epsilon} \leq v(t) \leq 1, \quad \text{for} \ \ 0 \leq t \leq t_0 \leq 2-\epsilon.
\end{equation}
Differentiating $\frac{\dot{w}}{w}$ with respect the variable $t$, we obtain
\begin{equation}
\frac{\ddot{w}}{w}-(\frac{\dot{w}}{w})^2=\frac{t\dot{v}-v}{t^2}.
\end{equation}
For $t_1\in [0,t_0]$ and $\theta \in [0, 2\pi]$, we then have
\begin{eqnarray}
   |f(t_1, \theta) - f(t_0, \theta)| 
   & \leq &   
   \int_{t_1}^{t_0}\sqrt{-\frac{\ddot{w}}{w}} dt \\
   & = & \int_{t_1}^{t_0}\frac{1}{t} \sqrt{-t \dot{v}+v-v^2} dt \\
    &\leq &  \int_{t_1}^{t_0}\frac{1}{t} \sqrt{-t \dot{v}+1/4} dt \\
    &\leq &  \int_{t_1}^{t_0}\frac{1}{t}(-t \dot{v}+1/2) dt 
    \\
    & \leq & \frac{2}{\epsilon} + \frac{1}{2} \ln t_0 - \frac{1}{2}\ln t_1.
\end{eqnarray}
In the second last inequality, we used the fact that $\sqrt{x + \frac{1}{2}} \leq x+ \frac{1}{2}$ for any $x \geq -\frac{1}{4}$. In the last inequality, we used (\ref{eqn: v-estimate}). As a result, for $t \in [0, t_0]$ and $\theta \in [0, 2\pi]$, we have
\begin{equation}\label{estimate-delta-varphi}
f(t, \theta) \geq f(t_0, \theta) - \frac{2}{\epsilon} - \frac{1}{2} \ln t_0 + \frac{1}{2} \ln t.
\end{equation}

For $t \in [t_0,2]$ and $\theta\in [0, 2\pi]$, by introducing $\tilde{v}(t):=(2-t)\frac{\dot{w}}{w}$, the same argument gives
\begin{equation}\label{estimate-delta-varphi2}
f(t, \theta) \geq f(t_0, \theta) - \frac{2}{\epsilon} + \frac{1}{2} \ln (2-t) - \frac{1}{2} \ln (2 - t_0).
\end{equation}
Then by the mean zero condition \eqref{eqn: f-average-zero-1}, (\ref{estimate-delta-varphi}) and (\ref{estimate-delta-varphi2}), we have
\begin{equation}\label{eqn: estimate-integral-of-f}
 \begin{aligned}
    0 & =     \int_{\Sph^2}fd\vol_{h} \\
    & =  \int^{2\pi}_0 d\theta\int_0^{t_0}f(t, \theta)w(t)dt+ \int^{2\pi}_0 d\theta\int_{t_0}^2f(t, \theta)w(t)dt \\
    &  \geq  \int^{2\pi}_0 d\theta  \int_0^{t_0}(f(t_0, \theta)- \frac{2}{\epsilon} - \frac{1}{2} \ln t_0 +\frac{1}{2}\ln t)w(t)dt \\ 
    &  \quad + \int^{2\pi}_0 d\theta \int_{t_0}^2(f(t_0, \theta)- \frac{2}{\epsilon}+ \frac{1}{2} \ln (2-t))  - \frac{1}{2}\ln (2-t_0)))w(t)dt.
  \end{aligned}
\end{equation}

For the first term on the right hand side of (\ref{eqn: estimate-integral-of-f}), by using $f(t_0, \theta) >0$, $\forall \theta \in [0, 2\pi]$, $-\frac{2}{\epsilon} - \tfrac{1}{2} \ln t_0 + \tfrac{1}{2}\ln t<0, \forall t \in [0, t_0]$, and Lemma \ref{lem: concave-function-properties}, we have
\begin{equation}\label{eqn: estimate-first-term-integral-of-f}
  \begin{aligned}
    & \quad  \int^{2\pi}_0  d\theta  \int^{t_0}_0  f(t_0, \theta) w(t) dt  + 2\pi \int^{t_0}_0 \left( -\frac{2}{\epsilon} - \frac{1}{2} \ln t_0+ \frac{1}{2} \ln t \right)w(t) dt \\
    &\geq   \int^{2\pi}_0  f(t_0, \theta) d\theta  \int^{t_0}_0  \frac{t}{t_0}w(t_0) dt  \\
    & \quad   + \frac{2\pi w(t_0)}{2-t_0} \int^{t_0}_0 \left( -\frac{2}{\epsilon} - \frac{1}{2} \ln t_0 + \frac{1}{2} \ln t  \right) (2-t) dt \\
    & =  \frac{t_0}{2} w(t_0) \int^{2\pi}_0  f(1, \theta) d\theta  + \frac{2\pi w(t_0)}{2-t_0} \left(-\frac{4t_0 - t^2_0}{\epsilon} - t_0 + \frac{t^2_0}{8} \right).
  \end{aligned}
\end{equation}

Similarly, for the second term on the right hand side of (\ref{eqn: estimate-integral-of-f}), we have
\begin{equation}\label{eqn: estimate-second-term-integral-of-f}
   \begin{aligned}
    &  \quad  \int^{2\pi}_{0}  d\theta \int^2_{t_0}  f(t_0, \theta) w(t) dt  + 2\pi \int^2_{t_0} \left( -\frac{2}{\epsilon} + \frac{1}{2} \ln\frac{2-t}{2-t_0}  \right)w(t) dt \\
    &  \geq   \frac{2 - t_0}{2} w( t_0) \int^{2\pi}_0  f(t_0, \theta) d\theta
    +
    \frac{2\pi w(t_0)}{t_0} \left(  - \frac{4-t^2_0}{\epsilon} - \frac{3}{2} + \frac{t_0}{2} + \frac{t^2_0}{8}\right).
   \end{aligned}
\end{equation}

Then by plugging (\ref{eqn: estimate-first-term-integral-of-f}) and (\ref{eqn: estimate-second-term-integral-of-f}) into (\ref{eqn: estimate-integral-of-f}), we obtain
\begin{equation}\label{eqn: f-integral-estimate}
\begin{aligned}
 \int^{2\pi}_0 f(t_0, \theta) d\theta  
& \leq 2\pi \left( \frac{4}{\epsilon t_0} + \frac{5t^2_0 - 10 t_0 + 12}{4t_0 (2-t_0)} \right)  \\
 & \leq \frac{\pi (5 \epsilon^3 - 10 \epsilon^2 -4 \epsilon + 32)}{2 \epsilon^2 (2-\epsilon)}.
 \end{aligned}
\end{equation} 
Then
\begin{equation}
C(\epsilon) := \max\left\{\frac{8\pi\sqrt{2(1-\epsilon)}}{\epsilon} + \frac{4\pi^2}{\sqrt{1-\epsilon}}, 
\frac{\pi (5 \epsilon^3 - 10 \epsilon^2 -4 \epsilon + 32)}{2 \epsilon^2 (2-\epsilon)}\right\}.
\end{equation}
gives a desired upper bound. 
\end{proof}

\begin{remark}\label{rmrk: f-integral-estimate-at-1}
{\rm
Letting $\epsilon = \frac{1}{2}, t_0 =1$ in (\ref{eqn: f-integral-estimate}), in particular, we have
\begin{equation}
\int^{2\pi}_0 f(1, \theta) d\theta \leq 20 \pi, \quad \forall \theta \in [0, 2\pi].
\end{equation}
Combined with Lemma \ref{lem: oscillation-f}, this implies
\begin{equation}\label{eqn: f(1, theta)-estimate}
\max\limits_{[0, 2\pi]} f(1, \theta) \leq 28 + 2\sqrt{2}\pi
\end{equation}
}
\end{remark}

\begin{lemma}\label{lem: integral-estimate-e^2f}
Let $h$ and $f$ be a Riemannian metric and a smooth function as in Lemma \ref{lem: oscillation-f}.  Then we have
\begin{equation}
\int_{\Sph^2} e^{2f} d\vol_{h} \leq 400\pi e^{100} w(1).
\end{equation}
\end{lemma}
\begin{proof}
To estimate 
\begin{equation}\label{eqn: integral-of-e^2f}
\int_{\Sph^2} e^{2 f} d\vol_{h}=  \int^{2\pi}_0 \left( \int_0^2 e^{2f(t, \theta)}w(t)dt \right) d\theta,
\end{equation} 
we consider $2 f(t, \theta)+\ln w(t)$ and differentiate it with respect to $t$. Recall that $w(t)$ is a concave function on $[0, 2]$ with $w(0) = w(2)=0$. 

For $t_1 \in [0, 1]$ and $\theta \in [0, 2\pi]$, we have
\begin{equation}
    \begin{aligned}
 \ln \left(e^{2 f(t_1, \theta)} w(t_1)\right)-\ln \left(e^{2 f(1, \theta)} w(1)\right)&=-\int_{t_1}^{1}\frac{\partial}{\partial t}(2 f(t, \theta)+\ln w(t))dt \\
 &\leq \int_{t_1}^{1} 2 \sqrt{-\frac{\ddot w}{w}}-\frac{\dot{w}}{w}dt
    \end{aligned}\label{estimate-ai-1}
\end{equation}
Here we used (\ref{eqn: curvature-condition-w}). To estimate the RHS of (\ref{estimate-ai-1}), we set 
\begin{equation}
\frac{\dot{w}}{w}=\frac{v}{t}, \ \ t \in (0, 2).
\end{equation} 
By the concavity of $w$, we have that for $-\frac{w}{2-t}\leq \dot{w}\leq \frac{w}{t}$ for $t \in (0, 1)$, and thus $-1\leq v\leq 1$. By taking derivative for $\frac{\dot{w}}{w}$ with respect the variable $t$, we obtain
\begin{equation}
\frac{\ddot{w}}{w}-(\frac{\dot{w}}{w})^2=\frac{t\dot{v}-v}{t^2}.
\end{equation}
Thus
\begin{equation}
\text{RHS of (\ref{estimate-ai-1})}
=\int_{t_1}^{1} \frac{1}{t}\left(2 \sqrt{-t \dot{v}+v-v^2}-v\right)\label{estimate-0}
\end{equation}
Applying the inequality (\ref{claim-1-1}) to the integrant with $x:=-t\dot{v}$ and $y:=v$, we have
\begin{equation}
\text{LHS of (\ref{estimate-ai-1})}
 \leq \int_{t_1}^{1} \frac{1}{t}\left(-2t \dot{v}+\frac{5}{8}\right) 
 \leq 4 - \frac{5}{8} \ln t_1.
\end{equation}
Here we used that $-1 \leq v\leq 1$.
Consequently, we obtain
\begin{equation}\label{eqn: estimate-efw-less1}
e^{2f(t, \theta)}w(t) \leq e^4 e^{2f(1, \theta)}w(1) t^{-\frac{5}{8}}, \ \ \text{for} \ \ t \in [0, 1] \ \ \text{and} \ \  \theta\in [0, 2\pi].
\end{equation} 

For $t\in [1, 2]$, by using the symmetry of interval $[0, 2]$ about $1$ and doing a change of variable $t \mapsto 2-t$, from (\ref{eqn: estimate-efw-less1}), one can obtain
\begin{equation}\label{eqn: estimate-efw-greater1}
e^{2f(t, \theta)}w(t) \leq e^4 e^{2f(1, \theta)}w(1)(2-t)^{-\frac{5}{8}}, \ \ \text{for} \ \ t \in [1, 2] \ \ \text{and} \ \  \theta\in [0, 2\pi].
\end{equation}
By combining (\ref{eqn: estimate-efw-less1}) and (\ref{eqn: estimate-efw-greater1}), since $\frac{5}{8} <1$, we obtain that for any $\theta \in [0, 2\pi]$,
\begin{equation}\label{eqn: estimate-e^2f-1}
 \int^2_0 e^{2f(t, \theta)}w(t) dt =  \int_0^{1}e^{2f(t, \theta)}w(t)dt+\int_{1}^2e^{2f(t, \theta)}w(t)dt \leq 200 w(1) e^{2f(1, \theta)}.
\end{equation}
Combined with estimate (\ref{eqn: f(1, theta)-estimate}) in Remark \ref{rmrk: f-integral-estimate-at-1}, estimate (\ref{eqn: estimate-e^2f-1}) then gives
\begin{equation}
\int_{\Sph^2} e^{2 f} d\vol_{h}=  \int^{2\pi}_0 \left( \int_0^2 e^{2f(t, \theta)}w(t)dt \right) d\theta \leq 400\pi e^{100} w(1).
\end{equation}
This complete the proof of the lemma.
\end{proof}

\begin{lemma}\label{lem: integral-ratio-estimate}
Let $h$ and $f$ be a Riemannian metric and smooth function as in Lemma \ref{lem: oscillation-f}.  Then we have
\begin{equation}
 \frac{\int_{\Sph^2} e^{2f}d\vol_{h}}{\int_{\Sph^2} e^{f}d\vol_{h}}
 \leq 
 200e^{100}.
\end{equation}
\end{lemma}
\begin{proof}
First, by using Lemma \ref{lem: integral-estimate-e^2f}, we have
\begin{equation}\label{eqn: estimate-e^2f}
\int_{\Sph^2} e^{2 f} d\vol_{h} \leq 400 \pi e^{100} w(1), \quad \forall i \in \mathbb{N}.
\end{equation}

To estimate $\int_{\Sph^2} e^{f} d \vol_{h}$, applying the Jensen's inequality to $e^{f}$ gives  
\begin{equation}
  \frac{\int_{\Sph^2} e^{f} d \vol_{h}}{{\rm Area}_{h}(\Sph^2)}
   \geq 
   \exp\left({\frac{\int_{\Sph^2} f d\vol_{h}}{{\rm Area}_{h}(\Sph^2)}}\right).
\end{equation}
Then, since $\int_{\Sph^2} f d\vol_{h} =0$, we obtain
\begin{equation}\label{estimate-e^f}
\int_{\Sph^2} e^{f} d \vol_{h} \geq  {\rm Area}_{h}(\Sph^2) = 2\pi \int^{2}_{0} w(t) dt
\geq 2\pi w(1), \quad \forall i \in \mathbb{N},
\end{equation}
where the lower bound estimate $w(t)$ in Lemma \ref{lem: concave-function-properties} with $t_0 =1$. 

Finally, by combining the estimates in (\ref{eqn: estimate-e^2f}) and (\ref{estimate-e^f}), we obtain
\begin{equation}
 \frac{\int_{\Sph^2} e^{2f}d\vol_{h}}{\int_{\Sph^2} e^{f}d\vol_{h}} 
 \leq 
 \frac{400 \pi e^{100}w(1)}{2\pi w(1)} = 200e^{100}, \quad \forall i \in \mathbb{N}.
\end{equation}
This completes the proof.
\end{proof}

Now we are ready to prove Proposition \ref{prop: diameter-bound}.
\begin{proof}[Proof of Proposition \ref{prop: diameter-bound}]
First we show that $\{a_i\}$ has a positive lower bound. It suffices to note that the circle $\{(r,\theta)|\theta=\theta_0\}$ is a smooth geodesic in $(\Sph^2, h_i)$, and it has length $2a_i$. Then Lemma \ref{lem: torus-mean-curvature} and $\mina(\Sph^2 \times \Sph^1, g_i) \geq A_0 >0$ imply that $\{a_i\}$ is positively lower bounded.

Let $\phi_i=1/\varphi_i$. By Lemmas \ref{lem: scalar-curvature-formula}, \ref{lem: mina-consequence-1}, \ref{lem: mina-consequence-2} and \ref{lem: volume-bound-consequence}, we have
   \begin{equation}\label{key}
      \begin{cases}
      2\frac{\left|\nabla^{h_i} \phi_i\right|^2}{\phi_i^2}\leq  \Scal_{h_i}, \cr
    \int_{\Sph^2} \phi_i^2 d\vol_{h_i} \geq C(\Lambda), \cr
     \int_{\Sph^2} \phi_i d\vol_{h_i}\leq C(\Lambda), \cr
     \max\limits_{[0, a_i]} u_i =u_i(r_i)\geq  C(\Lambda).
      \end{cases}
     \end{equation}
Here and in the rest of the proof, $C(\Lambda)$ denotes a positive constant depending only on $V$ and $A_0$, whose value may vary from step to step.

Now we prove that the sequence $\{a_i\}$ is bounded from above by contradiction. Assume, to the contrary, that the positive sequence $\{a_i\}$ is unbounded. Then we can pick a subsequence, still denoted by $\{a_i\} $, such that $ \lim\limits_{i\to \infty}a_i = \infty$. As a result, by the concavity of $u_i$ on $[0, a_i]$, $u_i(0) = u_i(a_i) =0$, and the fourth line in (\ref{key}), we have 
\begin{equation}\label{eqn: area-unbounded}
{\rm Area}_{h_i}(\Sph^2) = 2\pi \int^{a_i}_0 u_i(r)dr \geq \pi C(\Lambda)a_i \to \infty, \ \ \text{as} \ \ i \to \infty.
\end{equation}
By applying Jensen's inequality to $\log \phi_i$, the third line in (\ref{key}), and (\ref{eqn: area-unbounded}), we have
 \begin{equation}
    e^{\overline{\ln \phi}_i}\coloneqq e^{\left(\frac{1}{{\rm Area}_{h_i}(\Sph^2)}\int_{\Sph^2} \ln\phi_i d\vol_{h_i}\right)} 
    \leq \frac{\int_{\Sph^2} \phi_i d\vol_{h_i}}{{\rm Area}_{h_i}(\Sph^2)}  \leq \frac{C(\Lambda)}{{\rm Area}_{h_i}(\Sph^2)} \rightarrow 0\label{log-phi-mean}
\end{equation}
as $i\rightarrow \infty$. Thus 
\begin{equation}\label{eqn: mean-tends-to-negative-infinity}
\overline{\ln \phi}_i \to -\infty, \ \ \text{as} \ \ i \to \infty.
\end{equation} 
Let 
\begin{equation}
f_i := \ln \phi_i- \overline{\ln \phi}_i
\end{equation} 
be the mean-zero part of $\ln \phi_i$. Then by using the second and third lines of (\ref{key}), and (\ref{eqn: mean-tends-to-negative-infinity}), we obtain
\begin{equation}
\frac{\int_{\Sph^2} e^{2f_i}d\vol_{h_i}}{\int_{\Sph^2} e^{f_i}d\vol_{h_i}}\geq C(\Lambda)e^{-\overline{\ln \phi}_i} \to \infty, \ \ \text{as} \ \ i \to \infty.
\end{equation}

We rescale the metric $h_i$ by $\frac{4}{a^2_i}$ to obtain the metric 
\begin{equation}
\bar{h}_i =dt^2+w_i(t)^2d\theta^2
\end{equation} 
where 
\begin{equation}
t=\tfrac{2r}{a_i} \in [0,2],  \quad  w_i(t)=\tfrac{2}{a_i}u_i\left(\tfrac{a_it}{2}\right).
\end{equation}
We use `` $\dot{}$ " to denote derivative with respect to the variable $t$. In particular, ${\rm Diam}_{\bar{h}_i}(\Sph^2)=2$.

So far, we have obtained that
\begin{equation}\label{key2}
    \begin{cases}
    2 \left|\nabla^{\bar{h}_i} f_i \right|^2 \leq -2\frac{\ddot{w}_i}{w_i}, \cr
   \int_{\Sph^2} f_id\vol_{\bar{h}_i}=0, \cr 
   \end{cases}
\end{equation}
and 
\begin{equation}\label{key3}
  \frac{\int_{\Sph^2} e^{2f_i}d\vol_{\bar{h}_i}}{\int_{\Sph^2} e^{f_i}d\vol_{\bar{h}_i}} \to \infty, \ \ \text{as} \ \ i \to \infty
  \end{equation}
hold simultaneously. The first line of (\ref{key2}) follows from the first line of (\ref{key}).
  
However, by applying Lemma \ref{lem: integral-ratio-estimate}, (\ref{key2}) implies
\begin{equation}
 \frac{\int_{\Sph^2} e^{2f_i}d\vol_{\bar{h}_i}}{\int_{\Sph^2} e^{f_i}d\vol_{\bar{h}_i}} 
 \leq 
 \frac{400 \pi e^{100}w_i(1)}{2\pi w_i(1)} = 200e^{100}, \quad \forall i \in \mathbb{N}.
\end{equation}
This contradicts (\ref{key3}). Thus the sequence $\{ a_i \}$ must be bounded from above. This completes the proof.
\end{proof}

\begin{remark}\label{rmrk: rescale-metric}
{\rm
Proposition \ref{prop: diameter-bound} shows that
\begin{equation}
C(\Lambda)^{-1}\leq a_i\leq C(\Lambda), \quad \forall i\in\mathbb N.
\end{equation}
This uniform two-sided diameter bound is crucial for the convergence of the warping functions.  It allows us, after a harmless rescaling of the metrics, to assume that 
\begin{equation}
a_i=2, \quad \forall i \in \mathbb{N}.
\end{equation}  
Under this normalization, the constants in the estimates remain uniformly
controlled, and all functions $u_i$ are defined on the fixed interval $[0,2]$.
Thus the compactness arguments below can be carried out on a single common
domain.  Moreover, the estimates obtained above for the pair $(f,w)$ apply,
after this normalization, to the corresponding pair $(\ln\varphi_i,u_i)$.
Throughout the rest of the paper, we shall assume that $a_i=2$ for all
$i\in\mathbb N$.
}
\end{remark}


\section{Convergence of the warping functions}\label{sect: warping-function-convergence}

We now use the uniform diameter bound from Proposition \ref{prop: diameter-bound} to study the convergence of the warping functions.  By Remark \ref{rmrk: rescale-metric}, we may normalize the base metrics so that the rotational parameter always ranges over the fixed interval $[0,2]$.  

The convergence argument has three parts.  First, the scalar curvature lower bound implies concavity and uniform Lipschitz control for the base warping functions $u_i$.  Second, away from the poles of $\Sph^2$, the same curvature inequality, together with the volume and the $\MinA$ codnition, gives uniform two-sided bounds and H\"older estimates for the fiber warping factors $\varphi_i$.  Finally, these local estimates are extended to global Sobolev bounds on $\Sph^2$, allowing only mild degeneration at the poles.

\subsection{Convergence of $u_i$}\label{subsect: u-convergence}

We begin with the base warping functions $u_i$.  The key point is that the nonnegative scalar curvature condition forces the scalar curvature of the base metrics $h_i$ to be nonnegative, and hence the functions $u_i$ are concave.  Together with the smoothness conditions at the poles and the $\MinA$ condition, this gives compactness by the Arzel\'a--Ascoli theorem.

\begin{proposition}\label{prop: u_i-convergence}
Let $\{ g_i \}$ be a sequence of the warped product Riemannian metrics on $ \Sph^2 \times \Sph^1$ as in (\ref{metric-form}) satisfying condition $\Lambda$:
\begin{equation}\label{eqn: condition-lambda-estimate-varphi}
 \Scal_{g_i} \geq 0, \ \ {\rm Vol}_{g_i}(\Sph^2 \times \Sph^1) \leq V_0, \ \  \text{and} \ \ \mina(\Sph^2 \times \Sph^1, g_i) \geq A_0 >0.
\end{equation}
Then warping functions $u_i$ of the metrics $h_i$ on $\Sph^2$ are concave functions and have uniform estimate:
\begin{equation}
0 \leq u_i(r) \leq 1, \ \ |u^\prime_i(r)| \leq 1, \ \ \text{and} \ \ \max_{[0, 2]} u_i \geq \frac{A_0}{4\pi^2}, \ \ \forall i \in \mathbb{N}.
\end{equation}
Thus, $\{u_i\}$
has a subsequence, which uniformly converges to $u_\infty$ on $[0, 2]$, where $u_\infty$ is concave, 1-Lipschitz, and satisfies $u_\infty(0) = u_\infty(2) =0$ and $\max\limits_{[0, 2]} u_\infty \geq \frac{A_0}{4\pi^2}$. 
\end{proposition}
\begin{proof}
As explained in Remark \ref{rmrk: rescale-metric}, we assume that $a_i=2$ for all $i$. Note that $\Scal_{h_i} = -\frac{u^{\prime\prime}_i}{u}$. Then by Lemmas \ref{lem: scalar-curvature-formula} and \ref{lem: mina-consequence-1}, we have that $u_i$, $i \in \mathbb{N}$, are concave functions on $[0, 2]$ satisfying
\begin{equation}
    \max_{[0, 2]}u_i \geq \frac{A_0}{4\pi^2}, \ \ \forall i \in \mathbb{N}.
\end{equation}
Moreover, the smoothness of the Riemannian metrics $h_i$ on $\Sph^2$ implies
\begin{equation}
 u_i(0) = u_i(2) =0, \ \ u^\prime_i(0) =1, \ \ u^\prime_i(2) = -1, \ \ \forall i \in \mathbb{N}.
 \end{equation}
Then clearly $0 \leq u_i(r) \leq \min\{r, 2-r\} \leq 1$ and $|u^\prime_i(r)|\leq 1$ for all $i$. Thus the sub-convergence follows from Arzel\'a-Ascoli theorem, and the properties of the limit function $u_0$ directly follows from the uniform convergence.
\end{proof}

The preceding compactness statement also gives a uniform area bound for the base spheres.  This estimate will be used repeatedly below to control averages of $\varphi_i$ and $\varphi_i^{-1}$.

\begin{corollary}\label{cor: area-bound}
Let $\{ g_i \}$ be a sequence of the warped product Riemannian metrics on $ \Sph^2 \times \Sph^1$ as in (\ref{metric-form}) satisfying condition $\Lambda$ as in \eqref{eqn: condition-lambda-estimate-varphi}.
Then $\frac{A_0}{8\pi} \leq {\rm Area}_{h_i}(\Sph^2) \leq 4\pi$ holds for all $i \in \mathbb{N}$.
\end{corollary}
\begin{proof}
By Proposition \ref{prop: u_i-convergence} and Lemma \ref{lem: concave-function-properties} with $\epsilon =\frac{1}{2}$, we have $0 \leq u_i(r) \leq 1$ for all $r \in [0, 2]$ and all $ i \in \mathbb{N}$, and $u_i(r) \geq \frac{A_0}{16 \pi^2}$ for all $r \in [\tfrac{1}{2}, \tfrac{3}{2}]$ and $i \in \mathbb{N}$. Consequently,
\begin{equation}
\frac{A_0}{8\pi} \leq 2\pi \int^{\tfrac{3}{2}}_{\tfrac{1}{2}} u_i(r)dt \leq 2\pi \int^{2}_{0} u_i(r)dr \leq 4 \pi.
\end{equation}
Since ${\rm Area}_{h_i}(\Sph^2) = 2\pi \int^2_0 u_i(r)dr$, this proves the desired area estimate.
\end{proof}

\subsection{Convergence of $\varphi_i$ on $\Omega_\epsilon$}

We now turn to the fiber warping factors $\varphi_i$.  On compact subdomains away from the poles, the functions $u_i$ have a uniform positive lower bound.  This allows the scalar curvature inequality to be converted into estimates for $\ln \varphi_i$ and hence into two-sided bounds and H\"older control for $\varphi_i$.

\begin{proposition}\label{prop: varphi-C01/2-estimate}
    Let $\{ g_i \}$ be a sequence of the warped product Riemannian metrics on $ \Sph^2 \times \Sph^1$ as in (\ref{metric-form}) satisfying the condition $\Lambda$ as in \eqref{eqn: condition-lambda-estimate-varphi}.
Then for any $0 \leq \epsilon <1$, there exists a positive constant $C(\Lambda, \epsilon)$, which may tend to $+\infty$ as $\epsilon \to 0$, such that
\begin{equation}\label{eqn: varphi-lower-upper-bound}
    \frac{1}{C(\Lambda, \epsilon)} \leq \varphi_i \leq C(\Lambda, \epsilon), \ \ \text{on} \ \ \Omega_\epsilon, \ \ \forall i \in \mathbb{N},
\end{equation}
and
\begin{equation}
    \|\varphi_i\|_{C^{0, \frac{1}{2}}(\Omega_\epsilon)} \leq C(\Lambda, \epsilon), \quad \forall i \in \mathbb{N},
\end{equation}
where 
\begin{equation}
    \Omega_\epsilon = \{(r, \theta) \mid \epsilon \leq r \leq 2-\epsilon, \ \ 0\leq \theta \leq 2\pi\} \subset \Sph^2.
\end{equation}

Thus, there exists a limit function $\varphi_\infty$ on $(0, 2) \times [0, 2\pi]$ such that
\begin{equation}
\varphi_i \to \varphi_\infty \quad \text{uniformly in} \ \ C^{1, \alpha}(\Omega_\epsilon), \ \ \text{as} \ \ i \to \infty 
\end{equation}
for any $\alpha < \frac{1}{2}$ and any $0 < \epsilon <1$.
\end{proposition}
\begin{proof}
We prove the uniform boundedness of $\|\varphi_i\|_{C^{0, \frac{1}{2}}(\Omega_\epsilon)} = \|\varphi_i\|_{C^0(\Omega_\epsilon)} + [\varphi_i]_{C^{0, \frac{1}{2}}(\Omega_\epsilon)}$ in two steps. We first derive (\ref{eqn: varphi-lower-upper-bound}), thereby obtaining a uniform bound for $\|\varphi_i\|_{C^0(\Omega_\epsilon)}$, and then estimate $[\varphi_i]_{C^{0, \frac{1}{2}}(\Omega_\epsilon)}$.

{\bf Step 1: proof of (\ref{eqn: varphi-lower-upper-bound}).} 

 By Lemmas \ref{lem: mina-consequence-1} and \ref{lem: volume-bound-consequence}, the condition $\Lambda$ in (\ref{eqn: condition-lambda-estimate-varphi}) implies
\begin{equation}\label{eqn: key}
    \begin{cases}
         2 \frac{|\nabla^{h_i} \varphi_i |^2}{\varphi^2_i} \leq \Scal_{h_i} = - 2 \frac{u^{\prime\prime}}{u_i}, \cr
    \int_{\Sph^2} \frac{1}{\varphi_i^2} d\vol_{h_i} \geq A_0, \cr
    \int_{\Sph^2} \frac{1}{\varphi_i} d\vol_{h_i} \leq \frac{V_0}{2\pi}.    \end{cases}
\end{equation}
Let $\phi_i := \frac{1}{\varphi_i}$. Then we have
\begin{equation}\label{key4}
\begin{cases}
    2 \frac{|\nabla^{h_i} \phi_i |^2}{\phi^2_i} \leq \Scal_{h_i} = - 2 \frac{u^{\prime\prime}}{u_i}, \cr
    \int_{\Sph^2} \phi_i^2 d\vol_{h_i} \geq A_0, \cr
    \int_{\Sph^2} \phi_i d\vol_{h_i} \leq \frac{V_0}{2\pi}.
\end{cases}
\end{equation}
By Jensen's inequality, the third line of (\ref{key4}) and Corollary \ref{cor: area-bound} imply
\begin{equation}
    e^{\overline{\ln \phi}_i}
    := e^{\left(\frac{1}{{\rm Area}_{h_i}(\Sph^2)}\int_{\Sph^2} \ln\phi_i d\vol_{h_i}\right)} 
    \leq \frac{\int_{\Sph^2} \phi_i d\vol_{h_i}}{{\rm Area}_{h_i}(\Sph^2)}  
    \leq \frac{\tfrac{V_0}{2\pi}}{\tfrac{A_0}{8\pi}} = \frac{4V_0}{A_0}, \ \ \forall i \in \mathbb{N}.
\end{equation}
Consequently, the averages of $\ln \phi_i$ have a uniform upper bound, i.e.
\begin{equation}\label{eqn: average-psi-upper-bound}
  \overline{\ln \phi_i} \leq \ln \tfrac{4V_0}{A_0}, \ \ \forall i \in \mathbb{N}.
\end{equation}

Let 
\begin{equation}
f_i := \ln \phi_i- \overline{\ln \phi}_i
\end{equation} 
be the mean-zero part of $\ln \phi_i$. Then using the second and third lines of (\ref{key4}), we obtain
\begin{equation}\label{eqn: integral-ratio-lower-bound}
\frac{\int_{\Sph^2} e^{2f_i}d\vol_{h_i}}{\int_{\Sph^2} e^{f_i}d\vol_{h_i}}\geq C(\Lambda)e^{-\overline{\ln \phi}_i} .
\end{equation}
Moreover, the first line of (\ref{key4}) and the definition of $f_i$ imply that
\begin{equation}\label{key5}
    \begin{cases}
    2 \left|\nabla^{h_i} f_i \right|^2 \leq -2\frac{u^{\prime\prime}_i}{u_i}, \cr
   \int_{\Sph^2} f_id\vol_{h_i}=0 \cr 
   \end{cases}
\end{equation}
hold for all $i \in \mathbb{N}$.
Then by Lemma \ref{lem: integral-ratio-estimate}, (\ref{key5}) implies
\begin{equation}\label{eqn: integral-ratio-upper-bound}
\frac{\int_{\Sph^2} e^{2f_i}d\vol_{h_i}}{\int_{\Sph^2} e^{f_i}d\vol_{h_i}} \leq 200 e^{100}, \ \ \forall i \in \mathbb{N}.
\end{equation}
Combining (\ref{eqn: integral-ratio-lower-bound}) and (\ref{eqn: integral-ratio-upper-bound}), we obtain that
\begin{equation}\label{eqn: average-phi-lower-bound}
\overline{\ln \phi_i} \geq C(\Lambda), \ \ \forall i \in \mathbb{N}.
\end{equation}

We now fix any $\epsilon \in (0, 1)$. By Lemma \ref{lem: oscillation-f} with the triangle inequality and Lemma \ref{lem: integral-f-estimate}, (\ref{key5}) implies that there exists a constant $C(\epsilon)$ such that
\begin{equation}\label{eqn: uniform-f-difference-estimate}
    |f_i(r_1, \theta_1) - f_i(r_2, \theta_2)| \leq C(\epsilon), \ \ \forall (r_1, \theta_1), (r_2, \theta_2) \in \Omega_\epsilon, i \in \mathbb{N}.
\end{equation}
and 
\begin{equation}\label{eqn: uniform-f-integral-estimate}
    -C(\epsilon) \leq \int^{2-\epsilon}_{\epsilon} dr \int^{2\pi}_{0} f_i(r, \theta)d\theta \leq C(\epsilon), \ \ \forall i \in \mathbb{N}.
\end{equation}
By using the integral mean value theorem, (\ref{eqn: uniform-f-integral-estimate}) imply that for each $i \in \mathbb{N}$ there exists $(r_i, \theta_i) \in \Omega_\epsilon$ such that
\begin{equation}
    -\frac{C(\epsilon)}{4\pi(1 - \epsilon)} \leq f_i(r_i, \theta_i) \leq \frac{C(\epsilon)}{4\pi(1 - \epsilon)}
\end{equation}
Combining with (\ref{eqn: uniform-f-difference-estimate}), this gives 
\begin{equation}
    -\left( 1+ \tfrac{4\pi}{1-\epsilon} \right) C(\epsilon) 
    \leq f_i(r, \theta) 
    \leq \left( 1+ \tfrac{4\pi}{1-\epsilon} \right) C(\epsilon), \ \ \forall (r, \theta) \in \Omega_\epsilon.
\end{equation}
Together with the uniform upper and lower bounds for the averages $\overline{\phi}_i$ obtained in (\ref{eqn: average-psi-upper-bound}) and (\ref{eqn: average-phi-lower-bound}), this yields that there exists a positive constant $C(\Lambda, \epsilon)$, which may tend to $+\infty$ as $\epsilon \to 0$, such that $\frac{1}{C(\Lambda, \epsilon)} \leq \phi_i \leq C(\Lambda, \epsilon)$ on $\Omega_\epsilon$ for all $i \in \mathbb{N}$. Since $\varphi_i = \frac{1}{\phi_i}$, this gives the estimate in (\ref{eqn: varphi-lower-upper-bound}). In particular, $\|\varphi_i\|_{C^{0}(\Omega_\epsilon)} \leq C(\Lambda, \epsilon)$ for all $ i \in \mathbb{N}$.

{\bf Step 2: estimate $[\varphi_i]_{C^{0, \frac{1}{2}}(\Omega_\epsilon)}$.}

By Lemma \ref{lem: concave-function-properties} and Proposition \ref{prop: u_i-convergence}, there exists a positive constant $C(\Lambda, \epsilon)$ such that $u_i \geq C(\Lambda,\epsilon)$ on $\Omega_\epsilon$ for all $i \in \mathbb{N}$. Therefore, by applying the first line of (\ref{eqn: key}),
we have that for any $\epsilon \leq r_1 \leq r_2 \leq 2 - \epsilon$ and $0 \leq \theta \leq 2\pi$,
\begin{equation}
\begin{aligned}
 \quad\left|(\ln \varphi_i)\left(r_2, \theta\right)-\ln \varphi_i\left(r_1, \theta\right)\right| 
 & \leq \int_{r_1}^{r_2} \sqrt{-\frac{u^{\prime\prime}_i(r)}{u_i(r)}} dr\\
 &\leq C(\Lambda,\epsilon) \int_{r_1}^{r_2} \sqrt{-u^{\prime\prime}_i(r)} dr \\
&\leq C(\Lambda,\epsilon) \sqrt{\int_{r_1}^{r_2}-u^{\prime\prime}_i(r)dr} \cdot \sqrt{r_2-r_1} \\ 
&\leq C (\Lambda,\epsilon)\sqrt{r_2-r_1}\label{Holder1/2-1}.
\end{aligned}
\end{equation}
In the last inequality, we used that $|u^\prime_i(r)|\leq 1$ for all $r \in [0, 2\pi]$ and $i \in \mathbb{N}$, see Proposition \ref{prop: u_i-convergence}.

Next, we estimate $\left|\log \varphi\left(\theta_2, r\right)-\log \varphi\left(\theta_1, r\right)\right|$ for $ r \in [\epsilon, 2 - \epsilon]$. Let $\delta>0$ be a small number to be determined later. By (\ref{Holder1/2-1})
we have $\int_r^{r+\delta} \sqrt{-\frac{u^{\prime\prime}_i}{u_i}} \leq C(\Lambda,\epsilon) \sqrt{\delta}$. By the integral mean value theorem, there exists $r_{\delta, i} \in (r, r+\delta)$ such that
\begin{equation}
 \sqrt{-\frac{\ddot{u}(r_{\delta, i})}{u(r_{\delta, i})}} \leq C(\Lambda,\epsilon)\frac{1}{\sqrt{\delta}}
\end{equation}
By the first line of (\ref{eqn: key}), we further have
\begin{equation}
    |\nabla^{h_i} \ln \varphi_i|(r_{\delta, i}, \theta)\leq C(\Lambda, \epsilon) \frac{1}{\sqrt{\delta}}, \ \ \forall \theta \in [0, 2\pi].
\end{equation}
Then by using the triangle inequality, we have
\begin{equation}
\begin{aligned}
\quad &\left|\ln \varphi_i\left(r, \theta_2\right)-\ln \varphi_i\left(r, \theta_1\right)\right| \\
& \leq\left|\ln \varphi_i\left(r_{\delta, i}, \theta_2\right)-\ln \varphi_i\left(r, \theta_2\right)\right|
+\left|\ln \varphi_i\left(r_{\delta, i}, \theta_2\right)-\ln \varphi_i\left(r_{\delta, i}, \theta_1\right)\right|\\
&\quad+\left|\ln \varphi_i\left(r_{\delta, i}, \theta_1\right)-\ln \varphi_i\left(r, \theta_1\right)\right| \\
& \leq C(\Lambda,\epsilon) \sqrt{r_{\delta, i}-r}
+\left|\ln \varphi_i\left(r_{\delta, i}, \theta_2\right) -\ln \varphi_i\left(r_{\delta, i}, \theta_1\right)\right| \\
& \leq C(\Lambda,\epsilon) \sqrt{\delta}+\frac{C(\Lambda,\epsilon)}{\sqrt{\delta}}\left|\theta_2-\theta_1\right| 
\end{aligned}
\end{equation}
By taking $\delta =\left|\theta_2-\theta_1\right|$, we obtain
\begin{equation}
\left|\ln \varphi_i\left(r, \theta_2\right)-\ln \varphi_i\left(r, \theta_1\right)\right| \leq C(\Lambda,\epsilon) \sqrt{\theta_2-\theta_1}, \quad \forall r \in [\epsilon, 2 - \epsilon].
\end{equation}
Combined with the uniform boundedness of the $\{\varphi_i\}$ on $\Omega_\epsilon$ obtained in Step 1, this shows that  $\{\ln \varphi_i$\} is uniformly bounded in $C^{0, \tfrac{1}{2}}(\Omega_\epsilon)$, an hence that $\{\varphi_i\}$ is also unformly bounded in $C^{0, \tfrac{1}{2}}(\Omega_\epsilon)$. This completes the proof.
\end{proof}

The same local control also gives a Sobolev estimate for $\varphi_i$ on each $\Omega_\epsilon$.  In contrast with the preceding H\"older estimate, this estimate uses the integrated curvature inequality together with Gauss--Bonnet formula.

\begin{proposition}
    Let $\{ g_i \}$ be a sequence of the warped product Riemannian metrics on $ \Sph^2 \times \Sph^1$ as in (\ref{metric-form}) satisfying the condition $\Lambda$ as in \eqref{eqn: condition-lambda-estimate-varphi}.
Then for any $0< \epsilon <1$ there exists a positive constant $C(\Lambda, \epsilon)$ such that 
\begin{equation}
\|\varphi_i\|_{W^{1, 2}(\Omega_\epsilon, h_{\Sph^2})} \leq C(\Lambda, \epsilon), \quad \forall i \in \mathbb{N},
\end{equation}
where $h_{\Sph^2}$ is the standard metric on $\Sph^2$.
\end{proposition}

\begin{proof} 
By Proposition \ref{prop: varphi-C01/2-estimate}, $\varphi_i$, $i \in \mathbb{N}$, are uniformly bounded on $\Omega_\epsilon$, and so are $\|\varphi_i\|_{L^2(\Omega_\epsilon, h_{\Sph^2})}$.

Using Lemma \ref{lem: scalar-curvature-formula}, $\Scal_{g_i} \geq 0$ implies $2|\nabla^{h_i} \ln \varphi_i|^2 \leq \Scal_{h_i}$. Integrating this and using the Gauss-Bonnet formula, we have
\begin{equation}
    \begin{aligned}
        2\pi & = \frac{1}{2} \int_{\Sph^2} \Scal_{h_i} d\vol_{h_i} \\
        & \geq\int_{\Sph^2}|\nabla^{h_i} \ln \varphi_i|^2 d\vol_{h_i} \\
        &\geq C(\Lambda,\epsilon)\int_{\Omega_{\epsilon}}|\nabla^{h_{\Sph^2}} \ln \varphi_i|^2d \vol_{h_{\Sph^2}}\\
        &\geq C(\Lambda,\epsilon)\int_{\Omega_{\epsilon}}|\nabla^{h_{\Sph^2}}  \varphi_i|^2d \vol_{h_{\Sph^2}}.
    \end{aligned}
\end{equation}
In the second last inequality, we used that there exists a positive constant $C(\Lambda, \epsilon)$ such that $u_i \geq C(\Lambda,\epsilon)$ on $\Omega_\epsilon$ for all $i \in \mathbb{N}$, which follows from Lemma \ref{lem: concave-function-properties} and Proposition \ref{prop: u_i-convergence}. The last inequality follows from estimate (\ref{eqn: varphi-lower-upper-bound}) in Proposition \ref{prop: varphi-C01/2-estimate}.
\end{proof}

\subsection{Convergence of $\varphi_i$ on $\Sph^2$}

The estimates above are local away from the poles.  To obtain global compactness, we next derive quantitative bounds for the possible behavior of $\varphi_i$ near $r=0$ and $r=2$.  These bounds show that $\varphi_i$ can grow or decay near the poles at most by an arbitrarily small power.

\begin{lemma}\label{lem: varphi-bounds}
Let $\{ g_i \}$ be a sequence of the warped product Riemannian metrics on $ \Sph^2 \times \Sph^1$ as in \eqref{metric-form} satisfying the condition $\Lambda$ as in \eqref{eqn: condition-lambda-estimate-varphi}. Then for any $k \in \mathbb{N}$, there exists a constant $C(\Lambda, k)$ such that
\begin{equation}\label{eqn: estimate-varphi-final}
\left(C(\Lambda, k)\right)^{-1} r^{\frac{1}{k}} 
    \leq
    \varphi_i(r, \theta)
    \leq 
    C(\Lambda, k)r^{-\frac{1}{k}}, \quad \forall (r, \theta) \in (0, 1] \times [0, 2\pi],
\end{equation}
and
\begin{equation}\label{eqn: estimate-varphi-final-2}
    \left(C(\Lambda, k)\right)^{-1} (2-r)^{\frac{1}{k}} 
    \leq
    \varphi_i(r, \theta)
    \leq 
    C(\Lambda, k)(2-r)^{-\frac{1}{k}}, \ \ \forall (r, \theta) \in [1, 2) \times [0, 2\pi], 
\end{equation}
hold for all $i \in \mathbb{N}$
\end{lemma}
\begin{remark}\label{rmrk: estimate-varphi-infty}
{\rm
Applying the locally uniform convergence result in Proposition \ref{prop: varphi-C01/2-estimate}, we see that the estimates in \eqref{eqn: estimate-varphi-final} and \eqref{eqn: estimate-varphi-final-2} also hold for $\varphi_\infty$. 
}
\end{remark}
\begin{proof}
    By Proposition \ref{prop: u_i-convergence} and Lemma \ref{lem: concave-function-properties}, the condition $\Lambda$ implies that there exists a constant $C(\Lambda)$ such that \begin{equation}\label{eqn: u-lower-linear-bound}
    u_i(r) \geq 
    \begin{cases}
        C(\Lambda) r, & 0 \leq r \leq 1, \cr
        C(\Lambda) (2- r), & 1 \leq r \leq 2.
    \end{cases}
    \ \ \forall i \in \mathbb{N}.
\end{equation}
Combined with Lemma \ref{lem: scalar-curvature-formula}, this implies that for any $0 \leq r_1 \leq 1$ and $\theta \in [0, 2\pi]$, we have
\begin{equation}\label{eqn: estimate-delta-varphi-2}
\begin{aligned}
        |\ln\varphi_i (r_1, \theta)-\ln \varphi_i (1, \theta)|
        &\leq \int_{r_1}^1\sqrt{-\frac{u^{\prime\prime}(r)}{u(r)}}dr\\
        &\leq \bigg(\int_{r_1}^1-u^{\prime\prime}_i(r)dr \cdot \int_{r_1}^1\frac{1}{u(r)}dr\bigg)^{1/2}\\
        &\leq \bigg(\frac{2}{C(\Lambda)}\int_{r_1}^1\frac{1}{r}\bigg)^{\frac{1}{2}} \\
        & \leq C(\Lambda)\sqrt{-\ln r_1}.
    \end{aligned}
\end{equation}
In the third inequality, we used that $|u^\prime_i(r)|\leq 1$.
By Proposition \ref{prop: varphi-C01/2-estimate} with $\epsilon = \frac{1}{2}$, we have $0<\max\limits_{[0, 2\pi]}\varphi_i(1, \theta) \leq C(\Lambda)$ for all $i \in \mathbb{N}$. Combined with this, (\ref{eqn: estimate-delta-varphi-2}) gives that
\begin{equation}
     \left( C(\Lambda) \right)^{-1}\exp\left( -C(\Lambda)\sqrt{-\ln r}\right)
     \leq
    \varphi_i(r, \theta)
    \leq 
    C(\Lambda)\exp\left(C(\Lambda)\sqrt{-\ln r}\right), 
\end{equation}
holds $\forall (r, \theta) \in [0, 1] \times [0, 2\pi].$
Then because the mean value inequality gives that for any $k\in \mathbb{N}$, $\sqrt{x} \leq \frac{1}{C(\Lambda) k} x + \frac{C(\Lambda) k}{4}$, we further have that
\begin{equation}
    \varphi_i(r, \theta) 
    \leq C(\Lambda)\exp\left(C(\Lambda)\left(-\frac{1}{C(\Lambda) k}\ln r+\frac{C(\Lambda) k}{4}\right)\right)\leq C(\Lambda, k) r^{-\frac{1}{k}}
\end{equation}
and
\begin{equation}
    \varphi_i(r, \theta) 
    \geq \left( C(\Lambda) \right)^{-1} \exp\left(-C(\Lambda)\left(-\frac{1}{C(\Lambda) k}\ln r+\frac{C(\Lambda) k}{4}\right)\right)\geq \left( C(\Lambda, k) \right)^{-1} r^{\frac{1}{k}}
\end{equation}
holds for all $(r, \theta) \in [0, 1] \times [0, 2\pi]$ and $i \in \mathbb{N}$. This give \eqref{eqn: estimate-varphi-final}.

Similarly, one can establish \eqref{eqn: estimate-varphi-final-2}, and complete the proof.
\end{proof}

The preceding pointwise bounds near the poles are sufficient to upgrade the local estimates to global Sobolev estimates below the critical exponent $2$.  This gives the desired global compactness of the sequence $\{\varphi_i\}$.

\begin{proposition}\label{prop: varphi-sobolev-estimate}
    Let $\{ g_i \}$ be a sequence of the warped product Riemannian metrics on $ \Sph^2 \times \Sph^1$ as in \eqref{metric-form} satisfying the condition $\Lambda$ as in \eqref{eqn: condition-lambda-estimate-varphi}. Then for any $1 \leq p <2$, there exists a constant $C(\Lambda, p)$ such that
\begin{equation}
    \|\varphi_i\|_{W^{1, p}(\Sph^2, g_{\Sph^2})} \leq C(\Lambda, p),\quad \forall i \in \mathbb{N}.
\end{equation}
Hence there exist a subsequence $\{\varphi_{i_k}\}$ of $\{\varphi_i\}$, and $\varphi_\infty \in W^{1, p}(\Sph^2, g_{\Sph^2})$ for any $1 \leq p <2$ such that 
\begin{equation}
    \varphi_{i_k} \to \varphi_\infty  \ \ \text{in} \ \ L^{q}(\Sph^2, g_{\Sph^2}), \ \ \text{for any} \ \ 1 \leq q < \infty.
\end{equation}
\end{proposition}

\begin{proof}
For any fixed $p \geq 1$, combining (\ref{eqn: estimate-varphi-final}) and $(\ref{eqn: estimate-varphi-final-2})$ with $k > \frac{p}{2}$ implies
\begin{equation}\label{eqn: varphi-lp-inegral-estimate}
    \int_{\Sph^2} \varphi^p_i d\vol_{h_{\Sph^2}}= \int^{2\pi}_{0} d\theta \int^2_0 \varphi^p_i (r, \theta) \sin r dr \leq C(\Lambda, p).
\end{equation}

Then for any fixed $1 \leq p <2$, we have
\begin{equation}\label{eqn: varphi-gradient-integral-estimate}
    \begin{aligned}
       & \quad \int_{S^2}|\nabla^{h_{\Sph^2}} \varphi_i|^pd\vol_{h_{\Sph^2}} \\
        & \leq C(\Lambda)\int_{S^2}|\nabla^{h_i} \varphi_i|^p d\vol_{h_i} \\
    &=C(\Lambda)\int_{S^2}\left(\frac{|\nabla^{h_i} \varphi_i|}{\varphi_i}\right)^p\varphi^p_i d\vol_{h_i}\\
        &\leq C(\Lambda) \left(\int_{S^2} \frac{|\nabla^{h_i} \varphi_i|^2}{\varphi^2_i}d\vol_{h_i} \right)^{p/2}\left(\int_{S^2}\varphi^{\frac{2p}{2-p}}_i d\vol_{h_i} \right)^{(2-p)/2}\\
        &\leq C(\Lambda)(4\pi)^{p/2}\left(\int_{S^2}\varphi^{\frac{2p}{2-p}}_id\vol_{h_i} \right)^{(2-p)/2} \\
        & \leq C(\Lambda)(4\pi)^{p/2}\left(\int^{2\pi}_{0} d\theta \int^{2}_{0}\varphi^{\frac{2p}{2-p}}_i(r, \theta) u_i(t) dr\right)^{(2-p)/2} \\
        & \leq C(\Lambda, p).
        \end{aligned}
\end{equation}
The first inequality follows from (\ref{eqn: u-lower-linear-bound}), $u_i(r) \leq \min\{r, 2-r\}$ (see Lemma \ref{lem: concave-function-properties}) and the fact that $\sin r \sim r$ for small $r$. The third inequality follows from Lemma \ref{lem: scalar-curvature-formula}, $R_{g_i} \geq 0$, and the Gauss-Bonnet formula. The last inequality follows from the uniform boundedness of $u_i$ (see Lemma \ref{lem: concave-function-properties}) and the estimates (\ref{eqn: estimate-varphi-final}) and (\ref{eqn: estimate-varphi-final-2}) with $k \geq \frac{2p}{2-p}$. 

Finally, combining (\ref{eqn: estimate-varphi-final}) and (\ref{eqn: varphi-gradient-integral-estimate}), we obtain the uniform $L^p$ bound of $\varphi_i$ on $
\Sph^2$, and complete the proof.
\end{proof}


\section{Gromov--Hausdorff and Intrinsic Flat Convergence}\label{sect: gh-if-convergence}

We now pass from the compactness of the warping functions to the convergence of the associated metric spaces and integral current spaces.  The main point is that, although the limiting metric may degenerate at the two polar circles, the degeneration is sufficiently mild to allow uniform diameter bounds, convergence of volumes, and convergence of the induced distance functions on compact subsets of the regular region.  These estimates then imply Gromov--Hausdorff convergence to the completion of the regular limit space and, after identifying the natural limiting current, Sormani--Wenger intrinsic flat convergence.

Let
\begin{equation}\label{eqn: mathring-M-defn}
  M:=\Sph^2\times \Sph^1,
  \qquad
  \mathring M:=(0,2)\times \Sph^1_\theta\times \Sph^1_\xi
  \subset \Sph^2\times \Sph^1.
\end{equation}
On \(\mathring M\), we define the limiting Riemannian metric by
\begin{equation}\label{eqn: limit-metric-g-infty}
g_\infty
:=
\varphi_\infty^{-2}
\left(dr^2+u_\infty^2\,d\theta^2\right)
+
\varphi_\infty^2\,d\xi^2 .
\end{equation}


\subsection{Uniform boundedness of diameters}

We first prove a uniform diameter bound.  The argument is elementary: any two points are connected by moving radially to the middle slice, then moving in the two angular directions, and finally moving radially back to the target point.  The the quantitative bounds of $\varphi_i$ in Lemma \ref{lem: varphi-bounds} control the radial lengths near the two poles, while the middle slice is uniformly controlled.

\begin{proposition}\label{prop: diameter-uniform-bound}
  Let $\{ g_i \}$ be a sequence of the warped product Riemannian metrics on $ \Sph^2 \times \Sph^1$ as in \eqref{metric-form}, satisfying the condition $\Lambda$ in \eqref{eqn: condition-lambda-estimate-varphi}. Then the diameters of $(\Sph^2 \times \Sph^1, g_i)$ are uniformly bounded. More precisely, there exits a constant $C(\Lambda)$ such that
  \begin{equation}
      \Diam_{g_i}(\Sph^2 \times \Sph^1) \leq C(\Lambda), \ \ \forall i \in \mathbb{N}.
  \end{equation}
\end{proposition}
\begin{proof}
 Let
$
p=(r,\theta,\xi)
$
and
$
q=(\bar r,\bar\theta,\bar\xi)
$
be two arbitrary points in $\Sph^2 \times \Sph^1$. We connect $p$ to $q$ by the piecewise smooth path $\gamma$ as
\begin{equation}
(r,\theta,\xi)
\xrightarrow{\gamma_1}
(1,\theta,\xi)
\xrightarrow{\gamma_2}
(1,\bar\theta,\xi)
\xrightarrow{\gamma_3}
(1,\bar\theta,\bar\xi)
\xrightarrow{\gamma_4}
(\bar r,\bar\theta,\bar\xi).
\end{equation}
We estimate the lengths of these four pieces.

For $\gamma_1$ and $ \gamma_4$, along a radial curve, \(d\theta=d\xi=0\), and hence
\begin{equation}
ds_{g_i}=\varphi_i^{-1}\,dr.
\end{equation}
By Lemma \ref{lem: varphi-bounds}, 
\begin{equation}
\int_0^1 \varphi_i^{-1}\,dr
\le
C(\Lambda)\int_0^1 r^{-1/2}\,dr
=
2C(\Lambda),
\end{equation}
and
\begin{equation}
\int_1^2 \varphi_i^{-1}\,dr
\le
C(\Lambda)\int_1^2(2-r)^{-1/2}\,dr
=
2C(\Lambda).
\end{equation}
Thus every point can be joined to the middle slice $\{r=1\}$ by a radial curve of length at most $2C(\Lambda)$. Consequently,
\begin{equation}
    L_{g_i} (\gamma_1) + L_{g_i}(\gamma_4) \leq 4 C(\Lambda), \quad \forall i \in \mathbb{N}.
\end{equation}

Next consider the \(\theta\)-direction curve $\gamma_2$ along the slice \(r=1\). There,
\begin{equation}
ds_{g_i}
=
\varphi_i(1,\theta)^{-1}u_i(1)\,d\theta.
\end{equation}
By Proposition \ref{prop: u_i-convergence} and Lemma \ref{lem: varphi-bounds},  we have $u_i(1)\leq 1 $ and $\varphi_i(1,\theta)^{-1}\leq C(\Lambda)$. Hence
\begin{equation}
\int^{2\pi}_{0}\varphi_i(1,\theta)^{-1}u_i(1) d\theta \leq 2\pi C(\Lambda),
\end{equation}
and therefore
\begin{equation}
   L_{g_i}(\gamma_2) \leq 2\pi C(\Lambda), \quad \forall i \in \mathbb{N}.
\end{equation}

Finally, for $\gamma_3$, along the \(\xi\)-direction on the slice \(r=1\),
\begin{equation}
ds_{g_i}=\varphi_j(1,\bar\theta)\,d\xi.
\end{equation}
Again by Lemma \ref{lem: varphi-bounds}, $\varphi_j(1,\bar\theta)\le C$. Hence
\begin{equation}
 L_{g_i}(\gamma_3) \leq 2\pi C(\Lambda), \quad \forall i \in \mathbb{N}.
\end{equation}
Combining the estimates, we obtain
\begin{equation}
d_{g_j}(p,q)
\le
4C(\Lambda)+2\pi C(\Lambda)+2\pi C(\Lambda)
=
(4+4\pi)C(\Lambda), 
\quad
\forall i \in \mathbb{N}.
\end{equation}
Since \(p,q\) were arbitrary, the desired uniform diameter bound follows.
\end{proof}


\subsection{Convergence of the volumes}

We next prove convergence of the total volumes.  On compact subsets of the regular region, the convergence follows from the uniform convergence of the metric coefficients.  Near the two polar circles, the volume contribution is uniformly small by the endpoint estimates for $u_i$ and $\varphi_i$.

\begin{proposition}\label{prop: volume-convergence}
  Let $\{ g_i \}$ be a sequence of the warped product Riemannian metrics on $ \Sph^2 \times \Sph^1$ as in \eqref{metric-form} satisfying the condition $\Lambda$ as in \eqref{eqn: condition-lambda-estimate-varphi}. Then 
  \begin{equation}
  \Vol_{g_i}(\mathring{M}) \to \Vol_{g_\infty}(\mathring{M}),  \quad \text{as} \ \ i \to \infty,
  \end{equation}
  where $\mathring{M}$ is defined in \eqref{eqn: mathring-M-defn}.
  
  Moreover,
  \begin{equation}\label{eqn: limit-cap-volume-tend-zero}
  \Vol_{g_\infty}(\mathring{M} \setminus M_\epsilon) \to 0 \ \ \text{and} \ \  \Area_{g_\infty}(\partial M_\epsilon) \to 0, \quad \text{as} \ \ \epsilon \to 0,
  \end{equation}
  where
  \begin{equation}
  M_\epsilon:= [\epsilon, 2-\epsilon] \times \Sph^1_\theta \times \Sph^1_\xi \subset \Sph^2 \times \Sph^1.
  \end{equation}
 \end{proposition}

\begin{proof}
The volume form of $g_i$ is 
\begin{equation}
d\vol_{g_j}
=
\frac{u_j}{\varphi_j}\,dr\,d\theta\,d\xi.
\end{equation}
For $ 0 < \epsilon <1$, we split $\mathring{M}$ into the regular middle region and two cap regions:
\begin{equation}
\mathring{M} = M_\epsilon \cup (0, \epsilon) \times \Sph^1_\theta  \times \Sph^1_\xi \cup (2-\epsilon, 2) \times \Sph^1_\theta \times \Sph^1_\xi.
\end{equation}

We first estimate the volume of the cap regions. By Proposition \ref{prop: u_i-convergence} and Lemma \ref{lem: varphi-bounds}, we have
\begin{equation}\label{eqn: volume-density-estimate-g-i}
   \frac{u_i(r)}{\varphi_i(r, \theta)} \leq C(\Lambda) r^{-\frac{1}{2}}, \ \ \forall (r, \theta) \in (0, 1] \times [0, 2\pi], \ \ i \in \mathbb{N}.
\end{equation}
Therefore, 
\begin{equation}
\begin{aligned}
\Vol_{g_i}\left((0, \epsilon) \times \Sph^1_\theta  \times \Sph^1_\xi \right)
&=
\int_0^{2\pi}\int_0^{2\pi}\int_0^\epsilon
\frac{u_i(r)}{\varphi_i(r,\theta)}
\,dr\,d\theta\,d\xi \\
&\le
4\pi^2 C(\Lambda)\int_0^\epsilon r^{-1/2}\,dr \\
&=
8\pi^2 C(\Lambda)\epsilon^{1/2}.
\end{aligned}
\end{equation}
Similarly, 
\begin{equation}
\Vol_{g_i}\left((2-\epsilon, 2) \times \Sph^1_\theta  \times \Sph^1_\xi \right)
\leq
8\pi^2 C(\Lambda)\epsilon^{1/2}.
\end{equation}
Thus
\begin{equation}\label{eqn: cap-volume-estimate}
\Vol_{g_i}(\mathring{M} \setminus M_\epsilon)
\le
C(\Lambda)\epsilon^{1/2}, \quad \forall i \in \mathbb{N}.
\end{equation}

The same estimate holds for the limit $g_\infty$. Indeed, since $u_i \to u_\infty$ and $\varphi_i \to \varphi_\infty$ pointwise on $(0, 2) \times [0, 2\pi]$, passing to the limit for \eqref{eqn: volume-density-estimate-g-i} gives
\begin{equation}\label{eqn: volume-density-estimate-g-infty}
   \frac{u_\infty(r)}{\varphi_\infty(r, \theta)} \leq C(\Lambda) r^{-\frac{1}{2}}, \ \ \forall (r, \theta) \in (0, 1] \times [0, 2\pi], \ \ i \in \mathbb{N}.
\end{equation}
The analogous estimate holds near \(r=2\). Hence
\begin{equation}\label{eqn: limit-cap-volume-estimate}
\Vol_{g_\infty}(\mathring{M} \setminus M_\epsilon)
\le
C(\Lambda) \epsilon^{1/2}.
\end{equation}

Moreover, by Proposition \ref{prop: u_i-convergence},
\begin{equation}
\Area_{g_\infty}(\partial M_\epsilon) = 4\pi^2 (u_\infty(\epsilon) + u_{\infty}(2-\epsilon)) \to 0, \ \ \text{as} \ \ \epsilon \to 0.
\end{equation}
Together with \eqref{eqn: cap-volume-estimate}, this gives \eqref{eqn: limit-cap-volume-tend-zero}.

Now fix \(\epsilon>0\). By  Propositions \ref{prop: u_i-convergence} and  \ref{prop: varphi-C01/2-estimate}
we have
\begin{equation}
\frac{u_j}{\varphi_j}
\to
\frac{u_\infty}{\varphi_\infty}
\quad\text{uniformly on } \Omega_\epsilon.
\end{equation}
and
\begin{equation}
g_j\to g_\infty
\quad\text{uniformly on } M_\epsilon.
\end{equation}
Consequently, we have uniform convergence of the volume densities on $M_\epsilon$, and hence
\begin{equation}
\Vol_{g_j}(M_\epsilon)
\to
\Vol_{g_\infty}(M_\epsilon).
\end{equation}
Finally,
\begin{equation}
\begin{aligned}
& \quad \left|
\Vol_{g_i}(\mathring{M})-\Vol_{g_\infty}(\mathring{M})
\right| \\
&\le
\left|
\Vol_{g_i}(M_\epsilon)-\Vol_{g_\infty}(M_\epsilon)
\right| 
+\Vol_{g_i}(\mathring{M} \setminus M_\epsilon)
+\Vol_{g_\infty}( \mathring{M} \setminus M_\epsilon).
\end{aligned}
\end{equation}
Taking $\limsup\limits_{i \to \infty}$ with $ \epsilon$ fixed, and using \eqref{eqn: cap-volume-estimate} and \eqref{eqn: limit-cap-volume-estimate}, we obtain
\begin{equation}
\limsup_{i\to\infty}
\left|
\Vol_{g_i}(\mathring{M})-\Vol_{g_\infty}(\mathring{M})
\right|
\le
C(\Lambda) \epsilon^{1/2}.
\end{equation}
Letting $\epsilon\to0$ gives the desired convergence of volumes.
\end{proof}


\subsection{The limit metric space}

We now construct the limiting distance on the regular region.  The uniform convergence of the metrics on compact subsets of $\mathring M$ gives local equicontinuity of the distance functions.  A diagonal Arzel\`a--Ascoli argument then produces a limiting metric $d_*$, which is locally induced by $g_\infty$ but may be globally no larger than the intrinsic distance $d_{g_\infty}$.
  
\begin{proposition}\label{prop: limit-metric-space}
  Let $\{ g_i \}$ be a sequence of the warped product Riemannian metrics on $ \Sph^2 \times \Sph^1$ in \eqref{metric-form}, satisfying the condition $\Lambda$ in \eqref{eqn: condition-lambda-estimate-varphi}.
  There exists a metric $d_*$
         on $\mathring{M}$ such that, after passing to a subsequence,
          \begin{equation}
         d_{g_i}\big|_{K\times K}\to d_*
         \quad\text{uniformly on }K\times K
         \end{equation}
         for every compact set $ K\Subset \mathring{M}$. Moreover,
         \begin{equation}
         d_* \leq d_{g_\infty}
        \quad\text{on }\mathring{M}\times \mathring{M},
       \end{equation}
        and $d_*$ agrees locally with $d_{g_\infty}$. In particular, 
        $d_*$ induces the same local topology and the same local Riemannian current
         structure as $g_\infty$.
\end{proposition}
\begin{proof}
We write
\begin{equation}
d_i:=d_{g_i}.
\end{equation}
For each sufficiently large integer $m$, let
\begin{equation}
K_m:=\left\{\frac1m\le r\le a-\frac1m\right\}\times S^1 \subset \Sph^2 \times \Sph^1.
\end{equation}
Fix \(m\), and choose \(m'>m\) such that
\begin{equation}
K_m\Subset {\rm Int}K_{m'}\Subset \mathring M.
\end{equation}
Let $g_0$ be the standard product metric on $\Sph^2 \times \Sph^1$, with
distance $d_0$. By  Propositions \ref{prop: u_i-convergence} and  \ref{prop: varphi-C01/2-estimate}, 
\begin{equation}
g_i\to g_\infty
\quad\text{uniformly on }K_{m'}.
\end{equation}
Thus, there exists \(C_m<\infty\), independent of \(i\), such that
\begin{equation}
g_i\le C_m g_0
\quad\text{on }K_{m'}, \quad \forall i \in \mathbb{N}
\end{equation}
Since $K_m\Subset\operatorname{Int}K_{m'}$, there exists $\rho_m>0$
such that, whenever $x,x'\in K_m$ and
\begin{equation}
d_0 (x,x')<\rho_m,
\end{equation}
a $g_0$-short curve joining $x$ to $x'$ is contained in
\(K_{m'}\). Hence
\begin{equation}
d_i(x,x')
\le
\sqrt{C_m}\, d_0(x, x')
\end{equation}
for such $x, x'$.

Therefore, 
for any $(x, y), (x', y') \in K_m \times K_m$ satsifying
\begin{equation}
d_0 (x, x')<\rho_m \quad \text{and} \quad d_0 (y, y')<\rho_m,
\end{equation}
by the triangle inequality,
\begin{equation}
\begin{aligned}
|d_i(x,y)- d_i(x',y')|
&\le d_i(x,x')+d_i(y,y')  \\
&\le \sqrt{C_m}\, d_0(x,x')
+\sqrt{C_m}\, d_0(y,y').
\end{aligned}
\end{equation}
Thus
$d_i(x, y)$ is equicontinuous on $K_m\times K_m$. Moreover,
\begin{equation}
d_i(x,y)\le \operatorname{Diam}(M_i,g_i)\le D_0, \quad \forall (x, y) \in K_m \times K_m,
\end{equation}
for some constant $D_0$ independent of $i$.
By Arzel\`a--Ascoli, after passing to a subsequence,
$
d_i\big|_{K_m\times K_m}
$
converges uniformly on \(K_m\times K_m\). A diagonal argument in \(m\)
gives a function
\begin{equation}
d_*:\mathring M\times\mathring M\to[0,\infty)
\end{equation}
such that
\begin{equation}
d_i\big|_{K\times K}\to d_*\big|_{K\times K}
\end{equation}
uniformly for every compact set \(K\Subset\mathring M\).

Since each $d_i$ is symmetric and satisfies the triangle inequality, the
limit $d_*$ is symmetric and satisfies the triangle inequality. 
Clearly
$d_*(x,x)=0$.
It remains to show that $d_*(x,y)>0$ whenever $x\ne y \in \mathring{M}$. Fix distinct
points \(x,y\in\mathring M\). Choose a relatively compact open neighborhood
\(B\Subset\mathring M\) of \(x\) such that
\begin{equation}
y\notin \overline B.
\end{equation}
Since \(g_i\to g_\infty\) uniformly on \(\overline B\), and since
\(g_\infty\) is positive definite on \(\overline B\), there exists a
constant \(c>0\), independent of \(i\), such that for all sufficiently large
\(i\),
\begin{equation}
g_i\ge c\,g_0
\qquad\text{on }\overline B.
\end{equation}
Let
\begin{equation}
\rho:=\operatorname{dist}_{g_0}(x,\partial B)>0.
\end{equation}
Every curve joining $x$ to $y$ must leave $B$. Therefore its portion
from $x$ to the first exit point from $B$ has $g_0$-length at least
$\rho$, and hence $g_i$-length at least $\sqrt c\,\rho$. Taking the
infimum over all curves from $x$ to $y$, we get
\begin{equation}
d_i(x,y)\ge \sqrt c\,\rho
\end{equation}
for all sufficiently large \(i\). Passing to the limit gives
\begin{equation}
d_*(x,y)\ge \sqrt c\,\rho>0.
\end{equation}
Thus  $d_*$ separates points, and hence $d_*$ is a metric on
\(\mathring M\).

We next comparison $d_*$ with $d_{g_\infty}$.
We first show that
\begin{equation}
d_*\le d_{g_\infty}.
\end{equation}
Let $ x,y\in\mathring M$, and let $\gamma\subset\mathring M$ be any
smooth curve joining $x$ to $y$. Since $\gamma$ is compactly contained
in $\mathring M$, the uniform convergence of $g_i$ to $g_\infty$ along
\(\gamma\) gives
\begin{equation}
L_{g_i}(\gamma)\to L_{g_\infty}(\gamma).
\end{equation}
Since
\begin{equation}
d_i(x,y)\le L_{g_i}(\gamma),
\end{equation}
we obtain
\begin{equation}
d_*(x,y)\le L_{g_\infty}(\gamma).
\end{equation}
Taking the infimum over all such \(\gamma\), we get
\begin{equation}
d_*(x,y)\le d_{g_\infty}(x,y).
\end{equation}

It remains to prove that $d_*$ agrees locally with $d_{g_\infty}$. Fix
$ x\in\mathring M$. Choose open sets
\begin{equation}
B'\Subset B\Subset \mathring M
\end{equation}
with $x\in B'$. We claim that, after shrinking $B'$ if necessary, every
$g_i$-minimizing curve between points of $B'$ remains inside $B$, for
all large $i$.

Indeed,  set
\begin{equation}
\eta:=\operatorname{dist}_{g_0}(B',\partial B)>0.
\end{equation}
Since \(g_i\to g_\infty\) uniformly on \(\overline B\), there exist a constant
$C>0$, independent of \(i\), such that
\begin{equation}
C^{-1}g_0 \le g_i\le C g_0
\quad\text{on }\overline B.
\end{equation}
Any curve joining two points of $B'$ and leaving $B$ has $g_i$-length
at least $C^{-\frac{1}{2}}\eta$. On the other hand, any two points $x, y\in B'$
can be joined inside $B$ by a curve of $g_i$-length at most
\begin{equation}
\sqrt C\,d_0(p,q)
\le
\sqrt C\, \Diam_{g_0}(B').
\end{equation}
After shrinking $B'$ so that
\begin{equation}
\sqrt C\, \Diam_{g_0}(B')
<
C^{-\frac{1}{2}}\,\eta,
\end{equation}
no $g_i$-minimizing curve between points of $B'$ can leave $B$. Therefore, for
$p, q \in B'$,
\begin{equation}\label{eqn: di-dBi}
d_i(p, q)=d_{g_i}^{B}(p, q), \quad \forall i \in \mathbb{N}
\end{equation}
where $d_{g_i}^{B}$ denotes the intrinsic length distance computed among
curves contained in $B$.

Since $g_i\to g_\infty$ uniformly on $\overline B$, the intrinsic
distances $d_{g_i}^{B}$ converge uniformly on $B'\times B'$ to
\(d_{g_\infty}^{B}\). Therefore
\begin{equation}
d_*(p, q)=d_{g_\infty}^{B}(p, q)
\qquad \forall p, q\in B'.
\end{equation}
Moreover, since $g_\infty$ is continuous and positive definite on $\bar{B}$, the same argument used to obtain \eqref{eqn: di-dBi} gives, after shrinking \(B'\) further if necessary,
\begin{equation}
d_{g_\infty} (p, q) = d^B_{g_\infty}(p, q) \quad \forall p, q \in B'.
\end{equation}
Therefore $d_*$ agrees locally with the Riemannian distance of $g_\infty$.
\end{proof}


\subsection{Gromov-Hausdorff convergence}

We now promote the local convergence of distances on $\mathring M$ to global Gromov--Hausdorff convergence.  The cap estimates show that the parts near \(r=0\) and \(r=2\) are uniformly close to the truncated region \(M_\delta\).  On each fixed truncated region, the distances converge uniformly.  Combining these two facts yields convergence to the completion of $(\mathring M,d_*)$.

\begin{proposition}
Let $\{ g_i \}$ be a sequence of the warped product Riemannian metrics on $ \Sph^2 \times \Sph^1$ as in \eqref{metric-form}, satisfying the condition $\Lambda$ in \eqref{eqn: condition-lambda-estimate-varphi}. Let $d_*$ be the metric on $\mathring{M}=\{0 < r < 2\} \times \Sph^1 \subset \Sph^2 \times \Sph^1$ obtained in Proposition. Denote
\begin{equation}
(M_i, d_i) := (\Sph^2 \times \Sph^1, d_{g_i}).
\end{equation} 
Then, after passing to a subsequence,
           \begin{equation}
               (M_i, d_i) \longrightarrow Y
           \end{equation}
           in the Gromov-Hausdorff sense, where 
           \begin{equation}
           Y := \overline{(\mathring{M}, d_*)}
           \end{equation} 
           is the metric completion of $(\mathring{M}, d_*)$.
\end{proposition}
\begin{proof}
For  $0 \leq \delta <1$, let
\begin{equation}
M_\delta:=\{\delta\le r\le a-\delta\} \times S^1 \subset \Sph^2 \times \Sph^1.
\end{equation}
We first prove the uniform Hausdorff estimate
\begin{equation}
d_H^{(M_i,d_i)}(M_i,M_\delta)\le \omega(\delta), \quad \forall i \in \mathbb{N},
\end{equation}
where $\omega(\delta)\to0$ as $\delta\to0$.

Let $x=(r,\theta,\xi) \in M_i$ with $0<r<\delta$. Join $x$ to
\begin{equation}
x_\delta:=(\delta,\theta,\xi)
\end{equation}
by the radial segment. Along this segment,
\begin{equation}
ds_{g_i}=\varphi_i^{-1}\,dr.
\end{equation}
By Lemma \ref{lem: varphi-bounds}, there exists a constant $C$, independent of $i$, such that
\begin{equation}
\varphi_i(r,\theta) ^{-1} \leq C r^{-\frac{1}{2}}, \quad \forall (r, \theta) \in (0, 1) \times [0, 2\pi], \ \ \forall i \in \mathbb{N}
\end{equation}
Consequently,
\begin{equation}
d_i(x,x_\delta)
\le
\int_r^\delta \varphi_i(s,\theta)^{-1}\,ds
\le
C\int_0^\delta s^{-\frac{1}{2}}\,ds
=
2C \delta^{\frac{1}{2}}.
\end{equation}
The same estimate near $r=2$ gives
\begin{equation}
d_i(x,M_\delta)\le
2C\delta^{\frac{1}{2}}, \quad \forall x \in M_i.
\end{equation}
Thus
\begin{equation}\label{eqn: Hausdorff-dist-Mi}
d_H^{(M_i,d_i)}(M_i,M_\delta)
\le
\omega(\delta), \ \ \forall i \in \mathbb{N},
\qquad
\omega(\delta):=2C \delta^{\frac{1}{2}}.
\end{equation}

We also need the corresponding density statement in the limit space
$
\overline{(\mathring M,d_*)}.
$
Let \(x\in\mathring M\) with \(0<r<\delta\). Since
$
d_*\le d_{g_\infty},
$
by Proposition \ref{prop: limit-metric-space},
and since $\varphi_\infty(r,\theta)^{-1}\leq C r^\frac{1}{2}$ by Remark \ref{rmrk: estimate-varphi-infty}, the same radial
curve gives
\begin{equation}
d_*(x,x_\delta)
\le
d_{g_\infty}(x,x_\delta)
\le
\omega(\delta).
\end{equation}
The same estimate holds near $r=2$. Passing to limits of Cauchy sequences in
\((\mathring M,d_*)\), we obtain
\begin{equation}\label{eqn: Hausdorff-dist-Y}
d_H^Y\bigl(Y,\overline{M_\delta}^{\,d_*}\bigr)
\le
\omega(\delta).
\end{equation}
In particular, $Y$ is compact, since it is complete and is uniformly
approximated by the compact sets $\overline{M_\delta}^{\,d_*}$.

Fix $0<\delta< 1$. By construction of \(d_*\),
\begin{equation}
d_i\big|_{M_\delta\times M_\delta}
\to
d_*\big|_{M_\delta\times M_\delta}
\quad
\text{uniformly}.
\end{equation}
Set
\begin{equation}
\lambda_{i,\delta}:=
\sup_{x,y\in M_\delta}
\left|
d_i(x,y)-d_*(x,y)
\right|.
\end{equation}
Then
\begin{equation}\label{eqn: distortion-limit}
\lambda_{i,\delta} \to 0, \quad \text{as} \ \ i \to \infty.
\end{equation}
The identity correspondence between
\begin{equation}
(M_\delta,d_i|_{M_\delta\times M_\delta})
\quad\text{and}\quad
(M_\delta,d_*|_{M_\delta\times M_\delta})
\end{equation}
has distortion at most \(2\lambda_{i,\delta}\). Hence
\begin{equation}\label{eqn: GH-distortion}
d_{GH}\left(
(M_\delta,d_i),
(M_\delta,d_*)
\right)
\le
\lambda_{i,\delta},
\end{equation}
see, for example, \cite[Theorem 7.3.25]{BBI-book}.

Using the triangle inequality for $d_{GH}$, and viewing $M_\delta$ as a
subspace of both $M_i$ and of $Y$, we get
\begin{equation}
\begin{aligned}
d_{GH}\bigl((M_i,d_i),Y\bigr)
&\le
d_H^{(M_i,d_i)}(M_i,M_\delta)\\
&\quad+
d_{GH}\left((M_\delta,d_i),(M_\delta,d_*)\right)\\
&\quad+
d_H^Y\bigl(Y,\overline{M_\delta}^{\,d_*}\bigr).
\end{aligned}
\end{equation}
Therefore, by  \eqref{eqn: Hausdorff-dist-Mi}, \eqref{eqn: Hausdorff-dist-Y}, \eqref{eqn: distortion-limit} and \eqref{eqn: GH-distortion}, 
\begin{equation}
\limsup_{i\to\infty}
d_{GH}\bigl((M_i,d_i),Y\bigr)
\le
2\omega(\delta).
\end{equation}
Letting $\delta\to0$, we obtain
\begin{equation}
\lim_{i\to\infty}
d_{GH}\bigl((M_i,d_i),Y\bigr)=0.
\end{equation}
This completes the proof.
\end{proof}

\subsection{Intrinsic flat convergence}

Finally, we identify the intrinsic flat limit.  Since $d_*$ agrees locally with the Riemannian distance of $g_\infty$, the oriented Riemannian current associated with $g_\infty$ defines a natural integral current structure on the regular region.  We then use an exhaustion by compact submanifolds with boundary, apply the Lakzian--Sormani estimate on each fixed compact piece, and control the remaining caps by their small volume.

Recall that an integral current space $(X, d, T)$ consists of a metric space $(X, d)$ together with an integral current structure $T$. An oriented Riemannian manifold $(M^m, g)$ of finite volume can be viewed as an integral current space with its induced metric and the canonical current $T$, which
acts on smooth compactly supported differential $m$-forms $\omega$ by
\begin{equation}
  T(\omega) = \int_M \omega.
\end{equation}
The mass $\mathbf M(T)$ of an integral current space may be understood as a weighted volume, and the boundary $\partial T$ is defined so that Stokes' theorem holds. In particular, when the integral current space comes from an oriented Riemannian manifold, its mass is the Riemannian volume and its boundary is just the usual boundary. We refer to Ambrosio-Kirchheim \cite{AmbrosioKirchheim2000} for the theory of integral currents in metric spaces.

Let $Z$ be a metric space, and let $T_1$ and $T_2$ be two $m$-integral currents on $Z$. The flat distance between $T_1$ and $T_2$, in the sense of Federer--Fleming, is defined by 
\begin{equation}
    d^Z_F(T_1, T_2) = \inf \{\mathbf M(B^{m+1}) + \mathbf M(A^m) \mid T_1 - T_2 = A + \partial B\},
\end{equation}
where the infimum is take over all $m$-dimensional integral current $A$ and $(m+1)$-dimensional integral currents $B$ in $Z$.

The {\em Sormani-Wenger intrinsic flat distance} between two integral current spaces $(X_1, d_1, T_1)$ and $(X_2, d_2, T_2)$ is defined by
\begin{equation}
    d_{\mathcal F} ((X_1, d_1, T_1), (X_2, d_2, T_2)) = \inf 
    \{d^Z_{F}(\varphi_{1\#}T_1, \varphi_{2\#}T_2) \mid \varphi_i: X_i \to Z\},
\end{equation}
where the infimum is taken over all complete metric spaces $Z$ and all isometric embeddings $\varphi_i: X_i \to Z$, where $\varphi_{i\#}$ is the push-forward map on integral currents. We refer to Sormani-Wenger \cite{SormaniWenger2011} for the definition and basic properties of the intrinsic flat distance.

We now apply this framework to the limiting space obtained in Proposition \ref{prop: limit-metric-space}.  By
that proposition, the distance $d_*$ agrees locally
with the Riemannian distance induced by $g_\infty$ on $\mathring M$. Therefore
the oriented Riemannian current determined by $g_\infty$ defines an integral
current on the metric space $(\mathring M,d_*)$.  We denote this current by
\begin{equation}
T=[\![\mathring{M}]\!]_{g_\infty}.
\end{equation}
Its mass is
\begin{equation}
\mathbf M(T)=\operatorname{Vol}_{g_\infty}(\mathring{M})<\infty.
\end{equation}

To see that $T$ has no boundary in the completion, consider the exhaustion
\begin{equation}
K_m
=
\left\{
\frac1m\le r\le 2-\frac1m
\right\}
\times \Sph^1_\theta\times \Sph^1_\xi .
\end{equation}
By Proposition \ref{prop: volume-convergence}, we have
\begin{equation}\label{eqn: cap-volume-go-zero}
\Vol_{g_\infty}(\mathring M\setminus K_m)\to0,
\qquad
\Area_{g_\infty}(\partial K_m)\to0 .
\end{equation}
It follows that
\begin{equation}
[\![K_m]\!]_{g_\infty}\to T
\end{equation}
as integral currents, and the boundary masses satisfy
\begin{equation}
\mathbf M(\partial[\![K_m]\!]_{g_\infty})
=
\Area_{g_\infty}(\partial K_m)\to0.
\end{equation}
Hence $\partial T=0$.  Thus $T$ defines an integral current structure on the
metric completion of $(\mathring M,d_*)$.

Let
\begin{equation}
(X,d_X,T)
\end{equation}
denote the settled completion of this integral current space. We then obtain the following intrinsic flat convergence result.

\begin{proposition}
Let $\{ g_i \}$ be a sequence of the warped product Riemannian metrics on $ \Sph^2 \times \Sph^1$ as in \eqref{metric-form}, satisfying the condition $\Lambda$ in \eqref{eqn: condition-lambda-estimate-varphi}. Then, after passing to a subsequence,
          \begin{equation}
          (M_i,d_{g_i},[\![M_i]\!])
          \longrightarrow
          (X,d_X,T)
           \end{equation}
          in the Sormani--Wenger intrinsic flat sense.
\end{proposition}
\begin{proof}

For each $m$, let
\begin{equation}
T^m:=T\llcorner K_m.
\end{equation}
Similarly, let
\begin{equation}
T_i=[\![M_i]\!]_{g_i},
\qquad
T_i^m:=T_i\llcorner K_m.
\end{equation}
We view
\begin{equation}
K_m^i:=(K_m,d_{g_i}|_{K_m\times K_m},T_i^m)
\end{equation}
as a subspace of \(M_i\) with the restricted ambient distance, and
\begin{equation}
K_m^*:=(K_m,d_*|_{K_m\times K_m},T^m)
\end{equation}
as a subspace of $X$ with the restricted distance.

For fixed $m$, Proposition \ref{prop: limit-metric-space} gives
\begin{equation}
d_{g_i}\big|_{K_m\times K_m}\to d_*
\quad\text{uniformly on }K_m\times K_m.
\end{equation}
Equivalently,
\begin{equation}
\lambda_i := \sup_{x, y \in K_m} |d_i(x, y,) - d_*(x, y)| \to 0, \quad \text{as} \ \ i \to \infty.
\end{equation}
Also, since
\begin{equation}
g_i\to g_\infty
\quad\text{uniformly on }K_m,
\end{equation}
the corresponding mass measures on \(K_m\) converge:
\begin{equation}
\mathbf M(T_i^m)
=
\operatorname{Vol}_{g_i}(K_m)
\to
\operatorname{Vol}_{g_\infty}(K_m)
=
\mathbf M(T^m).
\end{equation}
The boundary masses are uniformly bounded for fixed $m$, because
$\partial K_m$ is a fixed smooth hypersurface and the metrics converge
uniformly there.
Hence, by the
Lakzian--Sormani bridge estimate \cite[Theorem 4.6]{LakzianSormani2013},
applied to the two copies of the same compact manifold-with-boundary
$K_m$, we obtain
\begin{equation}\label{eqn: interior-flat-distance-limit}
d_{\mathcal F}(K_m^i,K_m^*)\to0, \quad \text{as} \ \ i \to \infty,
\end{equation}
for every fixed $m$.

Since $K_m^i$ is equipped with the restricted ambient distance from
$(M_i,d_{g_i})$, the inclusion
\begin{equation}
K_m^i\hookrightarrow M_i
\end{equation}
is isometric. Hence, using the common ambient space $M_i$,
\begin{equation}
T_i-T_i^m=T_i\llcorner(M_i\setminus K_m).
\end{equation}
By the definition of the flat norm, 
\begin{equation}
d_{\mathcal F}(M_i,K_m^i)
\le
\mathbf M\bigl(T_i\llcorner(M_i\setminus K_m)\bigr)
=
\operatorname{Vol}_{g_i}(M_i\setminus K_m).
\end{equation}

Similarly, since $K_m^*$ is equipped with the restricted distance from
$X$, we have
\begin{equation}
d_{\mathcal F}(X,K_m^*)
\le
\mathbf M\left(T\llcorner(X\setminus K_m)\right)
=
\Vol_{g_\infty}(\mathring{M}\setminus K_m).
\end{equation}

By the triangle inequality,
\begin{equation}
\begin{aligned}
d_{\mathcal F}(M_i,X)
&\le
d_{\mathcal F}(M_i,K_m^i)
+
d_{\mathcal F}(K_m^i,K_m^*)
+
d_{\mathcal F}(K_m^*,X)\\
&\le
\operatorname{Vol}_{g_i}(M_i\setminus K_m)
+
d_{\mathcal F}(K_m^i,K_m^*)
+
\operatorname{Vol}_{g_\infty}(\mathring{M}\setminus K_m).
\end{aligned}
\end{equation}
Taking $\limsup\limits_{i\to\infty}$, and using \eqref{eqn: interior-flat-distance-limit},
we obtain
\begin{equation}
\limsup_{i\to\infty}d_{\mathcal F}(M_i,X)
\le
\limsup_{i\to\infty}\operatorname{Vol}_{g_i}(M_i\setminus K_m)
+
\operatorname{Vol}_{g_\infty}(\mathring{M}\setminus K_m).
\end{equation}
By the \eqref{eqn: cap-volume-go-zero}, the right-hand side tends to zero as
$m\to\infty$. Therefore
\begin{equation}
\limsup_{i\to\infty}d_{\mathcal F}(M_i,X)=0.
\end{equation}
This proves the proposition.
\end{proof}


\section{Nonnegative distributional scalar curvature of the limit metric}\label{sect: distributional-scalar}

The goal of this section is to prove that the limit metric obtained in the previous
section retains nonnegative scalar curvature in the distributional sense.  Since
$g_\infty$ may fail to be smooth, scalar curvature
cannot be interpreted pointwise everywhere.  We therefore use the distributional
formulation of Lee--LeFloch \cite{LeeLeFloch2015}.  We first establish strong
$L^p$ convergence of the metric tensors.  We then compute the Lee--LeFloch
quantities $V$ and $F$ for our warped product ansatz, prove nonnegative
distributional scalar curvature away from the singular circles, and finally pass
to the whole manifold by controlling the boundary terms near the poles.


\subsection{$L^p$ convergence of the Riemannian metrics}\label{subsect: metric-convergence}

We begin with the convergence of the metric tensors themselves.  Although the
limit metric is naturally defined only on the regular part $\mathring M$, the
estimates for $u_i$ and $\varphi_i$ obtained earlier give enough control near
the two missing circles to obtain global $L^p$ convergence with respect to a
fixed smooth background metric.

\begin{proposition}\label{prop: Lp-convergence-of-metrics}
Let $\{ g_i \}$ be a sequence of the warped product Riemannian metrics on $ M = \Sph^2 \times \Sph^1$ as in \eqref{metric-form}, satisfying the condition $\Lambda$ in \eqref{eqn: condition-lambda-estimate-varphi}. 
Then, for every
\(1\le p<\infty\),
\[
g_i\to g_\infty
\quad\text{in }L^p(M,g_0)
\]
as symmetric \((0,2)\)-tensors, where $g_\infty$ is defined in \eqref{eqn: limit-metric-g-infty}, and $g_0$ is the standard product metric on $\Sph^2 \times \Sph^1$.
\end{proposition}
\begin{remark}
{\rm
The metric $g_\infty$ is defined pointwise on $\mathring M$.  Since
$M\setminus \mathring M$ has measure zero with respect to the
background metric $g_0$ on $M$, we may also regard $g_\infty$ as a measurable
symmetric $2$-tensor on $M$ after assigning arbitrary values on
$M\setminus \mathring M$.  This convention does not affect any $L^p$-statement.
}
\end{remark}
\begin{proof}
On $\mathring M$, we have
\begin{equation}
g_i-g_\infty
=
\left(\varphi_i^{-2}-\varphi_\infty^{-2}\right)dr^2
+
\left(\varphi_i^{-2}u_i^2-\varphi_\infty^{-2}u_\infty^2\right)d\theta^2
+
\left(\varphi_i^2-\varphi_\infty^2\right)d\xi^2 .
\end{equation}
By the  convergence of \(u_i\) and \(\varphi_i\) obtained in Propositions \ref{prop: u_i-convergence} and \ref{prop: varphi-sobolev-estimate}, the metric
coefficients converge pointwise a.e. on \(\mathring M\). Hence
\begin{equation}
g_i\to g_\infty
\end{equation}
pointwise a.e. on $M$ as symmetric $2$-tensors.

Fix $1\le p<\infty$. Choose an integer $k>p$. By
Lemma \ref{lem: varphi-bounds}, for this fixed $k$ we have
\begin{equation}
\varphi_i^2\le C(\Lambda,k)^2 r^{-2/k},
\qquad
\varphi_i^{-2}\le C(\Lambda,k)^2 r^{-2/k},
\end{equation}
and the same estimates hold for $\varphi_\infty$, see Remark \ref{rmrk: estimate-varphi-infty}. Although the constant
$C(\Lambda,k)$ may depend on $k$, it is independent of $i$, which is all
that is needed for dominated convergence.

Thus the \(dr^2\)- and \(d\xi^2\)-components are dominated by
\begin{equation}
C r^{-2/k}.
\end{equation}
For the angular component, using the endpoint control
\begin{equation}
u_i(r)\le Cr,
\qquad
u_\infty(r)\le Cr
\end{equation}
and the fact that, near \(r=0\), the background metric satisfies
\begin{equation}
g_0=dr^2+u_0^2\,d\theta^2+d\xi^2,
\qquad
u_0(r)\sim r,
\end{equation}
we have
\begin{equation}
|u_i^2d\theta^2|_{g_0}\le C,
\qquad
|u_\infty^2d\theta^2|_{g_0}\le C.
\end{equation}
Therefore
\begin{equation}
\left|
\left(\varphi_i^{-2}u_i^2-\varphi_\infty^{-2}u_\infty^2\right)d\theta^2
\right|_{g_0}
\le
C r^{-2/k}.
\end{equation}
The same argument applies near \(r=2\). Away from \(r=0\) and \(r=2\), the
estimate is immediate from local convergence and local boundedness. Hence
\begin{equation}
|g_i-g_\infty|_{g_0}
\le
C\left(r^{-2/k}+(2-r)^{-2/k}+1\right),
\end{equation}
where $C$ may depend on $p$, $\Lambda$, and the chosen $k>p$, but not
on $i$.

Since \(k>p\), the right-hand side belongs to \(L^p(M,g_0)\). Indeed, near
\(r=0\),
\begin{equation}
\int_0^\varepsilon r^{-2p/k}\,r\,dr<\infty
\end{equation}
because \(p<k\), and the same calculation holds near \(r=2\). Therefore the
dominated convergence theorem gives
\begin{equation}
\int_M |g_i-g_\infty|_{g_0}^p\,d\mu_{g_0}\to0.
\end{equation}
Thus
\begin{equation}
g_i\to g_\infty
\quad\text{in }L^p(M,g_0)
\end{equation}
for every $1\le p<\infty$.
\end{proof}


\subsection{Preliminary computations on distributional scalar curvature}

We next recall the Lee--LeFloch definition of scalar curvature in the
distributional sense and specialize it to the warped product metrics considered
here.  The purpose of the following computations is to express the distribution
in terms of quantities involving only $u$, $\varphi$, and their first derivatives.
This is important because the limit metric has precisely this level of weak
regularity on the regular part.

\begin{definition}[Lee--LeFloch \cite{LeeLeFloch2015}]\label{def: Lee-LeFloch}
{\rm Let $M$ be a smooth manifold endowed with a smooth background metric $g_0$. Let $g$ be a metric tensor defined on $M$ with $L_{l o c}^{\infty} \cap W_{l o c}^{1,2}$ regularity and locally bounded inverse $g^{-1} \in L_{l o c}^{\infty}$.

The {\em scalar curvature distribution} $R_g$ is defined as a distributions in $M$ such that for every test function $v \in C_{0}^{\infty}(M)$

\begin{equation}\label{NNSC-dis}
\left\langle R_g, v\right\rangle:=\int_M\left(-V \cdot \bar{\nabla}\left(v \frac{d \mu_g}{d \mu_{0}}\right)+F v \frac{d \mu_g}{d \mu_{0}}\right) d \mu_{0}.
\end{equation}
Here the dot product is taken using the metric $g_0$, $\bar{\nabla}$ is the Levi-Civita connection of $g_0$, $d \mu_g$ and $d \mu_{0}$ are volume measure with respect to $g$ and $g_0$ respectively, the vector field $V$ and the function $F$ are given by
\begin{equation}\label{NNSC-dis-V}
V^k:=g^{i j} \Gamma_{i j}^k-g^{i k} \Gamma_{j i}^j, 
\end{equation}
and
\begin{equation}\label{NNSC-dis-F}
F:=\bar{R}-\bar{\nabla}_k g^{i j} \Gamma_{i j}^k+\bar{\nabla}_k g^{i k} \Gamma_{j i}^j+g^{i j}\left(\Gamma_{k l}^k \Gamma_{i j}^l-\Gamma_{j l}^k \Gamma_{i k}^l\right),
\end{equation}
where
\begin{equation}\label{eqn: Lee-LeFloch-Gamma}
\Gamma_{i j}^k:=\frac{1}{2} g^{k l}\left(\bar{\nabla}_i g_{j l}+\bar{\nabla}_j g_{i l}-\bar{\nabla}_l g_{i j}\right), \\
\end{equation}
and
\begin{equation}
\bar{R}:=g^{i j}\left(\partial_k \bar{\Gamma}_{i j}^k-\partial_i \bar{\Gamma}_{k j}^k+\bar{\Gamma}_{i j}^l \bar{\Gamma}_{k l}^k-\bar{\Gamma}_{k j}^l \bar{\Gamma}_{i l}^k\right).
\end{equation}
The Riemannian metric $g$ is said to have {\em nonnegative distributional scalar curvature}, if $\langle R_g, u \rangle \geq 0$ for every nonnegative test function $u$ in the integral in (\ref{NNSC-dis}).
}
\end{definition}

We now apply this definition to a metric on $\Sph^2 \times \Sph^1$ of the form
\begin{equation}\label{eqn: metric-NNSC-distr}
g=\frac{1}{\varphi^2}h+\varphi^2d\xi^2,
\qquad
h=dr^2+u^2(r)d\theta^2,
\end{equation}
where \(\varphi=\varphi(r,\theta)>0\). We take the standard metric on $\Sph^2 \times \Sph^1$,
\begin{equation}\label{eqn: background-metric}
g_0=dr^2+\sin^2r\,d\theta^2+d\xi^2.
\end{equation}
as the background metric.
Let \(\bar\nabla\) and \(\bar\Gamma^k_{ij}\) denote respectively the Levi-Civita
connection and Christoffel symbols of \(g_0\). The only nonzero Christoffel
symbols of \(g_0\) are
\begin{equation}
\bar\Gamma^r_{\theta\theta}
=
-\sin r\cos r,
\qquad
\bar\Gamma^\theta_{r\theta}
=
\bar\Gamma^\theta_{\theta r}
=
\frac{\cos r}{\sin r}.
\end{equation}
Here we use coordinate $\{r, \, \theta, \, \xi\}$ on $\Sph^2 \times \Sph^1$, where $(r, \, \theta)$ is a polar coordinate on $\Sph^2$ and $\xi$ is a coordinate on $\Sph^1$. Under the corresponding local frame $\{\partial_r, \, \partial_\theta, \, \partial_\xi\}$, both $g$ and $g_0$ are diagonal and given as
\begin{equation}\label{eqn: diagonal-g_infty}
g_\infty=\begin{pmatrix}
\frac{1}{\varphi^2} & 0 &0\\
0& \frac{u^2}{\varphi^2} & 0\\
0&0& \varphi^2\\
\end{pmatrix}.
\ \ 
\textrm{ and }
\ \ 
g_0 =\begin{pmatrix}
1 & 0 &0\\
0& \sin^2 r & 0\\
0&0& 1\\
\end{pmatrix}
\end{equation}
Moreover,
\begin{equation}
d\mu_g=\frac{u}{\varphi}\,dr\,d\theta\,d\xi,
\qquad
d\mu_0=\sin r\,dr\,d\theta\,d\xi,
\end{equation}
and hence
\begin{equation}
\frac{d\mu_g}{d\mu_0}
=
\frac{u}{\varphi\sin r}.
\end{equation}

\begin{lemma}\label{lem: Gamma-components}
For the metric in \eqref{eqn: metric-NNSC-distr} and the backgound metric in \eqref{eqn: background-metric}, the nonzero components of $\Gamma^k_{ij}$ defined in (\ref{eqn: Lee-LeFloch-Gamma}) are
\begin{equation}
\Gamma^r_{rr} = -\frac{\varphi_r}{\varphi},
\qquad
\Gamma^r_{\theta\theta}
=
\sin r\cos r-u u'
+u^2\frac{\varphi_r}{\varphi},
\qquad
\Gamma^r_{\xi\xi}
=
-\varphi^3\varphi_r,
\end{equation}
\begin{equation}
\Gamma^r_{r\theta} = \Gamma^r_{\theta r}
=
-\frac{\varphi_\theta}{\varphi},
\end{equation}
\begin{equation}
\Gamma^\theta_{rr}
=
\frac{\varphi_\theta}{u^2\varphi},
\qquad
\Gamma^\theta_{\theta\theta}
=
-\frac{\varphi_\theta}{\varphi},
\qquad
\Gamma^\theta_{\xi\xi}
=
-\frac{\varphi^3}{u^2}\varphi_\theta,
\end{equation}
\begin{equation}
\Gamma^\theta_{r\theta} = \Gamma^\theta_{\theta r}
=
\frac{u'}{u}
-
\frac{\varphi_r}{\varphi}
-
\frac{\cos r}{\sin r},
\end{equation}
and
\begin{equation}
\Gamma^\xi_{r\xi}
=
\frac{\varphi_r}{\varphi},
\qquad
\Gamma^\xi_{\theta\xi}
=
\frac{\varphi_\theta}{\varphi}.
\end{equation}
\end{lemma}

\begin{proof}
The formulas follow by substituting
\begin{equation}
g_{rr}=\varphi^{-2},
\qquad
g_{\theta\theta}=\varphi^{-2}u^2,
\qquad
g_{\xi\xi}=\varphi^2
\end{equation}
into
\begin{equation}
\Gamma^k_{ij}
=
\frac12 g^{k\ell}
\left(
\bar\nabla_i g_{j\ell}
+
\bar\nabla_j g_{i\ell}
-
\bar\nabla_\ell g_{ij}
\right).
\end{equation}
For example,
\begin{equation}
\begin{aligned}
\Gamma^r_{\theta\theta}
&=
\frac12 g^{rr}
\left(
2\bar\nabla_\theta g_{r\theta}
-
\bar\nabla_r g_{\theta\theta}
\right)  \\
&=
\frac12\varphi^2
\left[
2\sin r\cos r\,\varphi^{-2}
-
\partial_r\left(\frac{u^2}{\varphi^2}\right)
\right]  \\
&=
\sin r\cos r-u u'
+u^2\frac{\varphi_r}{\varphi}.
\end{aligned}
\end{equation}
The remaining components are obtained similarly.
\end{proof}

\begin{lemma}\label{lem:V-computation}
For the metric in \eqref{eqn: metric-NNSC-distr} and the backgound metric in \eqref{eqn: background-metric}, the nonzero components of the vector field $V$ defined in are given by
\begin{equation}
V^r
=
\frac{\sin r\cos r}{u^2}\varphi^2
+
\frac{\cos r}{\sin r}\varphi^2
-
2\frac{u'}{u}\varphi^2,
\qquad
V^\theta=0,
\qquad
V^\xi=0.
\end{equation}
\end{lemma}

\begin{proof}
Using Lemma \ref{lem: Gamma-components}, we compute
\begin{equation}
\begin{aligned}
V^r
&=
g^{ij}\Gamma^r_{ij}-g^{rr}\Gamma^j_{jr} \\
&=
g^{\theta\theta}\Gamma^r_{\theta\theta}
+
g^{\xi\xi}\Gamma^r_{\xi\xi}
-
g^{rr}\Gamma^\theta_{\theta r}
-
g^{rr}\Gamma^\xi_{\xi r} \\
&=
\frac{\varphi^2}{u^2}
\left(
\sin r\cos r-u u'
+u^2\frac{\varphi_r}{\varphi}
\right)
-\varphi\varphi_r \\
&\quad
-\varphi^2
\left(
\frac{u'}{u}
-
\frac{\varphi_r}{\varphi}
-
\frac{\cos r}{\sin r}
\right)
-\varphi^2\frac{\varphi_r}{\varphi}.
\end{aligned}
\end{equation}
The terms involving \(\varphi_r\) cancel. Hence
\begin{equation}
V^r
=
\frac{\sin r\cos r}{u^2}\varphi^2
+
\frac{\cos r}{\sin r}\varphi^2
-
2\frac{u'}{u}\varphi^2.
\end{equation}

Similarly,
\begin{equation}
\begin{aligned}
V^\theta
&=
g^{ij}\Gamma^\theta_{ij}
-
g^{\theta\theta}\Gamma^j_{j\theta} \\
&=
g^{rr}\Gamma^\theta_{rr}
+
g^{\theta\theta}\Gamma^\theta_{\theta\theta}
+
g^{\xi\xi}\Gamma^\theta_{\xi\xi}
-
g^{\theta\theta}
\left(
\Gamma^r_{r\theta}
+
\Gamma^\theta_{\theta\theta}
+
\Gamma^\xi_{\xi\theta}
\right)
=0.
\end{aligned}
\end{equation}
Finally, \(V^\xi=0\), because all quantities are independent of \(\xi\) and
the only nonzero $\Gamma$-components with upper index \(\xi\)
have one lower index equal to \(\xi\).
\end{proof}

Since \(V\) has only an \(r\)-component, its divergence with respect to the
background metric \(g_0\) is
\begin{equation}
\overline{\operatorname{div}}V
=
\frac1{\sin r}\partial_r\left(\sin r\,V^r\right).
\end{equation}
The previous lemma therefore gives the following formula.

\begin{lemma}\label{lem:divV}
For the metric in \eqref{eqn: metric-NNSC-distr} and the backgound metric in \eqref{eqn: background-metric}, one has
\begin{equation}
\begin{aligned}
\overline{\operatorname{div}}V
=&
-2\varphi^2\frac{u''}{u}
+
2\cos r
\left(
\sin r\frac{\varphi^2}{u^2}
+
\frac{\varphi^2}{\sin r}
\right)
\left(
\frac{\varphi_r}{\varphi}
-
\frac{u'}{u}
\right)  \\
&+
(2\cos^2 r-\sin^2 r)\frac{\varphi^2}{u^2}
+
2\frac{(u')^2}{u^2}\varphi^2
-
4\frac{u'}{u}\varphi\varphi_r
-\varphi^2 .
\end{aligned}
\end{equation}
\end{lemma}

\begin{proof}
By Lemma \ref{lem:V-computation},
\begin{equation}
V^r
=
\varphi^2
\left(
\frac{\sin r\cos r}{u^2}
+
\frac{\cos r}{\sin r}
-
2\frac{u'}{u}
\right).
\end{equation}
Thus
\begin{equation}
\sin r\,V^r
=
\varphi^2
\left(
\frac{\sin^2 r\cos r}{u^2}
+
\cos r
-
2\sin r\frac{u'}{u}
\right).
\end{equation}
Taking \(\frac1{\sin r}\partial_r\) of the above expression gives the desired
identity.
\end{proof}

For smooth metrics, the Lee--LeFloch decomposition satisfies
\begin{equation}
\Scal_g=\overline{\operatorname{div}}V+F.
\end{equation}
By Lemma \ref{eqn: scalar-curvature}, the scalar curvature of a smooth metric $g$ as in \eqref{eqn: metric-NNSC-distr} is 
\begin{equation}
\Scal_g
=
-2\varphi^2\frac{u''}{u}
-
2|\nabla_h\varphi|_h^2.
\end{equation}
Combining this identity with Lemma \ref{lem:divV}, we have the following formula for $F$.
\begin{lemma}\label{lem: F-expression}
For the metric in \eqref{eqn: metric-NNSC-distr} and the backgound metric in \eqref{eqn: background-metric}, the function $F$ defined in \eqref{NNSC-dis-F} is given by
\begin{equation}
\begin{aligned}
F
=
&
-2|\nabla_h\varphi|_h^2
+\varphi^2 \\
&-
2\cos r
\left(
\sin r\frac{\varphi^2}{u^2}
+
\frac{\varphi^2}{\sin r}
\right)
\left(
\frac{\varphi_r}{\varphi}
-
\frac{u'}{u}
\right) \\
&-
(2\cos^2 r-\sin^2 r)\frac{\varphi^2}{u^2}
-
2\frac{(u')^2}{u^2}\varphi^2
+
4\frac{u'}{u}\varphi\varphi_r .
\end{aligned}
\end{equation}
\end{lemma}
\begin{remark}
{\rm
In the derivation above we used the identity $\Scal_g=\overline{\operatorname{div}}V+F$ for smooth metrics $g$.  However, the final expression for $F$ involves only $u$, $\varphi$, and their first derivatives.  Hence the same formula can be obtained directly from the Lee--LeFloch definition for metrics with first-order weak derivatives.  Since this direct verification is purely computational, we omit the details.
}
\end{remark}

\subsection{Nonnegative distributional scalar curvature on $M_\epsilon$}

We first prove the desired lower bound away from the two singular circles.  On
$M_\epsilon$, the functions $u_i$ and $\varphi_i$ have uniform two-sided bounds
and $\varphi_i$ has a uniform local $W^{1,2}$ estimate.  These estimates allow
us to pass to the limit in the Lee--LeFloch formula, with the only loss coming
from the weak lower semicontinuity of the quadratic gradient term.

\begin{proposition}\label{prop: NNSC-interior}
Let
\begin{equation}
g_\infty
=
\frac{1}{\varphi_\infty^2}
\left(dr^2+u_\infty^2d\theta^2\right)
+
\varphi_\infty^2d\xi^2
\end{equation}
be the limit metric on $\Sph^2 \times \Sph^1$, and
\begin{equation}
M_\epsilon := \Omega_\epsilon \times \Sph^1,
\end{equation} 
where 
\begin{equation}
    \Omega_\epsilon = \{(r, \theta) \mid \epsilon \leq r \leq 2-\epsilon, \ \ 0\leq \theta \leq 2\pi\} \subset \Sph^2.
\end{equation}
Then \(g_\infty\) has
nonnegative scalar curvature in the distributional sense on \(M_\epsilon\).
In other words, for every nonnegative \(v\in C_c^\infty(M_\epsilon)\),
\[
\left\langle R_{g_\infty},v\right\rangle\ge0.
\]
\end{proposition}

\begin{lemma}\label{lem: semi-continuity-interior}
For any $\epsilon >0$ and  any $v \in C^{\infty}(M_\epsilon)$, define
\begin{equation}
I_i(\epsilon, v)
:=
\int_{M_\epsilon}
\left[
- V_i\cdot \bar\nabla\left( v\frac{d\mu_{g_i}}{d\mu_0} \right)
+
F_i v \frac{d\mu_{g_i}}{d\mu_0}
\right]\,d\mu_0.
\end{equation}
where $V_i$ and $F_i$ are vector fields and functions defined in \eqref{NNSC-dis-V} and \eqref{NNSC-dis-F} for the metric $g_i$ with the background metric $g_0$.  Similarly, $I_\infty(\epsilon, v)$ is defined for the limit metric $g_\infty$. Then we have
\begin{equation}\label{eqn: semi-continuity-interior}
I_\infty (\epsilon, v)
\geq
\limsup_{i\to\infty} I_i(\epsilon, v).
\end{equation}
\end{lemma}
\begin{proof}
We first treat the $F_i$-term. By Lemma \ref{lem: F-expression},
\begin{equation}
\begin{aligned}
F_i
={}&
-2|\nabla_{h_i}\varphi_i|_{h_i}^2
+\varphi_i^2 \\
&-
2\cos r
\left(
\sin r\frac{\varphi_i^2}{u_i^2}
+
\frac{\varphi_i^2}{\sin r}
\right)
\left(
\frac{\partial_r\varphi_i}{\varphi_i}
-
\frac{u_i'}{u_i}
\right)  \\
&-
(2\cos^2r-\sin^2r)\frac{\varphi_i^2}{u_i^2}
-
2\frac{(u_i')^2}{u_i^2}\varphi_i^2
+
4\frac{u_i'}{u_i}\varphi_i\partial_r\varphi_i .
\end{aligned}
\end{equation}
The only quadratic term in \(\nabla\varphi_i\) is
$
-2|\nabla_{h_i}\varphi_i|_{h_i}^2.
$
For this term,
\begin{equation}
    \begin{aligned}
    &\quad \limsup_{i\to \infty} \int_{M_{\epsilon}}-|\nabla \varphi_i|^2\frac{u_i}{\varphi_i}vd\mu_0\\
    &=\limsup_{i \to \infty} \int_{M_{\epsilon}}-|\nabla \varphi_i|^2\frac{u_\infty}{\varphi_\infty}vd\mu_0
         +\limsup_{i \to \infty} \int_{M_{\epsilon}}-|\nabla \varphi_i|^2 \left( \frac{u_i}{\varphi_i}-\frac{u_\infty}{\varphi_\infty} \right)vd\mu_0\\
    &\leq  \int_{M_{\epsilon}}-|\nabla \varphi_\infty |^2\frac{u_\infty}{\varphi_\infty}vd\mu_0
    \end{aligned}
\end{equation}
where the inequality follows from the last line of (\ref{NNSC-convergence-local}), and the second term in the second line tends to zero by $\varphi_i\geq C(\Lambda,\epsilon)>0$ on $M_\epsilon$ [Proposition \ref{prop: varphi-C01/2-estimate}], pointwise convergence of $u_i$ and $\varphi_i$, together with the boundedness of $\int_{M_{\epsilon}}|\nabla \varphi_i|^2 $.
All terms in \(F_i\) that are linear in \(\partial_r\varphi_i\) converge by
the weak \(W^{1,2}\)-convergence of \(\varphi_i\), because their coefficients
converge uniformly on \(M_\epsilon\). The remaining zeroth-order terms
converge by dominated convergence. Therefore
\begin{equation}\label{eq:F-term-limsup}
\limsup_{i\to\infty}
\int_{M_\epsilon}
F_i v \frac{d\mu_{g_i}}{d\mu_0}\,d\mu_0
\le
\int_{M_\epsilon}
F_\infty v\frac{d\mu_{g_\infty}}{d\mu_0}\,d\mu_0.
\end{equation}

Next we consider the $V_i$-term. By Lemma \ref{lem:V-computation},
\begin{equation}
V_i^\theta=V_i^\xi=0,
\end{equation}
and
\begin{equation}
V_i^r
=
\frac{\sin r\cos r}{u_i^2}\varphi_i^2
+
\frac{\cos r}{\sin r}\varphi_i^2
-
2\frac{u_i'}{u_i}\varphi_i^2 .
\end{equation}
In particular,
\begin{equation}
V_i^r\to V_\infty^r \ \ \text{uniformly on} \,  M_\epsilon, \ \ \text{as} \ \ i \to \infty.
\end{equation}
Since $V_i $ has only an $ r$-component,
\begin{equation}
\begin{aligned}
V_i\cdot \bar\nabla \left( v\frac{d\mu_{g_i}}{d\mu_0} \right)
& =
V_i^r \partial_r \left( v\frac{d\mu_{g_i}}{d\mu_0} \right) \\
& =
V^r_i \frac{d\mu_{g_i}}{d\mu_0} \left[
\partial_r v
+
v\left(
\frac{u_i'}{u_i}
-
\frac{\partial_r\varphi_i}{\varphi_i}
-
\frac{\cos r}{\sin r}
\right)
\right]
\end{aligned}
\end{equation}
The coefficients of \(\partial_r\varphi_i\) converge uniformly, while
\(\partial_r\varphi_i\rightharpoonup\partial_r\varphi_\infty\) weakly in
\(L^2\). Therefore
\begin{equation}\label{eq:V-term-limit}
\lim_{i\to\infty}
\int_{M_\epsilon}
- V_i\cdot \bar\nabla \left(v \frac{d\mu_{g_i}}{d\mu_0}\right)\,d\mu_0
=
\int_{M_\epsilon}
- V_\infty\cdot \bar\nabla \left(v \frac{d\mu_{g_\infty}}{d\mu_0} \right)\,d\mu_0.
\end{equation}

Combining \eqref{eq:F-term-limsup} and \eqref{eq:V-term-limit}, we obtain \eqref{eqn: semi-continuity-interior}. This complete the proof of the lemma.
\end{proof}

\begin{proof}[Proof of Proposition \ref{prop: NNSC-interior}]
Since $g_i$ is smooth, the Lee--LeFloch distributional scalar curvature formula agrees
with the classical scalar curvature formula. Hence for any $0 \leq v \in C^\infty_0(M_\epsilon)$,
\begin{equation}
I_i(\epsilon, v)
=
\int_{M_\epsilon} \Scal_{g_i}v\,d\mu_{g_i} \geq 0,
\end{equation}
since $\Scal_{g_i}\ge0$ and $v\ge0$. 
Then by Lemma \ref{lem: semi-continuity-interior},
\begin{equation}
I_\infty(\epsilon, v)
\ge
\limsup_{i\to\infty} I_i(\epsilon, v) \geq 0.
\end{equation}
That is,
\begin{equation}
\int_{M_\epsilon}
\left[
- V_\infty\cdot \bar\nabla
\left(
v\frac{d\mu_{g_\infty}}{d\mu_S}
\right)
+
F_\infty v\frac{d\mu_{g_\infty}}{d\mu_0}
\right]\,d\mu_0
\ge0.
\end{equation}
Hence
\begin{equation}
\left\langle R_{g_\infty},v\right\rangle\ge0
\end{equation}
for every nonnegative \(v\in C_c^\infty(M_\epsilon)\). This proves that
\(g_\infty\) has nonnegative scalar curvature in the distributional sense on
\(M_\epsilon\).
\end{proof}

\subsection{Nonnegative distributional scalar curvature on $M$}

It remains to remove the restriction to $M_\epsilon$.  The only issue is the
boundary contribution arising from integration by parts on the truncated
manifold.  The endpoint estimates for $u_i$ and the global power bounds for
$\varphi_i$ imply that this boundary contribution tends to zero as
$\epsilon\to0$.

\begin{theorem}
The limit metric $g_\infty$ has nonnegative distributional scalar curvature on $M=\Sph^2 \times \Sph^1$ as in Definition \ref{def: Lee-LeFloch}.
\end{theorem}

\begin{proof}

Let $ v\in C^\infty(M)$ be nonnegative. Since
$\Scal_{g_i}\ge 0$, we have
\begin{equation}
0 \le \int_{\Omega_\epsilon}\Scal_{g_i}v\,d\mu_{0}.
\end{equation}
On $M_\epsilon$, the metric $g_i$ is smooth, and hence
\begin{equation}
\Scal_{g_i}
=
\overline{\operatorname{div}}V_i+F_i.
\end{equation}
Therefore, doing integration by parts gives
\begin{equation}\label{eqn: weak-boundary}
\begin{aligned}
0
&\le
\int_{M_\epsilon}\Scal_{g_i}v\,d\mu_{g_i}  \\
&=
\int_{M_\epsilon}
(\overline{\operatorname{div}}V_i+F_i)
v\frac{d\mu_{g_i}}{d\mu_0}\,d\mu_0  \\
&=
\int_{M_\epsilon}
\left[
- V_i\cdot
\bar\nabla
\left(
v\frac{u_i}{\varphi_i\sin r}
\right)
+
F_i v\frac{u_i}{\varphi_i\sin r}
\right]\,d\mu_0  \\
&\quad
+
\int_{\partial M_\epsilon}
\bar g(V_i,\bar n)
v\frac{u_i}{\varphi_i\sin r}\,dA_0 .
\end{aligned}
\end{equation}
Here \(\bar n\) denotes the outward unit normal with respect to the background
metric \(g_0\), and \(dA_0\) denotes the induced boundary measure.

The first integral on the right hand side of \eqref{eqn: weak-boundary} is $I_{i}(\epsilon, v)$ defined in Lemma \ref{lem: semi-continuity-interior}. For the boundary term, set
\begin{equation}
B_i(\epsilon,v)
:=
\int_{\partial M_\epsilon}
\bar g(V_i,\bar n)
v\frac{u_i}{\varphi_i\sin r}\,dA_0 .
\end{equation}
Then \eqref{eqn: weak-boundary} gives
\begin{equation}
0\le I_i(\epsilon,v)+B_i(\epsilon,v),
\end{equation}
and hence
\begin{equation}\label{eqn: I-B-inequality}
I_i(\epsilon,v)\ge -B_i(\epsilon,v).
\end{equation}

We now estimate the boundary term. It is enough to treat the boundary
component
\begin{equation}
\Sigma_\epsilon^- :=\{r=\epsilon\};
\end{equation}
the other component \(\Sigma_\epsilon^+=\{r=2 -\epsilon\}\) is treated in the
same way, using the distance to the endpoint.

On \(\Sigma_\epsilon^-\), the outward unit normal is
$
\bar n=-\partial_r.
$
, and
$
dA_0=\sin\epsilon\,d\theta\,d\xi.
$
Thus
\begin{equation}\label{boundary-left-start}
\begin{aligned}
B_i^-(\epsilon,v)
&:=
\int_{\Sigma_\epsilon^-}
\bar g(V_i,\bar n)
v\frac{u_i}{\varphi_i\sin r}\,dA_0  \\
&=
-\int_{S^1\times S^1}
V_i^r(\epsilon,\theta)
\frac{u_i(\epsilon)}{\varphi_i(\epsilon,\theta)}
v(\epsilon,\theta,\xi)\,d\theta\,d\xi.
\end{aligned}
\end{equation}
By using Lemma \ref{lem:V-computation}, we get

\begin{equation}\label{eqn: boundary-left-formula}
\begin{aligned}
B_i^-(\epsilon,v)
&=
u_i(\epsilon)
\left(
2\frac{u_i'(\epsilon)}{u_i(\epsilon)}
-
\frac{\sin\epsilon\cos\epsilon}{u_i(\epsilon)^2}
-
\frac{\cos\epsilon}{\sin\epsilon}
\right)  \\
&\quad \times
\int_{S^1\times S^1}
\varphi_i(\epsilon,\theta)
v(\epsilon,\theta,\xi)\,d\theta\,d\xi .
\end{aligned}
\end{equation}

Since \(u_i\) is concave near the endpoint and satisfies
\begin{equation}
u_i(0)=0,
\qquad
u_i'(0)=1,
\end{equation}
we have
\begin{equation}
u_i'(\epsilon)\le \frac{u_i(\epsilon)}{\epsilon},
\qquad
u_i(\epsilon)\le \epsilon.
\end{equation}
Using \(v\ge0\), we obtain from \eqref{eqn: boundary-left-formula}
\begin{equation}
\begin{aligned}
B_i^-(\epsilon,v)
&\le
u_i(\epsilon)
\left(
\frac{2}{\epsilon}
-
\frac{\sin\epsilon\cos\epsilon}{\epsilon^2}
-
\frac{\cos\epsilon}{\sin\epsilon}
\right)  \\
&\quad \times
\int_{S^1\times S^1}
\varphi_i(\epsilon,\theta)
v(\epsilon,\theta,\xi)\,d\theta\,d\xi .
\end{aligned}
\end{equation}
As $\epsilon\to0 $,
\begin{equation}
\frac{\sin\epsilon\cos\epsilon}{\epsilon^2}
=
\frac1\epsilon-\frac23\epsilon+O(\epsilon^3),
\qquad
\frac{\cos\epsilon}{\sin\epsilon}
=
\frac1\epsilon-\frac13\epsilon+O(\epsilon^3).
\end{equation}
Hence
\begin{equation}
\frac{2}{\epsilon}
-
\frac{\sin\epsilon\cos\epsilon}{\epsilon^2}
-
\frac{\cos\epsilon}{\sin\epsilon}
\le C\epsilon
\end{equation}
for all sufficiently small \(\epsilon>0\). Consequently,
\begin{equation}
B_i^-(\epsilon,v)
\le
C\epsilon u_i(\epsilon)
\int_{S^1\times S^1}
\varphi_i(\epsilon,\theta)
v(\epsilon,\theta,\xi)\,d\theta\,d\xi .
\end{equation}
Using \(u_i(\epsilon)\le\epsilon\), we get
\begin{equation}
B_i^-(\epsilon,v)
\le
C\epsilon^2\|v\|_{L^\infty}
\int_{S^1\times S^1}
\varphi_i(\epsilon,\theta)\,d\theta\,d\xi .
\end{equation}
In addition, by the estimate \eqref{eqn: estimate-varphi-final} with $k=2$, we have
\begin{equation}
\int_{S^1\times S^1}
\varphi_i(\epsilon,\theta)\,d\theta\,d\xi
\le
C(\Lambda)\epsilon^{-1/2}.
\end{equation}
Consequently,
\begin{equation}
B_i^-(\epsilon,v)
\le
C(\Lambda)\epsilon^{3/2}\|v\|_{L^\infty}.
\end{equation}

The same argument at the other endpoint gives
\begin{equation}
B_i^+(\epsilon,v)
\le
C(\Lambda)\epsilon^{3/2}\|v\|_{L^\infty}.
\end{equation}
Therefore
\begin{equation}\label{NNSC-estimate-boundary}
B_i(\epsilon,v)
=
B_i^-(\epsilon,v)+B_i^+(\epsilon,v)
\le
C(\Lambda)\epsilon^{3/2}\|v\|_{L^\infty}.
\end{equation}
Combining this with \eqref{eqn: I-B-inequality}, we obtain
\begin{equation}\label{eqn: interior-lower-bound}
I_i(\epsilon,v)
\ge
-C(\Lambda)\epsilon^{3/2}\|v\|_{L^\infty}.
\end{equation}
For every fixed
$ \epsilon>0$, by applying Lemma \ref{lem: semi-continuity-interior}, and taking the limit, this then implies 
\begin{equation}
I_\infty(\epsilon,v)
\ge
\limsup_{i\to\infty} I_i(\epsilon,v)
\ge
-C(\Lambda)\epsilon^{3/2}\|v\|_{L^\infty}.
\end{equation}
Finally, letting \(\epsilon\to0\), the domains \(\Omega_\epsilon\) exhaust
$M$, and the boundary error tends to
zero. Hence
\begin{equation}
\left\langle R_{g_\infty},v\right\rangle
\ge 0
\end{equation}
for every nonnegative \(v\in C^\infty(M)\). Therefore \(g_\infty\) has
nonnegative scalar curvature in the distributional sense on all of \(M\).
\end{proof}

\section{A $C^{1,\alpha}$ example and the volume-limit test}\label{sect: c1alpha-example}

In this section we construct a $C^{1,\alpha}$ Riemannian metric, in the form of \eqref{metric-form}, with nonnegative scalar curvature on its smooth part, and use it to illustrate a subtle point concerning volume-ratio formulations of scalar curvature lower bounds. The guiding observation is that, in the smooth setting, the small-ball volume expansion detects only the scalar curvature because the Euclidean ball averages out the off-diagonal Ricci terms.  In \S \ref{subsect: NNSC-volume-ratio}, we first recall this mechanism in the smooth case and the generalize the volume-limit formulation of NNSC from geodesic ball to more general regions. In \S \ref{subsect: example-construction}, we then construct a $C^{1,\alpha}$ metric on $\Sph^2 \times \Sph^1$, analyze its regularity and scalar curvature, discuss the behavior of geodesics and the exponential map near the singular circle, and finally show that it does not satisfy the generalized volume-limit NNSC condition. In \S \ref{subsect: remark}, we explain why the geodesic ball volume limit of the example keeps unchaged.


\subsection{Nonnegative scalar curvature condition in volume-limit sense}\label{subsect: NNSC-volume-ratio}

This subsection recalls the smooth volume expansion behind the volume-limit formulation of nonnegative scalar curvature, and then explains the cancellation mechanism that motivates the example below.  In the smooth setting, the symmetry of geodesic balls eliminates all off-diagonal Ricci terms and leaves only the scalar curvature.  The construction in the next subsection is designed to keep the scalar curvature nonnegative while arranging, on carefully chosen regions, that an off-diagonal Ricci term contributes to the fifth-order volume coefficient.

Let $(M^3,g)$ be a smooth Riemannian manifold.  Choose normal coordinates $\{x_i\}$ centered at a point $p$, and write $|x|=r$.  The volume density has the standard expansion (see, e.g. \cite{gray1974volume})
\begin{equation}\label{det-expansion}
    \sqrt{\det g_{ij}(x)}
    =1-\frac{1}{6}\sum_{i,j=1}^3\Ric_{ij}(p)x_i x_j+O(r^3).
\end{equation}
Integrating this expansion over the Euclidean ball in $T_pM$ gives the usual small geodesic ball volume expansion
\begin{equation} \label{vol-expansion}
    \Vol_g(B_g(p,t))
    =\frac{4\pi}{3}t^3\left(1-\frac{\Scal_g(p)}{30}t^2+O(t^3)\right),
\end{equation}
where $B_{g}(p, t)$ is the geodesic ball centered at $p$ with radius $t$ in $M$.
In fact, by further odd-parity cancellations the order estimate of the remainder can be improved, but the weaker form above is sufficient for our purposes. This characterization of the scalar curvature of smooth metrics can be naturally generalized to the more general metric measure spaces.
\begin{definition}[Infinitesimal volume-limit NNSC condition]\label{defn: volume-limit-NNSC}
Let $(X,d,\mu)$ be a metric measure space which is three-dimensional in the
sense that
\begin{equation}
\lim_{t\to0}\frac{\mu(B(p,t))}{\frac{4\pi}{3}t^3}=1
\end{equation}
at the points under consideration. We say that $X$ satisfies the infinitesimal
nonnegative scalar curvature (NNSC) condition in the volume-limit sense at $p$ if
\begin{equation}\label{def-NNSC-volume}
\limsup_{t\to0}
\frac{1}{t^5}
\left(
\mu(B(p,t))-\frac{4\pi}{3}t^3
\right)
\le 0.
\end{equation}
If this holds for every $p\in X$, we say that $X$ satisfies the condition
globally.
\end{definition}
For a smooth Riemannian metric, this condition is exactly consistent with $R(p)\ge0$, because \eqref{vol-expansion} gives
\[
\frac{1}{t^5}\left(\Vol_g(B_g(p,t))-\frac{4\pi}{3}t^3\right)
\longrightarrow
-\frac{2\pi}{45}\Scal_g(p).
\]
The reason only the scalar curvature appears in \eqref{vol-expansion}, rather than the full Ricci tensor in \eqref{det-expansion}, is the symmetry of balls.  Namely, if $\Omega=B_E(0,1)\subset T_pM$, then
\begin{equation}\label{symmetry}
    \int_{\Omega}x_i^2= \int_{\Omega}x_j^2 \quad\forall i,j,
    \qquad
    \int_{\Omega}x_i x_j=0 \quad i\neq j.
\end{equation}
The passage from \eqref{det-expansion} to \eqref{vol-expansion} uses this integral symmetry to cancel the off-diagonal terms $\Ric_{ij}$, $i\neq j$.

More generally, suppose $\Omega\subset\mathbb R^3$ is a bounded star-shaped region containing the origin and satisfying the symmetry condition \eqref{symmetry}.  For $t>0$, set
\begin{equation}
    t\Omega:=\{tx:x\in\Omega\}.
\end{equation}
Then the same computation gives
\begin{equation}  \label{def-NNSC-volume-new}
        \lim_{t\rightarrow 0}\frac{1}{t^5}\left(\Vol_g(\exp_p(t\Omega))-\Vol_{g_E}(\Omega)t^3\right)
        =-C(\Omega)\Scal_g(p),
\end{equation}
where the constant $C(\Omega)>0$ depends only on $\Omega$.  Thus, for smooth metrics, the sign of \eqref{def-NNSC-volume-new} is still governed solely by scalar curvature whenever the tangent region has the symmetry \eqref{symmetry}.

The example below is based on the following observation.  The determinant expansion \eqref{det-expansion} contains the full Ricci tensor.    If one can preserve nonnegative scalar curvature while arranging that the relevant limiting region sees an off-diagonal Ricci term, then the fifth-order volume coefficient may no longer be controlled only by the scalar curvature.

We now record the particular off-diagonal term used in the construction.  Consider the local warped product model \eqref{metric-form}, written in coordinates as
\begin{equation}
    g=\varphi^{-2}\left(dr^2+u(r)^2d\theta^2\right)+\varphi^2d\xi^2.
\end{equation}
A direct computation gives
\begin{equation}
\Ric_g(\partial_r, \xi) = \Ric_g(\partial_\theta, \partial_\xi) =0,
\end{equation}
and
\begin{eqnarray}
\Ric_g \left(\partial_r, \partial_\theta\right) 
& = &  \frac{1}{\varphi^2}{\rm Rm}_g \left(\partial_r, \partial_{\xi}, \partial_{\xi}, \partial_\theta\right) \\
& = &  \frac{1}{\varphi^2}\left\langle\nabla_{\partial_r} \nabla_{\partial_\xi} \partial_{\xi}-\nabla_{\partial_\xi} \nabla_{\partial_r} \partial_\xi, \partial_\theta\right\rangle_g \\
& = & \frac{1}{\varphi^2}\left\langle\nabla_{\partial_r}\left(-\varphi^3 \varphi_r \partial_r-\frac{\varphi^3 \varphi_\theta}{u^2} \partial_\theta\right)-\nabla_{\partial_\xi}\left(\frac{\varphi_r}{\varphi} \partial_\xi\right), \partial_\theta\right\rangle_g \\
&= & \frac{1}{\varphi^2}\left( -\varphi^3 \varphi_r\left\langle\nabla_{\partial_r} \partial_r, \partial_\theta\right\rangle_g 
-\partial_r\left(\frac{\varphi^3 \varphi_\theta}{u^2}\right) \frac{u^2}{\varphi^2} \right.\\
&  & \qquad \left. -\frac{\varphi^3 \varphi_\theta}{u^2} \langle\nabla_{\partial_r} \partial_\theta, \partial_\theta\rangle 
-\frac{\varphi_r}{\varphi}\left\langle\nabla_{\partial_\xi} \partial_\xi, \partial_\theta\right\rangle_g \right) \\
& = &  \frac{1}{\varphi^2} \left( -\varphi_r \varphi_\theta-3 \varphi_r \varphi_\theta-\varphi \varphi_{r \theta}+2 \frac{u^\prime \varphi \varphi_\theta}{u} \right. \\
&  & \qquad \left.-\varphi \varphi_\theta\left(\frac{u_r}{u} -\frac{\varphi_r}{\varphi}\right)+\varphi_r \varphi_\theta \right) \\
& = & - \frac{\varphi_{r \theta}}{\varphi}-2 \frac{\varphi_r \varphi_\theta}{\varphi^2}
+\frac{\varphi_\theta}{\varphi} \frac{u^\prime}{u}\label{Ric-rtheta}
\end{eqnarray}
Here the only second derivative term is $\varphi_{r\theta}$, which controls the mixed Ricci component near the center of the construction, where the first derivatives tend to zero.  Taking $\varphi=1+kr\theta$ in the quadrant $r>0$, $\theta>0$, and reflecting across the two coordinate axes, produces a model in which the mixed component $\Ric_{r\theta}$ survives on the selected regions.  The rest of the section makes this idea precise in a $C^{1,\alpha}$ setting. In particular, one can take $k$ sufficiently large and choose a suitable region $\Omega$ satisfying the symmetric condition \eqref{symmetry} so that the metric has negative scalar curvature at the singular point in the sense \eqref{def-NNSC-volume-new}; see Lemma \ref{lem: volume-expansion-wedge-g-alpha-k}. On the other hand, it still satisfies NNSC condition in Definition \ref{defn: volume-limit-NNSC}; see Section \ref{subsect: remark}.


\subsection{A $C^{1, \alpha}$ example}\label{subsect: example-construction}

We now implement the preceding idea by constructing a local $C^{1,\alpha}$ warping function on the $(r,\theta)$-plane.  The function is equal to $|r\theta|$ on the pure regions, vanishes along the coordinate axes, and interpolates through thin transition regions.  The estimates below ensure that, after a small cutoff, the resulting warped product metric is globally defined in Example \ref{example: C1alpha-metric}, smooth away from a circle, and has uniformly positive scalar curvature on its smooth locus.

The first lemma records the regularity and derivative estimates for the local profile $\psi_\alpha$.  These estimates are the analytic input that allows the metric to be only $C^{1,\alpha}$ at the singular circle while remaining smooth elsewhere.

\begin{lemma}\label{lem: derivative-bounds-for-psi-alpha}
Let $0<\alpha<1$, and let $b\in C^\infty(\mathbb R)$ satisfy
\begin{equation}
b(x) =
\begin{cases}
0,  & x\le \frac12, \\
1,  & x\ge 1,
\end{cases}
\end{equation}
and
\begin{equation}
0\le b\le 1,\quad |b'|\le 10,\quad |b''|\le 100 \quad \text{on} \ \ \mathbb{R}.
\end{equation}
Define $\psi_\alpha:[-1,1]^2\to \mathbb R$ by
\begin{equation}
\psi_{\alpha}(r,\theta)=
\begin{cases}
 |r\theta|\,
b\!\left(\dfrac{|\theta|}{ |r|^{1/\alpha}}\right)
b\!\left(\dfrac{|r|}{ |\theta|^{1/\alpha}}\right),
& r\theta\neq 0,\\[3mm]
0,& r\theta=0.
\end{cases}
\end{equation}
Then $\psi_\alpha$ is smooth on $[-1,1]^2\setminus\{(0,0)\}$. Moreover,
there exists a constant $C=C(\alpha)>0$ such that, for all
$(r,\theta)\neq (0,0)$ sufficiently close to $(0,0)$, the following estimates
hold.

\begin{enumerate}
\item The first derivatives satisfy
\begin{equation}
|\partial_r\psi_\alpha(r,\theta)|
\le \left(11 + \frac{10}{\alpha} \right) |\theta|,
\qquad
|\partial_\theta\psi_\alpha(r,\theta)|
\le \left( 11 + \frac{10}{\alpha} \right) |r|.
\end{equation}
and
\begin{equation}\label{eqn: psi-first-derivative-estimate2}
|\partial_r\psi_\alpha(r,\theta)|
\le \left(11 + \frac{10}{\alpha} \right) |r|^\alpha,
\qquad
|\partial_\theta\psi_\alpha(r,\theta)|
\le \left( 11 + \frac{10}{\alpha} \right) |\theta|^\alpha .
\end{equation}

\item Away from the coordinate axes, the second derivatives satisfy
\begin{equation}\label{eqn: psi-second-derivative-estimate}
|\partial_{rr}\psi_\alpha(r,\theta)|
\le C |r|^{\alpha-1},
\quad
|\partial_{r\theta}\psi_\alpha(r,\theta)|
\le C,
\quad
|\partial_{\theta\theta}\psi_\alpha(r,\theta)|
\le C |\theta|^{\alpha-1}.
\end{equation}
On the coordinate axes away from the origin, $\psi_\alpha\equiv 0$ in a
neighborhood, and hence all derivatives vanish there.

\item After defining
\begin{equation}
\partial_r\psi_\alpha(0,0)
=
\partial_\theta\psi_\alpha(0,0)
=
0,
\end{equation}
one has $\psi_\alpha\in C^{1,\alpha}$ near $(0,0)$. More precisely,
\begin{equation}
|\nabla\psi_\alpha(r_1,\theta_1)
-
\nabla\psi_\alpha(r_2,\theta_2)|
\le
C |(r_1,\theta_1)-(r_2,\theta_2)|^\alpha
\end{equation}
for all $(r_1,\theta_1)$ and $(r_2,\theta_2)$ sufficiently close to $(0,0)$.
\end{enumerate}
\end{lemma}

\begin{figure}[H]
  \centering
  \includegraphics[width=0.6\textwidth]{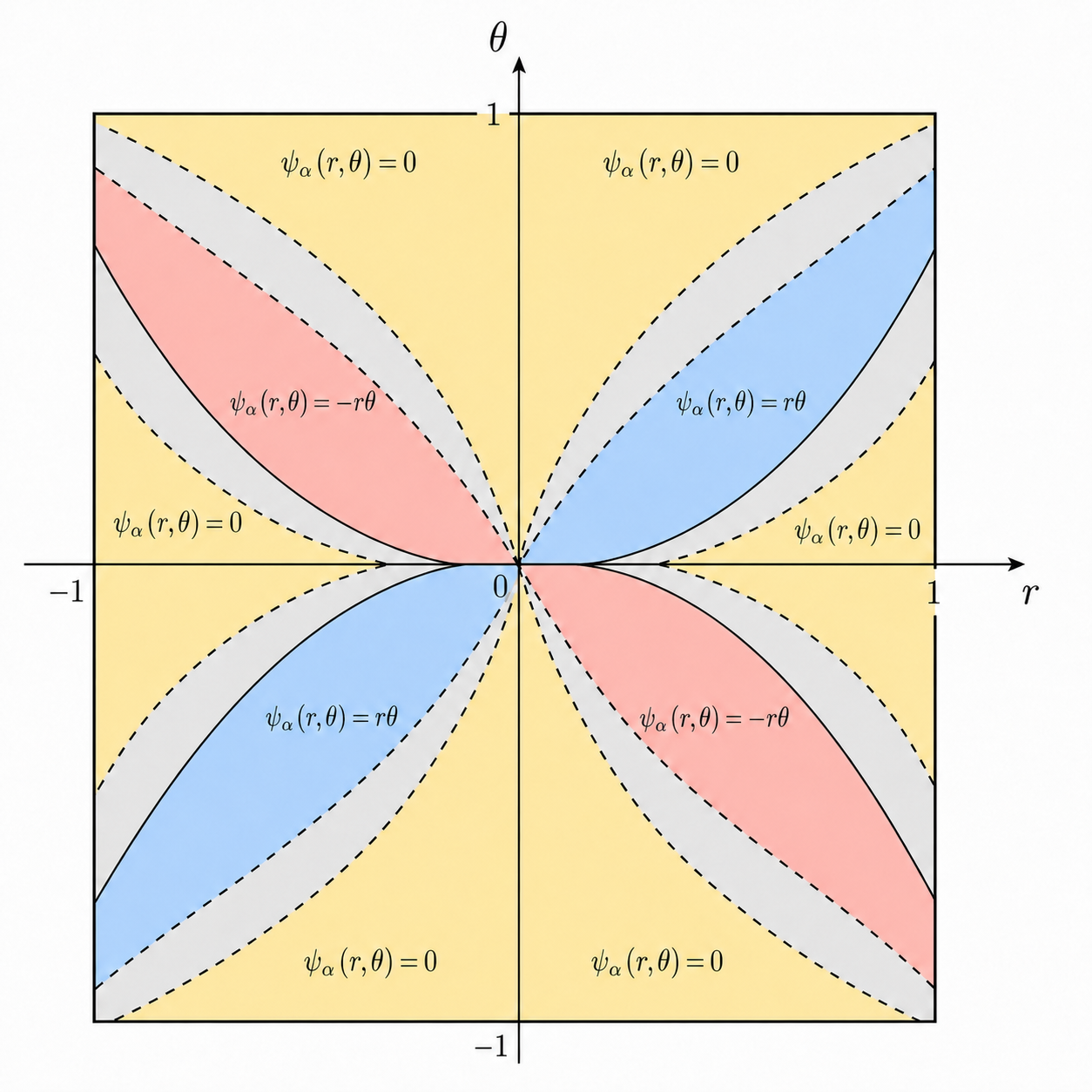}
   \caption{Schematic picture of the regions associated with $\psi_{\alpha}(r,\theta)$. The blue regions correspond to $\psi_{\alpha}(r,\theta)=r\theta$, the red regions correspond to $\psi_{\alpha}(r,\theta)=-r\theta$, and the yellow regions correspond to $\psi_{\alpha}(r,\theta)=0$. The gray regions denote the transition set. The solid curves are given by $|\theta|=|r|^{\alpha}$ and $|\theta|=|r|^{1/\alpha}$, and the dashed curves are given by $|\theta|=(2|r|)^{\alpha}$ and $|\theta|=\frac12 |r|^{1/\alpha}$.}
\end{figure}

\begin{proof}
For $r\theta\neq 0$, set
\begin{equation}
A=A(r,\theta):=\frac{|\theta|}{|r|^{1/\alpha}},
\qquad
B=B(r,\theta):=\frac{|r|}{|\theta|^{1/\alpha}}.
\end{equation}
On each open quadrant, set
\begin{equation}
\sigma:=\operatorname{sgn}(r\theta),
\end{equation}
then
$
|r\theta|=\sigma r\theta,
$
and hence, on that quadrant,
\begin{equation}\label{eqn:psi-quadrant}
\psi_\alpha(r,\theta)=\sigma r\theta\, b(A)b(B).
\end{equation}

We first show that $\psi_\alpha$ is smooth away from $(0,0)$. It is clearly
smooth in each open quadrant. If $(0,\theta_0)$ is a point on the $\theta$-axis
with $\theta_0\neq0$, then for $|r|$ sufficiently small, one has
\begin{equation}
B=\frac{|r|}{|\theta|^{1/\alpha}}\le \frac12,
\end{equation}
so $b(B)=0$. Hence $\psi_\alpha\equiv0$ in a neighborhood of $(0,\theta_0)$.
Similarly, in a neighborhood of a point $(r_0,0)$ on the $r$-axis with $r_0\neq0$,  $\psi_\alpha\equiv0$.
Thus $\psi_\alpha$ is smooth on $[-1,1]^2\setminus\{(0,0)\}$.

We now prove the derivative estimates near $(0,0)$. We work on a fixed open
quadrant and use \eqref{eqn:psi-quadrant}. The estimates are independent of
the quadrant, since only absolute values enter.

\smallskip
\noindent
\emph{Case 1. In the zero region:}
\begin{equation}
\mathcal{Z} : = \{(r, \theta) \in [-1, 1]^2 \mid A \leq  \frac{1}{2} \ \ \text{or} \ \ B \leq \frac{1}{2} \},
\end{equation}
one of the cutoff factors in the definition of $\varphi_\alpha$ vanishes, and hence
\[
\psi_\alpha\equiv0.
\]
Therefore all first and second derivatives vanish in this region, and the
desired estimates are immediate.

\smallskip
\noindent
\emph{Case 2. In the pure region}
\begin{equation}
\mathcal{P} : = \{(r, \theta) \in [-1, 1]^2 \mid A \geq  1 \ \ \text{and} \ \ B \geq 1 \},
\end{equation}
both cutoff factors are equal to $1$, and hence
\begin{equation}
\psi_\alpha=|r\theta|=\sigma r\theta.
\end{equation}
Therefore
\begin{equation}\label{eqn: pure-derivatives-psi-alpha-abs}
\partial_r\psi_\alpha=\sigma\theta,
\quad
\partial_\theta\psi_\alpha=\sigma r,
\quad
\partial_{rr}\psi_\alpha=0,
\quad
\partial_{r\theta}\psi_\alpha=\sigma,
\quad
\partial_{\theta\theta}\psi_\alpha=0.
\end{equation}
The second derivative estimates follow immediately.

For the first derivative estimates, the condition $A\ge1$ gives
\begin{equation}
|\theta|\ge |r|^{1/\alpha},
\qquad\text{hence}\qquad
|r|\le |\theta|^\alpha.
\end{equation}
The condition $B\ge1$ gives
\begin{equation}
|r|\ge |\theta|^{1/\alpha},
\qquad\text{hence}\qquad
|\theta|\le |r|^\alpha.
\end{equation}
Using \eqref{eqn: pure-derivatives-psi-alpha-abs}, we get
\begin{equation}
|\partial_r\psi_\alpha|
=
|\theta|
\le |r|^\alpha,
\end{equation}
and
\begin{equation}
|\partial_\theta\psi_\alpha|
=
|r|
\le |\theta|^\alpha.
\end{equation}
Thus the desired first derivative estimates hold in the pure region.

\smallskip
\noindent
\emph{Case 3. In the transition region:}
\begin{equation}
\mathcal{T} : = \left\{(r, \theta) \in [-1, 1]^2 \mid A > \frac{1}{2}, \ \ B > \frac{1}{2}, \ \  \text{and} \ \ \min\{A, B\} < 1\right\},
\end{equation}
both cutoff factors are nonzero, and at least one of them is non-constant.
On each fixed quadrant,
\begin{equation}
\partial_r A=-\frac1\alpha\frac{A}{r},
\qquad
\partial_\theta A=\frac{A}{\theta},
\qquad
\partial_r B=\frac{B}{r},
\qquad
\partial_\theta B=-\frac1\alpha\frac{B}{\theta}.
\end{equation}
Differentiating
\begin{equation}
\psi_\alpha=\sigma r\theta b(A)b(B),
\end{equation}
we obtain
\begin{equation}
\partial_r\psi_\alpha=\sigma \theta Q_r,
\qquad
\partial_\theta\psi_\alpha=\sigma r Q_\theta,
\end{equation}
where
\begin{equation}
Q_r
:=
b(A)b(B)
-\frac1\alpha A b'(A)b(B)
+
B b(A)b'(B),
\end{equation}
and
\begin{equation}
Q_\theta
:=
b(A)b(B)
+
A b'(A)b(B)
-\frac1\alpha B b(A)b'(B).
\end{equation}
Since $0\le b\le1$ and $|b'|\le10$, and since $b'$ is supported where the
corresponding argument lies in the interval $\left(\frac{1}{2},1\right)$, 
we have
\begin{equation}
|Q_r| \leq 11 + \frac{10}{\alpha}, \ \ \text{and} \ \  |Q_\theta|\leq 11 + \frac{10}{\alpha}.
\end{equation}
Hence
\begin{equation}\label{eqn: first-derivative-transition-psi-alpha-abs}
|\partial_r\psi_\alpha|\le \left( 11 + \frac{10}{\alpha} \right)|\theta|,
\qquad
|\partial_\theta\psi_\alpha|\le \left( 11 + \frac{10}{\alpha} \right) |r|.
\end{equation}

The assumptions $A>1/2$ and $B>1/2$ imply
\begin{equation}\label{eqn: transition-region-theta-r-1}
|\theta|>\frac12 |r|^{1/\alpha}
\qquad\Longrightarrow\qquad
|r|<(2|\theta|)^\alpha,
\end{equation}
and
\begin{equation}\label{eqn: transition-region-theta-r-2}
|r|>\frac12 |\theta|^{1/\alpha}
\qquad\Longrightarrow\qquad
|\theta|<(2|r|)^\alpha.
\end{equation}
Combining these with \eqref{eqn: first-derivative-transition-psi-alpha-abs},
we get
\begin{equation}
|\partial_r\psi_\alpha|
\le \left( 11 + \frac{10}{\alpha} \right)|r|^\alpha,
\quad
|\partial_\theta\psi_\alpha|
\le \left( 11 + \frac{10}{\alpha} \right)|\theta|^\alpha.
\end{equation}
This proves the first derivative estimates in the transition region.

We next estimate the second derivatives. Differentiating once more gives
\begin{equation}
\partial_{rr}\psi_\alpha
=
\sigma\theta
\left(
-\frac{A}{\alpha r}\partial_A Q_r
+
\frac{B}{r}\partial_B Q_r
\right),
\end{equation}
\begin{equation}
\partial_{r\theta}\psi_\alpha
=
\sigma
\left(
Q_r
+
A\partial_A Q_r
-
\frac{B}{\alpha}\partial_B Q_r
\right),
\end{equation}
and
\begin{equation}
\partial_{\theta\theta}\psi_\alpha
=
\sigma r
\left(
\frac{A}{\theta}\partial_A Q_\theta
-
\frac{B}{\alpha\theta}\partial_B Q_\theta
\right).
\end{equation}
Here
\begin{equation}
\partial_A Q_r
=
\left(1-\frac1\alpha\right)b'(A)b(B)
-
\frac{A}{\alpha}b''(A)b(B)
+
B b'(A)b'(B),
\end{equation}
\begin{equation}
\partial_B Q_r
=
2b(A)b'(B)
-
\frac{A}{\alpha}b'(A)b'(B)
+
B b(A)b''(B),
\end{equation}
\begin{equation}
\partial_A Q_\theta
=
2b'(A)b(B)
+
A b''(A)b(B)
-
\frac{B}{\alpha}b'(A)b'(B),
\end{equation}
and
\begin{equation}
\partial_B Q_\theta
=
\left(1-\frac1\alpha\right)b(A)b'(B)
+
A b'(A)b'(B)
-
\frac{B}{\alpha}b(A)b''(B).
\end{equation}
Using $|b'|\le10$ and $|b''|\le100$, we obtain
\begin{equation}
|\partial_A Q_r|
+
|\partial_B Q_r|
+
|\partial_A Q_\theta|
+
|\partial_B Q_\theta|
\le C(\alpha).
\end{equation}
Therefore
\begin{equation}\label{eqn: second-derivative-transition-psi-alpha-abs}
|\partial_{rr}\psi_\alpha|
\le C(\alpha)\frac{|\theta|}{|r|},
\qquad
|\partial_{r\theta}\psi_\alpha|
\le C(\alpha),
\qquad
|\partial_{\theta\theta}\psi_\alpha|
\le C(\alpha)\frac{|r|}{|\theta|}.
\end{equation}
Using again
(\ref{eqn: transition-region-theta-r-1}) and (\ref{eqn: transition-region-theta-r-2})
in the transition region, we obtain
\begin{equation}
|\partial_{rr}\psi_\alpha|
\le C\frac{|r|^\alpha}{|r|}
=
C|r|^{\alpha-1},
\ \ \text{and} \ \
|\partial_{\theta\theta}\psi_\alpha|
\le C\frac{|\theta|^\alpha}{|\theta|}
=
C|\theta|^{\alpha-1}.
\end{equation}
Together with the mixed derivative estimate in
\eqref{eqn: second-derivative-transition-psi-alpha-abs}, this proves the
second derivative estimates.

Combining the three cases proves the statements in items $(1)$ and $(2)$.

It remains to prove $C^{1,\alpha}$ regularity at the origin. From the first
derivative estimates \eqref{eqn: psi-first-derivative-estimate2},
we get
\begin{equation}
\partial_r\psi_\alpha(r,\theta)\to0,
\qquad
\partial_\theta\psi_\alpha(r,\theta)\to0
\qquad
\text{as }(r,\theta)\to(0,0).
\end{equation}
Thus $\nabla\psi_\alpha$ extends continuously to $(0,0)$ by setting
$
\nabla\psi_\alpha(0,0)=(0, 0).
$

Finally, the second derivative estimates in \eqref{eqn: psi-second-derivative-estimate} imply that this extension
is $\alpha$-Hölder. Indeed, joining $(r_1, \theta_1)$ to $(r_2, \theta_2)$ first by a horizontal segment and then by a vertical segment, and using the second derivative estimate in \eqref{eqn: psi-second-derivative-estimate}, one obtains
\begin{equation}
|\partial_r \psi_\alpha (r_1, \theta_1)-\partial_r \psi_\alpha(r_1, \theta_2)|
\le
C\bigl(|r_1-r_2|^\alpha+|\theta_1-\theta_2|\bigr)
\end{equation}
and
\begin{equation}
|\partial_\theta \psi_\alpha(r_1, \theta_1)- \partial_\theta \psi_\alpha(r_2, \theta_2)|
\le
C\bigl(|\theta_1-\theta_2|^\alpha+|r_1-r_2|\bigr).
\end{equation}
Since $(r_1, \theta_1)$ and $(r_2, \theta_2)$ lie in a fixed small neighborhood of the origin, the linear terms are bounded by a constant multiple of $|(r_1, \theta_1) - (r_2, \theta_2)|^\alpha$. Therefore
\begin{equation}
|\nabla \psi_\alpha(r_1, \theta_1)-\nabla \psi_\alpha(r_2, \theta_2)|\le C |(r_1, \theta_1) - (r_2, \theta_2)|^\alpha.
\end{equation}

Therefore $\psi_\alpha\in C^{1,\alpha}$ near $(0,0)$, as claimed.
\end{proof}

The next lemma localizes the profile $\psi_\alpha$ in Lemma \ref{lem: derivative-bounds-for-psi-alpha} by a cutoff.  The point is that the size of the support gives a smallness factor in the gradient estimate, which will be used to keep the scalar curvature positive.

\begin{lemma}\label{lem: cutoff-estimates}
Let $\chi\in C_c^\infty(\mathbb R^2)$ satisfy
\begin{equation}
0\le \chi\le 1,\qquad
\chi\equiv 1 \ \text{on }[-1,1]^2,\qquad
\chi\equiv 0 \ \text{outside }[-2,2]^2.
\end{equation}
For $\rho>0$, define
\begin{equation}
\chi_\rho(r,\theta):=\chi\!\left(\frac r\rho,\frac\theta\rho\right).
\end{equation}
Assume $0<\rho\le \frac14$, and let
\begin{equation}
h:=dr^2+\cos^2 r\,d\theta^2.
\end{equation}
Then, for the function $\psi_\alpha$ defined in Lemma \ref{lem: derivative-bounds-for-psi-alpha}, on the support of $\chi_\rho$,
\begin{equation}
|\chi_\rho \varphi_\psi | \leq 4\rho^2.
\end{equation}
Moreover, there exists a constant $C_{\alpha,\chi}>0$, depending only on $\alpha$ and $\chi$, such that on the smooth locus of $\chi_\rho \psi_\alpha$,
\begin{equation}
|\nabla^h(\chi_\rho \psi_\alpha)|_h\le C_{\alpha,\chi}\rho.
\end{equation}
\end{lemma}

\begin{proof}
Since $\chi_\rho$ is supported in
\begin{equation}
\{|r|\le 2\rho,\ |\theta|\le 2\rho\},
\end{equation}
Lemma~\ref{lem: derivative-bounds-for-psi-alpha} gives
\begin{equation}
|\chi_\rho \psi_{\alpha}|
\le
|\psi_\alpha|
\le
|r\theta|
\le
4\rho^2.
\end{equation}

Next, by the assumption $\rho\leq \frac{1}{4}$, on the support of $\chi_\rho$ we have $|r|\le 2\rho\le \frac12$, hence
$
\sec |r|\le \sec\frac12.
$
Using Lemma~\ref{lem: derivative-bounds-for-psi-alpha}, we obtain
\begin{equation}
|\nabla^h\psi_\alpha|_h^2
=
|\partial_r\psi_\alpha|^2+\sec^2r\,|\partial_\theta\psi_\alpha|^2
\le
\Bigl(11+\frac{10}{\alpha}\Bigr)^2
\Bigl(\theta^2+\sec^2\!\Bigl(\frac12\Bigr)r^2\Bigr),
\end{equation}
and therefore, on the support of $\chi_\rho$,
\begin{equation}
|\nabla^h\psi_\alpha|_h
\le
A_\alpha \rho.
\end{equation}
where
\begin{equation}
A_\alpha
:=
2\left(11+\frac{10}{\alpha}\right)\sqrt{1+\sec^2 \left(\frac12\right)}.
\end{equation}

On the other hand, since $\chi_\rho(r,\theta)=\chi(r/\rho,\theta/\rho)$,
\begin{equation}
|\partial_r\chi_\rho|\le \frac{\|\partial_1\chi\|_{C^0}}{\rho},
\qquad
|\partial_\theta\chi_\rho|\le \frac{\|\partial_2\chi\|_{C^0}}{\rho}.
\end{equation}
Hence, on $|r|\le \frac12$, we have
\begin{equation}
|\nabla^h\chi_\rho|_h
\le
\frac{B_\chi}{\rho},
\end{equation}
where
\begin{equation}
B_\chi:=\sqrt{1+\sec^2(1/2)}\,\|\nabla\chi\|_{C^0(\mathbb R^2)}.
\end{equation}

Therefore
\begin{equation}
|\nabla^h(\chi_\rho \psi_\alpha)|_h
\le
|\nabla^h\psi_\alpha|_h+|\psi_\alpha| \cdot |\nabla^h\chi_\rho|_h
\le
A_\alpha\rho+4\rho^2\cdot \frac{B_\chi}{\rho} = (A_\alpha+4B_\chi)\rho.
\end{equation}
This completes the proof by setting
$
C_{\alpha,\chi}:=A_\alpha+4B_\chi.
$
\end{proof}

We now turn the localized profile into a global metric on $\Sph^2\times\Sph^1$.  The parameter $\rho_*$ is chosen small enough to ensure both positivity of the warping factor and a uniform scalar-curvature lower bound.

\begin{example}\label{example: C1alpha-metric}
Fix $0<\alpha<1$ and $k\in\mathbb R$. Let $p\in \mathbb S^2$, and choose local coordinates $(r,\theta)$ centered at $p$ such that
\begin{equation}
g_{\mathbb S^2}=dr^2+\cos^2 r\,d\theta^2
\end{equation}
in a neighborhood of $(0,0)$. Let
\begin{equation}
\Gamma:=\{p\}\times \mathbb S^1\subset \mathbb S^2\times \mathbb S^1.
\end{equation}
Set
\begin{equation}\label{eqn: rho*-definition}
\rho_*
:=
\min\!\left\{
\frac14,\,
\frac{1}{4(1+|k|)},\,
\frac{1}{4(1+|k|)\left(15+\frac{10}{\alpha}\right)\sqrt{1+\sec^2(1/2)}}
\right\}.
\end{equation}
Let
\begin{equation}
\varphi_{\alpha, k}
:=
1+k\,\chi_{\rho_*}\psi_\alpha,
\end{equation}
extended by $1$ outside the support of $\chi_{\rho_*}$, which is defined in Lemma \ref{lem: cutoff-estimates}, and where $\psi_\alpha$ is defined in Lemma \ref{lem: derivative-bounds-for-psi-alpha}. Then define the Riemannian metric
\begin{equation}
g_{\alpha,k}
:=
\varphi_{\alpha, k}^{-2}\,g_{\mathbb S^2}
+
\varphi_{\alpha, k}^{2}\,d\xi^2
\end{equation}
on $\Sph^2 \times \Sph^1$.
\end{example}

The following lemma summarizes the basic geometric properties of the metric constructed in Example \ref{example: C1alpha-metric}.

\begin{lemma}\label{lem: example-metric-properties}
The metrics $g_{\alpha, k}$ defined in Example \ref{example: C1alpha-metric} have the following properties:
\begin{enumerate}
\item \(g_{\alpha,k}\) is a well-defined \(C^{1,\alpha}\) Riemannian metric on \(\mathbb S^2\times \mathbb S^1\);
\item \(g_{\alpha,k}\) is smooth on \((\mathbb S^2\times \mathbb S^1)\setminus \Gamma\);
\item the scalar curvature satisfies
\begin{equation}
\Scal_{g_{\alpha,k}}\ge \frac58
\qquad
\text{on }(\mathbb S^2\times \mathbb S^1)\setminus \Gamma.
\end{equation}
\end{enumerate}
\end{lemma}
\begin{proof}
Since $\psi_\alpha\in C^{1,\alpha}$ near $(0,0)$ and $\chi_\rho$ is smooth, the product $\chi_\rho\psi_\alpha$ is \(C^{1,\alpha}\). Because \(\chi_\rho\) has compact support in the coordinate chart, $\varphi_{\alpha, k}$ extends globally as a \(C^{1,\alpha}\) function on \(\mathbb S^2\), equal to \(1\) outside a small neighborhood of \(p\). Hence
\begin{equation}
g_{\alpha,k}
=
\varphi_{\alpha, k}^{-2}g_{\mathbb S^2}
+
\varphi_{\alpha, k}^{2}d\xi^2
\end{equation}
is a \(C^{1,\alpha}\) metric provided $\varphi_{\alpha, k}>0$. Since $\psi_\alpha$ is smooth away from $ (0,0)$, the metric is smooth away from \(\Gamma=\{p\}\times \mathbb S^1\).

By Lemma \ref{lem: cutoff-estimates},
$
|\chi_{\rho_*} \psi_\alpha | \leq 4\rho^2_*.
$
Hence
\begin{equation}
\varphi_{\alpha, k}
=
1+k \chi_{\rho_*}\psi_\alpha
\ge
1- 4|k|\rho^2_*.
\end{equation}
Since $ \rho_* \leq  \frac{1}{4(1+|k|)} $, we have
\begin{equation}
4|k|\rho^2_*
\le
4(1+|k|)\rho^2_*
\le
\rho_*
\le
\frac14,
\end{equation}
and therefore
\begin{equation}
\varphi_{\alpha, k}\ge \frac34>0.
\end{equation}
Thus \(g_{\alpha,k,\rho}\) is a well-defined Riemannian metric.

Next, Lemma~\ref{lem: cutoff-estimates} yields
\begin{equation}
|\nabla^{g_{\mathbb S^2}} \varphi_{\alpha, k}|_{g_{\mathbb S^2}}
=
|k|\,|\nabla^h(\chi_{\rho_*} \psi_\alpha)|_h
\le
2|k|\left(15+\frac{10}{\alpha}\right)\sqrt{1+\sec^2(1/2)}\,\rho_*.
\end{equation}
Since
\begin{equation}
\rho_* \leq
\frac{1}{4(1+|k|)\left(15+\frac{10}{\alpha}\right)\sqrt{1+\sec^2(1/2)}},
\end{equation}
it follows that
\begin{equation}
|\nabla^{g_{\mathbb S^2}} \varphi_{\alpha, k}|_{g_{\mathbb S^2}}
\le
\frac{|k|}{2(1+|k|)}
\le
\frac12.
\end{equation}

Therefore, applying the scalar curvature formula in Lemma \ref{lem: scalar-curvature-formula}, we obtain on the smooth part
\begin{equation}
\Scal_{g_{\alpha,k}}
=
2 \varphi_{\alpha,k}^2
-
2|\nabla^{g_{\mathbb S^2}} \varphi_{\alpha, k}|_{g_{\mathbb S^2}}^2\ge
2\left( \frac34 \right)^2-2\left( \frac12 \right)^2
=
\frac58.
\end{equation}

Finally, outside the support of \(\chi_{\rho_*}\), one has $ \varphi_{\alpha, k}\equiv 1$, so
\begin{equation}
g_{\alpha,k}=g_{\mathbb S^2}+d\xi^2,
\end{equation}
and in particular $ R_{g_{\alpha,k}}=2 $ on that region.
\end{proof}

We next analyze geodesics emanating from a point on the singular circle.  Away from the two coordinate planes in the tangent space, such geodesics immediately enter a pure region, where the metric agrees with a smooth model; this gives uniqueness in those directions.

\begin{lemma}\label{lem: geodesics-from-origin}
Let $g_{\alpha,k}$ be the metric defined in
Example~\ref{example: C1alpha-metric}, and let
\begin{equation}
q=(0,0,0)
\end{equation}
in the local coordinates $(r,\theta,\xi)$ near $\Gamma$. After possibly
shrinking the coordinate neighborhood, assume that
\begin{equation}
|r|<\frac12\rho_*,
\qquad
|\theta|<\frac12\rho_* ,
\end{equation}
where $\rho_*$ is defined in (\ref{eqn: rho*-definition}).
Then the cutoff function $\chi_{\rho_*}\equiv 1$ in this neighborhood, and hence
\begin{equation}
g_{\alpha,k}
=
(1+k\psi_\alpha)^{-2}
\bigl(dr^2+\cos^2 r\,d\theta^2\bigr)
+
(1+k\psi_\alpha)^2\,d\xi^2 .
\end{equation}

Let
\begin{equation}
v=a\partial_r+b\partial_\theta+c\partial_\xi\in T_qM,
\end{equation}
and define
\begin{equation}
E_{r\xi}:=\operatorname{span}\{\partial_r,\partial_\xi\},
\qquad
E_{\theta\xi}:=\operatorname{span}\{\partial_\theta,\partial_\xi\}.
\end{equation}
Then the following hold.

\begin{enumerate}
\item For every $v\in T_qM$, there exists at least one local geodesic
\begin{equation}
\gamma_v:[0,\varepsilon)\to M
\end{equation}
such that
\begin{equation}
\gamma_v(0)=q,
\qquad
\dot\gamma_v(0)=v.
\end{equation}

\item If $a\neq 0$ and $b\neq 0$, then the geodesic with initial data
$(q,v)$ is unique. More precisely, if
\begin{equation}
\gamma_v(s)=(r(s),\theta(s),\xi(s)),
\end{equation}
then, for all sufficiently small $s>0$,
\begin{equation}
\operatorname{sgn}(r(s)\theta(s))=\operatorname{sgn}(ab),
\end{equation}
and $\gamma_v(s)$ lies in the corresponding pure region
\begin{equation}\label{eqn: pure-region-defn}
\mathcal{P}^\sigma
:=
\left\{
(r,\theta):
\sigma r\theta>0,\ 
\frac{|\theta|}{|r|^{1/\alpha}}>1,\ 
\frac{|r|}{|\theta|^{1/\alpha}}>1
\right\},
\qquad
\sigma=\operatorname{sgn}(ab).
\end{equation}
In this region,
\begin{equation}
\psi_\alpha(r,\theta)=|r\theta|=\sigma r\theta,
\end{equation}
and hence
\begin{equation}
g_{\alpha,k}
=
(1+\sigma k r\theta)^{-2}
\bigl(dr^2+\cos^2 r\,d\theta^2\bigr)
+
(1+\sigma k r\theta)^2\,d\xi^2 .
\end{equation}

\item Consequently, the local exponential map at $q$ is well-defined on
\begin{equation}
T_qM\setminus(E_{r\xi}\cup E_{\theta\xi})
=
\{a\partial_r+b\partial_\theta+c\partial_\xi:ab\neq 0\}.
\end{equation}
\end{enumerate}
\end{lemma}

\begin{proof}
By Lemma \ref{lem: example-metric-properties}, $g_{\alpha,k}$ is a $C^{1,\alpha}$ Riemannian metric near
$q$. Therefore its Christoffel symbols are continuous. As a result, the geodesic equation
is a second-order ODE with continuous coefficients, so local solutions exist
for every initial condition. This proves $(1)$.

We next prove uniqueness of geodesics for $ab\neq 0$. Let
\begin{equation}
\gamma_v(s)=(r(s),\theta(s),\xi(s))
\end{equation}
be any geodesic with
\begin{equation}
\gamma_v(0)=q,
\qquad
\dot\gamma_v(0)=a\partial_r+b\partial_\theta+c\partial_\xi,
\end{equation}
where $a\neq 0$ and $b\neq 0$. Since $\gamma_v$ is $C^2$, we have
\begin{equation}
r(s)=as+o(s),
\qquad
\theta(s)=bs+o(s),
\qquad
\xi(s)=cs+o(s)
\qquad
\text{as }s\to0^+.
\end{equation}
Thus, for all sufficiently small $s>0$,
\begin{equation}
\operatorname{sgn}(r(s)\theta(s))=\operatorname{sgn}(ab).
\end{equation}
Moreover,
\begin{equation}
\frac{|\theta(s)|}{|r(s)|^{1/\alpha}}
\sim
\frac{|b|}{|a|^{1/\alpha}}s^{1-1/\alpha}
\to +\infty,
\end{equation}
and
\begin{equation}
\frac{|r(s)|}{|\theta(s)|^{1/\alpha}}
\sim
\frac{|a|}{|b|^{1/\alpha}}s^{1-1/\alpha}
\to +\infty,
\end{equation}
because $0<\alpha<1$. Therefore, for sufficiently small
$\varepsilon>0$, $\gamma_v(s)$ lies in the pure region
$
\mathcal{P}^\sigma,
$
$
\sigma=\operatorname{sgn}(ab),
$
for all $0<s<\varepsilon$.

On this pure region,
\begin{equation}
\psi_\alpha=|r\theta|=\sigma r\theta.
\end{equation}
Thus $g_{\alpha,k}$ agrees there with the smooth metric
\begin{equation}
g_\sigma
=
(1+\sigma k r\theta)^{-2}
\bigl(dr^2+\cos^2 r\,d\theta^2\bigr)
+
(1+\sigma k r\theta)^2\,d\xi^2 .
\end{equation}
Hence $\gamma_v$ satisfies the smooth geodesic equation for $g_\sigma$ on
$(0,\varepsilon)$. Since $g_\sigma$ is smooth up to the point $q$, the usual uniqueness
theorem for smooth ODEs implies that there is only one such geodesic with
initial data $(q,v)$. This proves $(2)$.

Finally, $(3)$ follows directly from the uniqueness statement in $(2)$.
\end{proof}

\begin{remark}\label{rmrk: geodesic-non-uniqueness}
{\rm
Lemma \ref{lem: geodesics-from-origin} proves uniqueness of geodesics only for directions with $ab\neq 0$. Thus the only directions for which uniqueness is not settled by that argument are those in the exceptional set
\begin{equation}
E_{r\xi}\cup E_{\theta\xi}.
\end{equation}
For every $a,c\in\mathbb R$, since the restriction of the metric to the totally geodesic surface
$\{\theta=0\}$ is
\begin{equation}
dr^2+d\xi^2,
\end{equation}
the curve
\begin{equation}
\gamma_{a,c}^{(r\xi)}(s)=(as,0,cs)
\end{equation}
is a local geodesic with initial data
\begin{equation}
\gamma_{a,c}^{(r\xi)}(0)=q,
\qquad
\dot\gamma_{a,c}^{(r\xi)}(0)=a\partial_r+c\partial_\xi.
\end{equation}
Similarly, for every $b,c\in\mathbb R$, the curve
\begin{equation}
\gamma_{b,c}^{(\theta\xi)}(s)=(0,bs,cs)
\end{equation}
is a local geodesic with initial data
\begin{equation}
\gamma_{b,c}^{(\theta\xi)}(0)=q,
\qquad
\dot\gamma_{b,c}^{(\theta\xi)}(0)=b\partial_\theta+c\partial_\xi.
\end{equation}
 However, additional geodesics with these initial data are possble.

More precisely, if one assumes that a second geodesic branch enters a pure region, then in the $E_{r\xi}$-case, i.e. the initial tangent vector $v=a\partial_r + c \partial_\xi$, by using the geodesic equation $\ddot{\gamma}^i = -\Gamma^i_{j} \dot{\gamma}^j \dot{\gamma}^k$ and  $\Gamma^i_{jk} = O(\sqrt{r^2 + \theta^2})$, 
 one first derives
\begin{equation}
r(s)=as+O(s^3),\qquad \xi(s)=cs+O(s^3).
\end{equation}
Then the $\theta$-equation in the pure region yields
\begin{equation}
\begin{aligned}
\ddot{\theta}(s) 
&= - \Gamma_{ij}^{\theta} \dot{\gamma}^i \dot{\gamma}^j  \\
&= -(\Gamma_{r r}^{\theta} a^2 +2  \Gamma_{r \xi}^{\theta} ac + \Gamma_{\xi \xi}^{\theta} c^2)+O(s^2) \\
&= - \frac{1}{\sin ^2 (r+\pi/2)} \frac{\sigma k r}{1+ \sigma k r \theta} a^2  +  \frac{(1+ \sigma k r \theta)^3}{\sin ^2 (r+\pi/2)} \sigma k r c^2+O\left(s^2\right) \\
&=- ka \left(a^2 - c^2\right) s+O(s^2)\label{geo-1}
\end{aligned}
\end{equation}
Integrating twice and using $\theta(0) = \dot{\theta}(0) =0$ gives
\begin{equation}
\theta(s)
=
-\frac16\,\sigma k\,a(a^2-c^2)s^3+o(s^3).\label{geo-theta}
\end{equation}
Likewise, in the $E_{\theta\xi}$-case one derives
\begin{equation}
\theta(s)=bs+o(s^3),\qquad \xi(s)=cs+o(s^3),
\end{equation}
and then the $r$-equation in the pure region yields
\begin{equation}
r(s)
=
-\frac16\,\sigma k\,b(b^2-c^2)s^3+o(s^3).\label{geo-r}
\end{equation}

For such cubic branches, the pure-region conditions become
\begin{equation}
\frac{|\theta(s)|}{\pi |r(s)|^{1/\alpha}}
\sim C\,s^{\,3-1/\alpha},
\qquad
\frac{|r(s)|}{\pi |\theta(s)|^{1/\alpha}}
\sim C' s^{\,1-3/\alpha},
\end{equation}
or the analogous relations with $r$ and $\theta$ interchanged. These can both tend to $+\infty$ only when $\alpha\le \frac13$. Thus, for $\alpha>\frac13$, such cubic branching into a pure region cannot occur. However, this still does not prove uniqueness on $E_{r\xi}\cup E_{\theta\xi}$, since a hypothetical second branch could remain in the zero or transition regions.}
\end{remark}

For the volume computation, we now choose a wedge of tangent directions that avoids the exceptional planes and remains inside the pure regions under the exponential map for short time.

\begin{lemma}\label{lem: wedge-region-pure-for-g-alpha-k}
Let $0<\alpha<1$, and $\rho_0>0$ satisfy
\begin{equation}\label{eqn: rho0-small-new}
\rho_0
\le
\min\left\{
\frac12\rho_*,
\sqrt{1+\alpha^2}\,\alpha^{\frac{\alpha}{1-\alpha}}
\right\},
\end{equation}
where $\rho_*$ is defined in \eqref{eqn: rho*-definition}.
Take the subsets of $T_q M = \{x \partial_r + y \partial_\theta + z \partial_\xi \mid (x, y, z) \in \mathbb{R}^3\}$ as 
\[
\Omega_+
:=
\left\{
(x, y)\in\mathbb{R}^2:
\sqrt{x^2+ y^2}\le \rho_0,\ 
xy >0,\ 
\alpha |x|\le |y|\le \tfrac1\alpha |x|
\right\} \times \left( - \tfrac{\sqrt{3}}{2} \rho_0, \tfrac{\sqrt{3}}{2} \rho_0 \right),
\]
\[
\Omega_-
:=
\left\{
(x, y)\in\mathbb{R}^2:
\sqrt{x^2+ y^2}\le \rho_0,\ 
xy <0,\ 
\alpha |x|\le |y|\le \tfrac1\alpha |x|
\right\} \times \left( - \tfrac{\sqrt{3}}{2} \rho_0, \tfrac{\sqrt{3}}{2} \rho_0 \right),
\]
and set
\begin{equation}
\Omega:=\Omega_+\cup\Omega_-  \subset T_q M.
\end{equation}
Then $\Omega_\pm$ is contained in the corresponding pure cone in the
coordinate sense. More precisely, for every $(x,y,z)\in\Omega_+$,
\begin{equation}
xy>0,\qquad
\frac{|y|}{|x|^{1/\alpha}}\ge 1,
\qquad
\frac{|x|}{|y|^{1/\alpha}}\ge 1,
\end{equation}
and for every $(x,y,z)\in\Omega_-$,
\begin{equation}
xy<0,\qquad
\frac{|y|}{|x|^{1/\alpha}}\ge 1,
\qquad
\frac{|x|}{|y|^{1/\alpha}}\ge 1.
\end{equation}

Moreover, for sufficiently small $t>0$,
\begin{equation}
\exp_q^{g_{\alpha,k}}(t\Omega_+)\subset \mathcal P^+,
\qquad
\exp_q^{g_{\alpha,k}}(t\Omega_-)\subset \mathcal P^-,
\end{equation}
where
$
\mathcal P^+
$
and
$
\mathcal P^-
$
defined in \eqref{eqn: pure-region-defn}.
\end{lemma}

\begin{proof}
We first prove the tangent space pure region inclusion. Let
$(x,y,z)\in\Omega$. Then
\begin{equation}
\alpha |x|\le |y|\le \frac1\alpha |x|,
\qquad
x^2+y^2\le \rho_0^2.
\end{equation}
Hence
\begin{equation}
(1+\alpha^2)x^2\le x^2+y^2\le \rho_0^2,
\end{equation}
and similarly
\begin{equation}
(1+\alpha^2)y^2\le x^2+y^2\le \rho_0^2.
\end{equation}
Therefore
\begin{equation}
|x|\le \frac{\rho_0}{\sqrt{1+\alpha^2}},
\qquad
|y|\le \frac{\rho_0}{\sqrt{1+\alpha^2}}.
\end{equation}
By the assumption \eqref{eqn: rho0-small-new}, we get
\begin{equation}
|x|\le \alpha^{\frac{\alpha}{1-\alpha}},
\qquad
|y|\le \alpha^{\frac{\alpha}{1-\alpha}}.
\end{equation}
Since $0<\alpha<1$, this implies
\begin{equation}
|x|^{1/\alpha}\le \alpha |x|,
\qquad
|y|^{1/\alpha}\le \alpha |y|.
\end{equation}
Using again
\begin{equation}
\alpha |x|\le |y|,
\qquad
\alpha |y|\le |x|,
\end{equation}
we obtain
\begin{equation}
|x|^{1/\alpha}\le \alpha |x|\le |y|,
\qquad
|y|^{1/\alpha}\le \alpha |y|\le |x|.
\end{equation}
Hence
\begin{equation}
\frac{|y|}{|x|^{1/\alpha}}\ge1,
\qquad
\frac{|x|}{|y|^{1/\alpha}}\ge1.
\end{equation}
Thus the coordinate wedge lies in the pure region. The sign condition
distinguishes $\Omega_+$ from $\Omega_-$.

We now prove the statement for the exponential image. It is enough to prove
the result for $\Omega_+$. The proof for $\Omega_-$ is similar.

Let
\begin{equation}
v=x\partial_r+y\partial_\theta+z\partial_\xi\in\Omega_+.
\end{equation}
Then $xy>0$ and
\begin{equation}
\alpha |x|\le |y|\le \frac1\alpha |x|.
\end{equation}
Let
\begin{equation}
\gamma_v(s)=(r_v(s),\theta_v(s),\xi_v(s))
\end{equation}
be the geodesic with
\begin{equation}
\gamma_v(0)=q,
\qquad
\dot\gamma_v(0)=v.
\end{equation}
For $v\in\Omega_+$, the geodesic is unique for short time and
enters the positive pure region $\mathcal{P}^+$ by Lemma \ref{lem: geodesics-from-origin}. In the pure region the metric is
the smooth metric
\begin{equation}
g_+
=
(1+kr\theta)^{-2}
\left(dr^2+\cos^2 r\,d\theta^2\right)
+
(1+kr\theta)^2\,d\xi^2.
\end{equation}
Therefore the geodesic has the usual Taylor expansion
\begin{equation*}
r_v(s)=sx+O(s^3\rho_0^2 |x|),
\quad
\theta_v(s)=sy+O(s^3\rho_0^2 |y|),
\quad
\xi_v(s)=sz+O(s^3\rho_0^2 |z|),
\end{equation*}
uniformly for $v\in\Omega_+$ and $0\le s\le s_0$, after shrinking $s_0>0$
if necessary. In particular, for $s>0$ sufficiently small,
\begin{equation}
|r_v(s)|\sim s|x|,
\qquad
|\theta_v(s)|\sim s|y|,
\end{equation}
with constants independent of $v\in\Omega_+$. Hence
\begin{equation}
\operatorname{sgn}(r_v(s)\theta_v(s))
=
\operatorname{sgn}(xy)
=
+1.
\end{equation}

Now compute the two pure-region ratios. Since $|y|\ge \alpha |x|$, we have
\begin{equation}
\frac{|\theta_v(s)|}{|r_v(s)|^{1/\alpha}}
\ge
C^{-1}
s^{1-\frac1\alpha}
\frac{|y|}{|x|^{1/\alpha}}.
\end{equation}
Since $1-\frac1\alpha<0$, the factor
$
s^{1-\frac1\alpha}
$
tends to $+\infty$ as $s\to0^+$. Therefore, for all sufficiently small
$s>0$,
\begin{equation}
\frac{|\theta_v(s)|}{|r_v(s)|^{1/\alpha}}>1.
\end{equation}
Similarly,
\begin{equation}
\frac{|r_v(s)|}{|\theta_v(s)|^{1/\alpha}}>1
\end{equation}
for all sufficiently small $s>0$.

Taking $s=t$ gives
\begin{equation}
\exp_q^{g_{\alpha,k}}(t v)=\gamma_v(t)\in\mathcal P^+
\end{equation}
for all sufficiently small $t>0$, uniformly for $v\in\Omega_+$. Hence
\begin{equation}
\exp_q^{g_{\alpha,k}}(t\Omega_+)\subset \mathcal P^+.
\end{equation}
The proof for $\Omega_-$ is identical and gives
\begin{equation}
\exp_q^{g_{\alpha,k}}(t\Omega_-)\subset \mathcal P^-.
\end{equation}
\end{proof}

The wedge region $\Omega$ given in Lemma \ref{lem: wedge-region-pure-for-g-alpha-k} satisfies the symmetric condition in \eqref{symmetry}.

\begin{lemma}\label{lem: wedge-region-coordinate-symmetry}
Let $\Omega=\Omega_+\cup\Omega_-$ be the wedge region defined in
Lemma~\ref{lem: wedge-region-pure-for-g-alpha-k}. Set
\begin{equation}
\delta:=\arctan\frac1\alpha-\arctan\alpha.
\end{equation}
Then
\begin{equation}
\int_\Omega x^2\,dx\,dy\,dz
=
\int_\Omega y^2\,dx\,dy\,dz
=
\int_\Omega z^2\,dx\,dy\,dz
=
\frac{\sqrt3}{2}\,\delta\,\rho_0^5,
\end{equation}
and
\begin{equation}
\int_\Omega xy\,dx\,dy\,dz
=
\int_\Omega xz\,dx\,dy\,dz
=
\int_\Omega yz\,dx\,dy\,dz
=
0.
\end{equation}
\end{lemma}

The final lemma in this subsection computes the fifth-order volume coefficient on this wedge.  The mixed Ricci term contributes with the same sign on $\Omega_+$ and $\Omega_-$, and hence can dominate the scalar curvature contribution when $k$ is sufficiently large.

\begin{lemma}\label{lem: volume-expansion-wedge-g-alpha-k}
Let $g_{\alpha,k}$ be the metric defined in
Example~\ref{example: C1alpha-metric}. Let $\Omega=\Omega_+\cup\Omega_-$
be the wedge region defined in
Lemma~\ref{lem: wedge-region-pure-for-g-alpha-k}. Then, for all sufficiently
small $t>0$, we have
\begin{equation}\label{eqn:volume-expansion-wedge-g-alpha-k}
\begin{aligned}
&\Vol_{g_{\alpha,k}}\!\left(\exp_q^{g_{\alpha,k}}(t\Omega)\right)
-
t^3\Vol_{g_E}(\Omega)
\\
&\qquad =
\frac{\sqrt3\,\rho_0^5}{6}
\left[
k\frac{1-\alpha^2}{1+\alpha^2}
-
\left(\arctan\frac1\alpha-\arctan\alpha\right)
\right]t^5
+
O(t^6).
\end{aligned}
\end{equation}
Therefore the coefficient of the $t^5$ term is positive if and only if
\begin{equation}
k>K_*(\alpha),
\end{equation}
where
\begin{equation}
K_*(\alpha)
:=
\frac{1+\alpha^2}{1-\alpha^2}
\left(
\arctan\frac1\alpha-\arctan\alpha
\right).
\end{equation}
\end{lemma}

\begin{proof}
By Lemma~\ref{lem: wedge-region-pure-for-g-alpha-k}, for sufficiently small $t>0$, $\exp_q(\Omega_\pm) \subset \mathcal{P}^\pm$. On these pure regions $\mathcal{P}^\pm$, the metric agrees with the smooth metrics
\[
g_+
=
(1+kr\theta)^{-2}
\bigl(dr^2+\cos^2 r\,d\theta^2\bigr)
+
(1+kr\theta)^2\,d\xi^2
\]
and
\[
g_-
=
(1-kr\theta)^{-2}
\bigl(dr^2+\cos^2 r\,d\theta^2\bigr)
+
(1-kr\theta)^2\,d\xi^2.
\]
Therefore
\[
\exp_q^{g_{\alpha,k}}(t\Omega_+)
=
\exp_q^{g_+}(t\Omega_+),
\qquad
\exp_q^{g_{\alpha,k}}(t\Omega_-)
=
\exp_q^{g_-}(t\Omega_-)
\]
for sufficiently small $t>0$.

For a smooth metric $g_\pm$, the usual Jacobian expansion of the exponential
map gives
\begin{equation}\label{eqn:smooth-expansion-pm}
\Vol_{g_\pm}\bigl(\exp_q^{g_\pm}(t\Omega_\pm)\bigr)
=
t^3\Vol_{g_E}(\Omega_\pm)
-
\frac{t^5}{6}
\int_{\Omega_\pm}
\Ric^{(\pm)}_q(v)\,dv
+
O(t^6).
\end{equation}
Here $v=(x, y, z)$, and $dv=dx\,dy\,dz $ is the Euclidean volume measure on \(T_qM\).

We now compute the Ricci tensors of $g_\pm$ at $q$. At $q$, in the basis
$
\{ \partial_r,\partial_\theta,\partial_\xi \},
$
one has
\begin{equation}
\Ric_q^{(+)}
=
\begin{pmatrix}
1 & -k & 0\\
-k & 1 & 0\\
0 & 0 & 0
\end{pmatrix},
\qquad
\Ric_q^{(-)}
=
\begin{pmatrix}
1 & k & 0\\
k & 1 & 0\\
0 & 0 & 0
\end{pmatrix}.
\end{equation}
Therefore at the point $q$ along the direction $x\partial_r + y \partial_\theta + z \partial_\xi$
\begin{equation}
\Ric_q^{(+)}(x, y, z)
=
x^2+y^2-2kxy,
\end{equation}
and
\begin{equation}
\Ric_q^{(-)}(x, y, z)
=
x^2+y^2+2kxy.
\end{equation}

A direct calculation gives
\begin{equation}
\int_{\Omega_+}(x^2+ y^2)\,dx dy dz
=
\int_{\Omega_-}(x^2+ y^2)\, dx dy dz
=
\frac{\sqrt3}{2}\,\delta\,\rho_0^5,
\end{equation}
\begin{equation}\label{eqn:rtheta-integrals-wedge}
\int_{\Omega_+}xy\,dx dy dz
=
\frac{\sqrt3\,\rho_0^5}{4}
\frac{1-\alpha^2}{1+\alpha^2},
\end{equation}
and
\begin{equation}
\int_{\Omega_-} xy \,dx dy dz
=
-
\frac{\sqrt3\,\rho_0^5}{4}
\frac{1-\alpha^2}{1+\alpha^2}.
\end{equation}
Here $\delta:=\arctan\frac1\alpha-\arctan\alpha.$.

Hence
\begin{equation}
\begin{aligned}
\int_{\Omega_+}\Ric_q^{(+)}(v,v)\,dv
&=
\int_{\Omega_+}(x^2+ y^2-2kxy)\,dx dy dz
\\
&=
\frac{\sqrt3}{2}\,\delta\,\rho_0^5
-
2k\cdot
\frac{\sqrt3\,\rho_0^5}{4}
\frac{1-\alpha^2}{1+\alpha^2}
\\
&=
\frac{\sqrt3\,\rho_0^5}{2}
\left[
\delta
-
k\frac{1-\alpha^2}{1+\alpha^2}
\right].
\end{aligned}
\end{equation}
Similarly,
\begin{equation}
\begin{aligned}
\int_{\Omega_-}\Ric_q^{(-)}(v,v)\,dv
&=
\int_{\Omega_-}(x^2+ y^2+2kxy)\,dx dy dz
\\
&=
\frac{\sqrt3\,\rho_0^5}{2}
\left[
\delta
-
k\frac{1-\alpha^2}{1+\alpha^2}
\right].
\end{aligned}
\end{equation}
Therefore
\begin{equation}
\int_{\Omega_+}\Ric_q^{(+)}(v,v)\,dv
+
\int_{\Omega_-}\Ric_q^{(-)}(v,v)\,dv
\\
=
\sqrt3\,\rho_0^5
\left[
\delta
-
k\frac{1-\alpha^2}{1+\alpha^2}
\right].
\end{equation}

Summing the two expansions in \eqref{eqn:smooth-expansion-pm}, we obtain
\begin{equation}
\begin{aligned}
\Vol_{g_{\alpha,k}}\!\left(\exp_q^{g_{\alpha,k}}(t\Omega)\right)
&=
t^3\Vol_{g_E}(\Omega)
\\
&\quad
-
\frac{t^5}{6}
\sqrt3\,\rho_0^5
\left[
\delta
-
k\frac{1-\alpha^2}{1+\alpha^2}
\right]
+
O(t^6).
\end{aligned}
\end{equation}
This completes the proof.
\end{proof}


\subsection{Cancellation in the geodesic-ball volume expansion}\label{subsect: remark}

The wedge region computation above shows that a carefully chosen tangent regions can detect the mixed Ricci term due to discontinuity of the Ricci curvature.  This does not, however, produce a counterexample to the original geodesic ball volume-limit condition \eqref{def-NNSC-volume}.  {\em For genuine geodesic balls, the boundary displacement terms cancel the mixed Ricci contribution, and the usual scalar curvature coefficient is recovered}.

More precisely, for the small geodesic ball $B_q^{g_{\alpha,k}}(t)$ centered at $q$ with radius $t$, one has
\begin{equation}\label{eqn: quadrant-volume-expansion-final}
\Vol_{g_{\alpha,k}}\left(B_q^{g_{\alpha,k}}(t)\right)
=
\frac{4\pi}{3}t^3-\frac{4\pi}{45}t^5+o(t^5), \ \ \text{as} \ \ t \to 0.
\end{equation}
We now give a detailed computation explaining this cancellation. 

By the symmetry of the metric, it suffices to carry out the computation in the region 
\begin{equation}
Q_+ := \{(r, \theta, \xi) \mid r \geq 0, \ \ \theta \geq 0\}
\end{equation}
We work in a sufficiently small coordinate neighborhood of $q$, so that
$
\chi_{\rho_*}\equiv 1.
$
Thus, near $q$,
\begin{equation}
g_{\alpha,k}
=
(1+k\psi_\alpha)^{-2}
\bigl(dr^2+\cos^2 r\,d\theta^2\bigr)
+
(1+k\psi_\alpha)^2\,d\xi^2.
\end{equation}
On a small neighborhood of $q$, we introduce the smooth comparison metric 
\begin{equation}
g_+
:=
(1+kr\theta)^{-2}
\bigl(dr^2+\cos^2 r\,d\theta^2\bigr)
+
(1+kr\theta)^2\,d\xi^2.
\end{equation}
The two metrics agree on the pure region
\begin{equation}\label{eqn: pure-region-remark}
\mathcal P_+
:=
\left\{
r>0,\ \theta>0:
\theta\ge r^{1/\alpha},\quad r\ge \theta^{1/\alpha}
\right\}.
\end{equation}
They differ only in the bad cusp region
\[
\mathcal B
:=
Q_+\setminus \mathcal P_+
\subset
\{\theta\le r^{1/\alpha}\}
\cup
\{r\le \theta^{1/\alpha}\}.
\]

We carry out the computation in two steps. First, we show that replacing
$g_{\alpha,k}$ by the smooth model $g_+$ changes the quadrant-volume expansion
only by an $o(t^5)$ error: 
\begin{equation}\label{eqn:reduce-to-smooth-model}
\Vol_{g_{\alpha,k}}\left(B_q^{g_{\alpha,k}}(t)\cap Q_+\right)
=
\Vol_{g_+}\left(B_q^{g_+}(t)\cap Q_+\right)
+
o(t^5),
\quad \text{as} \ \ t \to 0.
\end{equation}
Second, we compute the expansion of $\Vol_{g_+}\left(B_q^{g_+}(t)\cap Q_+\right)$.

{\bf Step 1.} For small $t>0$, both small balls are contained in a Euclidean coordinate
ball $B_E(0,C_0t)$ for some constant $C_0$. Hence the relevant part of the bad region is
\begin{equation}\label{eqn: bad-region-remark}
\mathcal B_t:=\mathcal B\cap B_E(0,C_0t) \subset \mathcal B_t^{(1)}\cup \mathcal B_t^{(2)},
\end{equation}
where
\begin{equation}
\mathcal B_t^{(1)}
:=
\left\{
0\le r\le C_0t,\ 
0\le \theta\le r^{1/\alpha},\ 
|z|\le C_0t
\right\},
\end{equation}
and
\begin{equation}
\mathcal B_t^{(2)}
:=
\left\{
0\le \theta\le C_0t,\ 
0\le r\le \theta^{1/\alpha},\ 
|z|\le C_0t
\right\}.
\end{equation}
Therefore
\begin{equation}
\Vol_E(\mathcal B_t^{(1)})
\le
\int_{-C_0t}^{C_0t}
\int_0^{C_0t}
\int_0^{r^{1/\alpha}}
d\theta\,dr\,dz  \\
=
2C_0t
\int_0^{C_0t} r^{1/\alpha}\,dr  \\
\le C t^{2+1/\alpha},
\end{equation}
where $\Vol_E$ denotes the Euclidean volume. The same estimate holds for $\mathcal B_t^{(2)}$, and hence
\begin{equation}\label{eqn: Euclidean-volume-bad-set}
\Vol_E(\mathcal B_t)\le C t^{2+1/\alpha}.
\end{equation}
The volume forms
are
\begin{equation}
d\vol_{g_{\alpha,k}}
=
(1+k\psi_\alpha)^{-1}\cos r\,dr\,d\theta\,d\xi,
\end{equation}
and
\begin{equation}
d\vol_{g_+}
=
(1+kr\theta)^{-1}\cos r\,dr\,d\theta\,d\xi.
\end{equation}
Moreover, on $\mathcal B_t$ one has
\begin{equation}
|\psi_\alpha-r\theta|
\le C|r\theta|
\le C t^{1+1/\alpha}.
\end{equation}
Consequently, for small $t$,
\begin{equation}\label{eqn: density-difference-control-bad-set}
\begin{aligned}
\left| 
\int_{\mathcal B_t}
d\vol_{g_{\alpha,k}}- \int_{\mathcal{B}_t}d\vol_{g_+}
\right|
& \le
C\int_{\mathcal B_t}
|\psi_\alpha-r\theta|\,dr\,d\theta\,d\xi \\
& \le
C t^{1+1/\alpha}\Vol_E(\mathcal B_t) \\
& \le
C t^{3+2/\alpha}.
\end{aligned}
\end{equation}
In the last inequality, we used the estimate (\ref{eqn: Euclidean-volume-bad-set}).

We next show that the boundary displacement of the two small balls satisfies the same bound as in \eqref{eqn: density-difference-control-bad-set}.
Indeed, the two metrics agree in the pure region described in (\ref{eqn: pure-region-remark}), and the possible boundary
mismatch is supported only in the bad angular directions described in (\ref{eqn: bad-region-remark}). In the bad cusp near the $r$-axis, we have 
\begin{equation}
0 \leq \theta \leq r^{1/\alpha}.
\end{equation}
Thus, for $r>0$, 
\begin{equation}
\frac{\theta}{r} \leq r^{\frac{1}{\alpha} -1}.
\end{equation}
Hence the angular width of this cusp is bounded by $C t^{1/\alpha-1}$. Since the area
of $\partial B_E(C_0 t)$ is $O(t^2)$, the area of this bad angular set is
bounded by
\[
C t^2\cdot t^{1/\alpha-1}
=
C t^{1+1/\alpha}.
\]
The same estimate holds for the cusp near the $\theta$-axis. Therefore the total
area of the bad angular part of $\partial B_E(0,t)$ is at most
\begin{equation}\label{eqn: bad-angular-area}
C t^{1+1/\alpha}.
\end{equation}

It remains to estimate the radial displacement of the two boundaries along
these bad directions. On these bad directions inside $B_E(0,C_0t)$ we have
\begin{equation}
|g_{\alpha,k}-g_+|_E
\le C|\psi_\alpha-r\theta|
\le C t^{1+1/\alpha}.
\end{equation}
Consequently, for a curve $\gamma$ from $q$ to a point $q^\prime$ in $B_E(C_0t)$ with Euclidean
length $O(t)$, we have
\begin{equation}
\begin{aligned}
|L_{g_{\alpha,k}}(\gamma)-L_{g_+}(\gamma)|
&\le
\int_\gamma
\bigl||\dot\gamma|_{g_{\alpha,k}}-|\dot\gamma|_{g_+}\bigr|\,ds  \\
&\le
C t^{1+1/\alpha}\int_\gamma |\dot\gamma|_E\,ds  \\
&\le
C t^{1+1/\alpha}\cdot t \\
&=
C t^{2+1/\alpha}.
\end{aligned}
\end{equation}
Taking the infimum over admissible curves gives
\begin{equation}\label{eqn: radial-displacement}
\bigl|
d_{g_{\alpha,k}}(q, q^\prime)-d_{g_+}(q, q^\prime)
\bigr|
\le C t^{2+1/\alpha}.
\end{equation}
By \eqref{eqn: bad-angular-area} and \eqref{eqn: radial-displacement}, the volume swept out by the boundary displacement is also
\begin{equation}\label{eqn: displacement-control}
O(t^{3+2/\alpha}).
\end{equation}
Since $0<\alpha<1$, we have
$
3+\frac{2}{\alpha}>5.
$
Combining this with \eqref{eqn: density-difference-control-bad-set}, we obtain \eqref{eqn:reduce-to-smooth-model}.

{\bf Step 2.}
We now derive the expansion of $\Vol_{g_+}\left(B_q^{g_+}(t)\cap Q_+\right)$. Let
\begin{equation}
U_t:=(\exp_q^{g_+})^{-1}
\bigl(B_q^{g_+}(t)\cap Q_+\bigr)\subset T_qM,
\end{equation}
and let
\begin{equation}
B_E^+(t):=B_E(0,t)\cap\{x\ge0,\ y\ge0\},
\end{equation}
where $(x,y,z)$ are the coordinates on $T_qM$ corresponding to the basis
$
\{ \partial_r, \, \partial_\theta, \, \partial_\xi\}.
$

We decompose
\begin{equation}\label{eqn: smooth-model-decomposition}
\begin{aligned}
&\Vol_{g_+}\bigl(B_q^{g_+}(t)\cap Q_+\bigr)
-
\tfrac{\pi}{3}t^3
\\
&\quad =
\left(
\Vol_{g_+}(\exp_q^{g_+}(U_t))-\Vol_E(U_t)
\right)
+
\left(
\Vol_E(U_t)- \tfrac{\pi}{3}t^3
\right).
\end{aligned}
\end{equation}

We first compute the Jacobian term. At $q$, the Ricci tensor of $g_+$ in the
basis $\{ \partial_r, \, \partial_\theta, \, \partial_\xi \}$ is
\begin{equation}
\Ric_q^{g_+}
=
\begin{pmatrix}
1 & -k & 0\\
-k & 1 & 0\\
0 & 0 & 0
\end{pmatrix}.
\end{equation}
Thus
\begin{equation}
\Ric_q^{g_+}(x,y,z)
=
x^2+y^2-2kxy.
\end{equation}
The standard Jacobian expansion for the exponential map then gives
\begin{equation}\label{eqn: jacobian-term-quadrant}
\begin{aligned}
\Vol_{g_+}(\exp_q^{g_+}(U_t))-\Vol_E(U_t)
 & =
-\frac16\int_{B_E^+(t)}
\Ric_q^{g_+}(x, y, z)\,dx dy dz
+
O(t^6) \\
&=
-\frac16
\left(
\frac{2\pi}{15}
-
\frac{4k}{15}
\right)t^5
+
O(t^6).
\end{aligned}
\end{equation}

We next compute the Euclidean boundary displacement term
$
\Vol_E(U_t)- \frac{\pi}{3}t^3.
$
This term appears because the coordinate quadrant
$
Q_+=\{r\ge0,\theta\ge0\}
$
does not pull back exactly to the Euclidean quadrant
$
\{x\ge0,y\ge0\}
$
under $\exp_q^{g_+}$.

There are two boundary faces. First consider the face $\{\theta=0\}$. Its
inverse image under $\exp_q^{g_+}$ is a graph over the $(x,z)$-plane. Writing
\begin{equation}
x=s\cos\beta,
\qquad
z=s\sin\beta,
\qquad
-\frac{\pi}{2}\le\beta\le\frac{\pi}{2},
\qquad
0\le s\le t,
\end{equation}
as discussed in Remark \ref{rmrk: geodesic-non-uniqueness}, this graph has the expansion
\begin{equation}
y
=
-\frac16 k\cos\beta
\left(\sin^2\beta-\cos^2\beta\right)s^3
+
O(s^4).
\end{equation}
Therefore its signed Euclidean volume contribution is
\begin{equation}\label{eqn:theta-face-contribution}
\begin{aligned}
V_E^*(\{\theta=0\})
&=
\int_{-\pi/2}^{\pi/2}\int_0^t
(-y)\,s\,ds\,d\beta
\\
&=
-\frac{k}{45}t^5+O(t^6).
\end{aligned}
\end{equation}
Here signed volume means that the part of $U_t\setminus B_E^+(t)$ is counted
positively, while the part of $B_E^+(t)\setminus U_t$ is counted negatively.

Similarly, the inverse image of the face $\{r=0\}$ gives
\begin{equation}\label{eqn:r-face-contribution}
V_E^*(\{r=0\})
=
-\frac{k}{45}t^5+O(t^6).
\end{equation}
Hence
\begin{equation}\label{eqn: boundary-displacement-term}
\Vol_E(U_t)- \frac{\pi}{3}t^3
=
-\frac{2k}{45}t^5+O(t^6).
\end{equation}

Combining \eqref{eqn: smooth-model-decomposition},
\eqref{eqn: jacobian-term-quadrant}, and
\eqref{eqn: boundary-displacement-term}, we obtain
\begin{equation}\label{eqn:smooth-model-final}
\begin{aligned}
\Vol_{g_+}\bigl(B_q^{g_+}(t)\cap Q_+\bigr)
-
\frac{\pi}{3}t^3
&=
-\frac16
\left(
\frac{2\pi}{15}
-
\frac{4k}{15}
\right)t^5
-\frac{2k}{45}t^5
+
O(t^6)
\\
&=
-\frac{\pi}{45}t^5+O(t^6).
\end{aligned}
\end{equation}

Finally, combining \eqref{eqn:reduce-to-smooth-model} and
\eqref{eqn:smooth-model-final}, we get \eqref{eqn: quadrant-volume-expansion-final} as claimed.\\

The $C^{1, \alpha}$ example above, together with the pulled-string example of  Basilio--Dodziuk--Sormani \cite{BasilioDodziukSormani-sewing} and the drawstring example of Kazaras--Xu in \cite{KazarasXu2023}, indicate that the volume-limit notion of scalar curvature is quite subtle. Note that in the Kazaras--Xu's example, the convergence is not $C^0$. It is therefore natural to ask the following question: is the volume-limit notion of nonegative scalar curvature preserved under stronger convergence?

More concretely, consider a metric $g$ of the form \eqref{metric-form}, with $u=\cos r$, and assume that $\log \varphi$ is a Lipschitz function with Lipschitz constant at most $1$ (or even $\log \varphi \in C^{1,\alpha}$ with $|\nabla \log \varphi|\leq 1$). Does $g$ have nonnegative scalar curvature in the volume-limit sense?


\appendix

\section{A drawstring-type example with \texorpdfstring{$d_*<d_{g_\infty}$}{d* < dg-infinity}}
\label{app:shortcut-example}

The purpose of this appendix is to illustrate that the global limiting distance
$d_*$ obtained from the sequence need not agree with the intrinsic length
distance $d_{g_\infty}$ induced by the local tensor limit $g_\infty$.  Example \ref{example-appendix} below is a drawstring type construction of
Kazaras--Xu \cite{KazarasXu2023}.  We include the details in the present
notation in order to isolate the distance comparison
$
        d_*<d_{g_\infty}
$
at some points,
and to check the hypotheses appearing in Conjecture \ref{conj: Scalar-Compactness}.

We construct a sequence of metrics on $\Sph^2 \times \Sph^1$ as
\begin{equation}
g_i := \varphi_i(r)^{-2}h_i + \varphi_i(r)^2 d\xi^2, \quad \xi\in[0, 2\pi]
\end{equation}
where $h_i$ are rotationally symmetric metrics on $\Sph^2$ as
\begin{equation}
h_i:= dr^2 + u_i(r)^2 d\theta^2, \quad r\in[0, 2], \ \ \theta\in[0, 2\pi],
\end{equation}
 in which the warping factors
\(\varphi_i\) develop a short cut region near the pole at $r=0$.  This regions has vanishing volume and is invisible in the local analytic convergence
away from the pole, but it makes the $S^1$-fiber direction extremely cheap
inside a tiny cap, thereby producing a genuine shortcut between points in the
regular region.
   
 \medskip  
 We begin with the construction of the warping functions $\varphi_i$. Let $A_i\to\infty$, and set
\begin{equation}\label{eqn: rho-defn}
\rho_i:=\frac{1}{10} e^{-A_i^3}.
\end{equation}
Let $\eta\in C^\infty(\mathbb{R})$ satisfy $0\le\eta\le1$, and
\begin{equation}
\eta(t)=
\begin{cases}
0, & t\leq 0, \\
1, & t\geq 1.
\end{cases}
\end{equation}
Set
\begin{equation}
\psi_i(r)
:=
\eta\!\left(\frac{\ln(r/\rho_i)}{\ln2}\right)
\eta\!\left(\frac{\ln(1/(10r))}{\ln2}\right).
\end{equation}
Then $0\le\psi_i\le1$, and
\begin{equation}\label{eqn: psi-value}
\psi_i(r)=
\begin{cases}
0,  & r \in (0,\rho_i]\cup[\frac{1}{10},\infty) \\
1,  & r \in [2\rho_i,1/20].
\end{cases}
\end{equation} 
Define
\begin{equation}\label{eqn: f-defn}
\lambda_i
:=
\frac{A_i}{\displaystyle\int_0^{\frac{1}{10}} \frac{\psi_i(s)}{s}\,ds},
\qquad
f_i(r)
:=
-A_i+\lambda_i\int_0^r\frac{\psi_i(s)}{s}\,ds .
\end{equation}
Then 
\begin{equation}
f_i(r) =
\begin{cases}
-A_i, & r \in [0,\rho_i] \\
0, &  r \in [\frac{1}{10},\infty),
\end{cases}
\end{equation} 
and
\begin{equation}\label{eqn: f-derivative-psi}
f_i'(r)=\lambda_i\frac{\psi_i(r)}{r} \geq 0.
\end{equation}
Moreover, by \eqref{eqn: psi-value}, we have
\begin{equation}
\int^{\frac{1}{10}}_0 \frac{\psi_i(s)}{s}ds \leq \int^{\frac{1}{10}}_{\rho_i} \frac{1}{s}ds = \ln \frac{\frac{1}{10}}{\rho_i} =A^3_i,
\end{equation}
and also
\begin{equation}
\begin{aligned}
\int_0^\frac{1}{10} \frac{\psi_i(s)}{s}\,ds
& =
\int^{2\rho_i}_{\rho_i} \frac{\psi_i(s)}{s}ds
+
\int_{2\rho_i}^{1/20}\frac{ds}{s}
+
\int^{\frac{1}{10}}_{\frac{1}{20}} \frac{\psi_i(s)}{s}ds + O(1) \\
&=
\log\frac{1}{40\rho_i} + O(1)
=
A_i^3 +O(1),
\end{aligned}
\end{equation}
since 
\begin{equation}
0\leq \int^{2\rho_i}_{\rho_i} \frac{\psi_i(s)}{s}ds \leq \ln 2 \ \ \text{and} \ \ 0\leq  \int^{\frac{1}{10}}_{\frac{1}{20}} \frac{\psi_i(s)}{s}ds \leq \ln 2.
\end{equation}
Thus, by the definition of $\lambda_i$ in \eqref{eqn: f-defn}, we obtain
\begin{equation}\label{eqn: lambda-lower-bound}
\lambda_i \geq \frac{1}{A^2_i}, \quad \forall i \in \mathbb{N},
\end{equation}
and for large $i$,
\begin{equation}\label{eqn: lambda-asymptotic}
\lambda_i = \frac{A_i}{A^3_i + O(1)} = \frac{1}{A^2_i} + O(A^{-5}_i).
\end{equation}
Combined with \eqref{eqn: f-derivative-psi}, this implies
\begin{equation}\label{eqn: f-derivative-estimate}
0\le f_i'(r)\le \frac{C}{A_i^2r},
\qquad \rho_i \leq r\le \frac{1}{10},
\end{equation}
with $C$ independent of $i$. 

Define
\begin{equation}\label{eqn: varphi-construction}
\varphi_i:=e^{f_i}.
\end{equation}
Then
\begin{equation}
\varphi_i(\rho_i)=e^{-A_i}\to0,
\end{equation}
while
\begin{equation}
\varphi_i\to1
\end{equation}
locally uniformly on \(S^2\setminus\{r=0\}\). For $0<r<\frac{1}{10}$, this can be seen by $|f_i(r)-f_i(\frac{1}{10})|=\lambda_i|\int_r^\frac{1}{10}\frac{\psi_i}{s}ds|\leq \lambda_i\log \frac{1}{10r}\rightarrow 0$ as $i\rightarrow \infty$.

\medskip
We next construct the base warping functions $u_i$ so that the resulting
metrics have nonnegative scalar curvature.

\begin{lemma}\label{lem: u-construction}
Let $f_i$ be the functions constructed above. Then after passing to large $i$, there exist smooth concave functions
\(u_i\in C^\infty([0,2])\) such that
\begin{equation}
u_i(0)=u_i(2)=0,\quad \sup u\geq \frac{1}{10},\quad u_i'(0)=1,\quad u_i'(2)=-1,
\quad
u_i>0\ \ \text{on }(0,2),
\end{equation}
\begin{equation}\label{u-ends}
\begin{cases}
u'_i(r) >0, & 0<r<\frac{1}{5}, \\
u_i(r) = r, & 0\leq r<\rho_i,\\
u_i(r) = 2-r, & \frac{9}{5}<r\leq 2,
\end{cases}
\end{equation}
and
\begin{equation}\label{eqn: u-f-differential-inq-example}
-\frac{u_i''(r)}{u_i(r)}\ge |f'_i|^2(r)
\quad\text{for }0<r<2.
\end{equation}
\end{lemma}

\begin{proof}
We prove the existence of $u_i$ in two steps. In the first step, we prove existence on the interval $[0, \frac{1}{5}]$ by solving IVP. In the second step, we extend the solution on $[\frac{1}{5}, 2]$.

\medskip
{\bf Step 1.}
Because $|f'_i|^2\in C^\infty([0,2])$, the initial value problem
\begin{equation}\label{eqn: IVP-u}
u_i''+|f'_i|^2 u_i=0,
\qquad
u_i(0)=0,\qquad
u_i'(0)=1
\end{equation}
has a unique smooth solution on $[0,\frac{1}{5}]$. 

Moreover, by \eqref{eqn: f-derivative-estimate}, we have
\begin{equation}
\int^{1/5}_0 |f'_i|^2(r)rdr = \int^{1/10}_{\rho_i} |f'_i|^2(r)rdr \leq \frac{C^2}{A_i} \to 0, \quad \text{as} \ \ i \to \infty.
\end{equation}
For large \(i\), this implies that the solution to \eqref{eqn: IVP-u} remains positive,
strictly increasing and concave on \([0,\frac{1}{5}]\). Indeed, as long as \(u_i\ge0\), we have
\begin{equation}
u_i''=-|f'_i|^2 u_i\le0,
\end{equation}
and hence $u_i(r)\le r$. Therefore
\begin{equation}
u_i'(r)
=
1-\int_0^r |f'_i|^2(s)u_i(s)\,ds
\ge
1-\int_0^r |f'_i|^2(s)s\,ds
\ge
\frac12
\end{equation}
for all $0\le r\le \frac{1}{5}$, once $i$ is large. Hence
\begin{equation}
u_i'(r)>0
\quad\text{and}\quad
u_i(r)>0
\qquad\text{on }(0,\frac{1}{5}].
\end{equation}
Set
\begin{equation}
\beta_i:=u_i(\frac{1}{5}),
\qquad
\alpha_i:=u_i'(\frac{1}{5}).
\end{equation}
Then, by the concavity of $u_i$, we have
\begin{equation}
0<\beta_i\le \frac{1}{5},
\qquad
\frac12\le \alpha_i\le1.
\end{equation}
Since $u_i$ is concave, it follows $u'_i\geq \frac{1}{2}$ on $[0,\frac{1}{5}]$ and $u(\frac{1}{5})\geq \frac{1}{10}$. Therefore $\sup u_i\geq \frac{1}{10}$.
Since $|f'_i|^2(r)=0$ for $r\in [0,\rho_i]$, the $u_i(r)=r$ in this region.

\medskip
{\bf Step 2.}
We now extend $u_i$ on $[\frac{1}{5}, 2]$. We  may choose a smooth function
$p_i(r):[\frac{1}{5}, 2]\to\mathbb R$ such that
\begin{equation}\label{eqn: p-1}
p_i(r)=
\begin{cases}
\alpha_i, & \text{near }\frac{1}{5}, \\
-1, & \text{on }[2- \frac{1}{5}, 2],
\end{cases}
\end{equation}
\begin{equation}
p_i'(r)\le0 \ \ \text{on} \ \ [\tfrac{1}{5}, 2],
\end{equation}
and
\begin{equation}\label{eqn: integral-p}
\int_{\frac{1}{5}}^2 p_i(s)\,ds=-\beta_i,
\end{equation}

Define, for $\frac{1}{5} \leq r\le2$,
\begin{equation}
u_i(r):=\beta_i+\int_{\frac{1}{5}}^r p_i(s)\,ds.
\end{equation}
Then
\begin{equation}
u_i(\tfrac{1}{5})=\beta_i,
\qquad
u_i'(\tfrac{1}{5})=\alpha_i,
\end{equation}
and since $p_i(r)=\alpha_i$ near $\frac{1}{5}$, this extension matches the IVP
solution on $[0, \frac{1}{5}]$ smoothly at $r=\frac{1}{5}$. Also,
for $r \in [2-\tfrac{1}{5}, 2]$, we have
\begin{equation}
    u_i(r) = \beta_i + \int^r_{\tfrac{1}{5}} p_i(s)ds
    = \beta_i + \int^2_{\tfrac{1}{5}}p_i(s)ds - \int^2_r p_i(s)ds
    =  \int^2_{r}ds = 2-r.
\end{equation}
Here we used \eqref{eqn: p-1} and \eqref{eqn: integral-p}.
Furthermore,
\begin{equation}
u_i''=p_i'\le0
\quad\text{on }[\tfrac{1}{5}, 2].
\end{equation}
Thus $u_i$ is concave on $[\tfrac{1}{5}, 2]$. Since $u_i(\frac{1}{5})=\beta_i>0$ and
$ u_i(2)=0$, concavity implies
\begin{equation}
u_i(r)>0
\quad\text{for }\frac{1}{5}<r<2.
\end{equation}

Finally, on $[0, \frac{1}{5}]$ we have
\begin{equation}
-\frac{u_i''}{u_i}=|f'_i|^2,
\end{equation}
while on $[\frac{1}{5}, 2]$ we have
\begin{equation}
-\frac{u_i''}{u_i}\ge0.
\end{equation}
Since $|f'_i|^2$ is supported in $(0,\frac{1}{10})$, it follows that
\begin{equation}
-\frac{u_i''}{u_i}\ge |f'_i|^2
\quad\text{on }(0,2).
\end{equation}

\end{proof}

\begin{example}\label{example-appendix}
    Consider the sequence of warped product Riemannian metrics on $\Sph^2 \times \Sph^2$ given by 
\begin{equation}
g_i = \varphi_i(r)^{-2} (dr^2 + u^2_i(r) d\theta^2) + \varphi_i(r)^2 d\xi^2, \quad r\in[0, 2], \ \ \theta, \xi \in [0, 2\pi],
\end{equation}
where $\varphi_i$ is constructed in \eqref{eqn: varphi-construction} and $u_i$ is obtained in Lemma \ref{lem: u-construction}. 
\end{example}

\begin{lemma}\label{lem: example-scalar-curvature}
The Riemannian metrics in Example \ref{example-appendix} are smooth on $\Sph^2 \times \Sph^1$. Moreover, $g_i$ have nonnegative scalar curvature, i.e.
\begin{equation}
\Scal_{g_i} \geq 0 \ \ \text{on} \ \ \Sph^2 \times \Sph^1, \ \ \forall i \in \mathbb{N}.
\end{equation}
\end{lemma}
\begin{proof}
By \eqref{u-ends}, the metric $dr^2 + u_i^2 d\theta^2$ is locally flat near the two poles $r = 0$ and $r = 2$. Hence it is a smooth metric on $\mathbb{S}^2$. Consequently, each $g_i$ is a smooth metric, since $\varphi_i$ are smooth positive functions on $\Sph^2$. The assertion of non‑negative scalar curvature then follows from \eqref{eqn: u-f-differential-inq-example} and Lemma \ref{lem: scalar-curvature-formula}.
\end{proof}

We first check the volume hypotheses.

\begin{lemma}[Uniform bound for volumes]\label{lem: example-volume-bound}
For the sequence of metrics constructed in Example \ref{example-appendix}, there exists a constant $V_0$ independent of $i$ such that
\begin{equation}
\Vol_{g_i}(\Sph^2 \times \Sph^1) \leq V_0, \ \ \forall i \in \mathbb{N}.
\end{equation}
\end{lemma}
\begin{proof}
The volume form of $g_i$ is
$
d\vol_{g_i}
=
\frac{u_i(r)}{\varphi_i(r)}\,dr\,d\theta\,d\xi.
$
Therefore
\begin{equation}
\Vol_{g_i}(S^2\times S^1)
=
4\pi^2\int_0^2\frac{u_i(r)}{\varphi_i(r)}\,dr.
\end{equation}
It suffices to prove that
\begin{equation}\label{eqn: integral-volume}
\int_0^2\frac{u_i(r)}{\varphi_i(r)}\,dr
\end{equation}
is uniformly bounded.

By Lemma \ref{lem: u-construction}, the functions $u_i$ are concave on \([0,2]\), with
\begin{equation}
u_i(0)=0,
\qquad
u_i'(0)=1.
\end{equation}
Hence \(u_i'\le1\) on \([0,2]\), and therefore
\begin{equation}\label{eqn: example-u-bound}
0\le u_i(r)\le r\le2.
\end{equation}

We split the integral \eqref{eqn: integral-volume} into four regions:
\begin{equation}
[0,2\rho_i],
\qquad
[2\rho_i,1/20],
\qquad
[1/20,1/10],
\qquad
[1/10,2].
\end{equation}

First, on $[0,2\rho_i]$, we have $f_i\ge -A_i$, hence
\begin{equation}
\varphi_i=e^{f_i}\ge e^{-A_i}.
\end{equation}
Using \(u_i(r)\le r\), we get
\begin{equation}
\int_0^{2\rho_i}\frac{u_i(r)}{\varphi_i(r)}\,dr
\le
e^{A_i}\int_0^{2\rho_i}r\,dr
=
2e^{A_i}\rho_i^2  = \frac{1}{50}e^{A_i-2A_i^3}
\to0.
\end{equation}
In particular, this integral is uniformly bounded.

Next consider the logarithmic region $2\rho_i\le r\le 1/20$. Since
$\psi_i=1$ on this interval, we have
\begin{equation}
\int_0^r\frac{\psi_i(s)}s\,ds
\ge
\int_{2\rho_i}^r\frac{ds}{s}
=
\ln\frac r{2\rho_i}.
\end{equation}
Therefore
\begin{equation}
f_i(r)
=
-A_i+\lambda_i\int_0^r\frac{\psi_i(s)}s\,ds
\ge
-A_i+\lambda_i\ln\frac r{2\rho_i}.
\end{equation}
Using \eqref{eqn: lambda-lower-bound} and \eqref{eqn: rho-defn}, we get
\begin{equation}
f_i(r)
\ge
-A_i+\frac1{A_i^2}\ln\frac r{2\rho_i} = -A_i+\frac1{A_i^2}
\left(
\ln(5r)+A_i^3
\right) 
=\frac1{A_i^2}\ln(5r).
\end{equation}
Consequently,
\begin{equation}
\frac1{\varphi_i(r)}
=
e^{-f_i(r)}
\le
\left(\frac{1}{5r}\right)^{1/A_i^2}.
\end{equation}
Using again $u_i(r)\le r$, we obtain
\begin{equation}
\begin{aligned}
\int_{2\rho_i}^{1/20}\frac{u_i(r)}{\varphi_i(r)}\,dr
&\le
\int_{2\rho_i}^{1/20}
r\left(\frac{1}{5r}\right)^{1/A_i^2}\,dr \\
&=
(\frac{1}{5})^{1/A_i^2}
\int_{2\rho_i}^{1/20}
r^{1-\frac1{A_i^2}}\,dr.
\end{aligned}
\end{equation}
For all sufficiently large \(i\), \(A_i^{-2}\le1/2\). Hence
\begin{equation}
1-\frac1{A_i^2}\ge \frac12,
\end{equation}
and the last integral is uniformly bounded by a constant,
\begin{equation}
\int_{2\rho_i}^{1/20}\frac{u_i(r)}{\varphi_i(r)}\,dr
\le C.
\end{equation}

Now consider the outer transition region $[1/20,1/10]$. For $r\ge 1/20$,
we have
\begin{equation}
\int_0^r\frac{\psi_i(s)}s\,ds
\ge
\int_{2\rho_i}^{1/20}\frac{ds}{s}
=
\ln\frac 1{40\rho_i}
=
A_i^3-\ln4.
\end{equation}
Combined with \eqref{eqn: lambda-lower-bound}, this implies that, for $r\in[1/20,1/10]$,
\begin{equation}
f_i(r)
\ge
-A_i+\frac1{A_i^2}(A_i^3-\ln4)
=
-\frac{\ln4}{A_i^2}.
\end{equation}
Therefore
\begin{equation}
\frac1{\varphi_i(r)}
=
e^{-f_i(r)}
\le
e^{\frac{\ln4}{A_i^2}}
\le C.
\end{equation}
Since \(u_i(r)\le2\), it follows that
\begin{equation}
\int_{1/20}^{1/10}\frac{u_i(r)}{\varphi_i(r)}\,dr
\le C.
\end{equation}

Finally, on $[1/10,2]$, we have $f_i=0$, hence $\varphi_i=1$. Therefore
\begin{equation}
\int_{1/10}^2\frac{u_i(r)}{\varphi_i(r)}\,dr
=
\int_{1/10}^2u_i(r)\,dr
\le
2\times \frac{19}{20}.
\end{equation}
Here we used \eqref{eqn: example-u-bound}.

Combining the four estimates, complete the proof of the lemma.
\end{proof}

We next check the $\MinA$ condition.  The argument is based on mean-convex
barriers and the monotonicity formula for minimal surfaces from Colding and Minicozzi's textbook \cite{ColdingMinicozziBook}; it is in the same
spirit as the proof of the corrresponding noncollapsing estimate in Sormani--Tian--Wang \cite[Theorem 5.1]{SormaniTianWang2024} and Kazaras--Xu \cite[Theorem 1.3]{KazarasXu2023}.

\begin{lemma}[Uniform bound for MinA]\label{lem: example-MinA}
For the sequence of metrics constructed in Example \ref{example-appendix}, there exists a constant $A_0>0$, independent of $i$ , such that every closed embedded minimal
surface $\Sigma\subset (S^2\times S^1,g_i)$ satisfies
\begin{equation}
\Area_{g_i}(\Sigma)\ge A_0.
\end{equation}
Consequently,
\begin{equation}
\MinA(\Sph^2\times \Sph^1,g_i)\ge A_0>0.
\end{equation}
\end{lemma}
\begin{proof}
By Lemma \ref{lem: u-construction}, we have
\begin{equation}\label{eqn: u-derivative-sign-cap}
\begin{cases}
u_i'(r)>0,
& \text{for} \ \ 0<r<\frac{1}{5}, \\
u_i'(r)<0,
& \text{for} \ \ \frac{9}{5}<r<2.
\end{cases}
\end{equation}

We first show that the two end regions are strict mean-convex barriers. For
$0<r<2$,   the mean
curvature of torus
\begin{equation}
T_r:=\{r=\mathrm{constant}\}\times S^1_\xi.
\end{equation}
 with respect to the unit normal vector $\nu_i=\varphi_i\partial_r$ is
\begin{equation}
H_i(r)
=
\varphi_i(r,\theta)\frac{u_i'(r)}{u_i(r)}.
\end{equation}
Hence, by \eqref{eqn: u-derivative-sign-cap}, for $0<r<\frac{1}{5}$,
$
H_i(r)>0
$
with respect to the outward normal of the region $\{s\le r\}$. 
Similarly, for \(9/5<r<2\), since $u_i'(r)<0$, the tori
$T_r$ are strictly mean-convex with respect to the outward normal of the
region $\{s\ge r\}$, namely $-\varphi_i\partial_r$.

Now let $\Sigma\subset (S^2\times S^1,g_i)$ be a closed embedded minimal
surface. We claim that $\Sigma$ must intersect the fixed compact core region
\begin{equation}
\mathcal K
:=
\{\frac{1}{5}\le r\le \frac{9}{5}\}
\times \Sph^1_\theta
\times \Sph^1_\xi .
\end{equation}
Indeed, suppose first that a connected component of $\Sigma$ were entirely
contained in the cap region
\begin{equation}
\{0<r<\frac{1}{5}\}
\times \Sph^1_\theta
\times \Sph^1_\xi
\subset \Sph^2\times\Sph^1 .
\end{equation}
Then the function $r|_\Sigma$ attains a maximum $r_0<1/5$. At a maximum
point, $\Sigma$ is tangent to $T_{r_0}$ and lies on the side
$\{r\le r_0\}$. Since $T_{r_0}$ is strictly mean-convex with respect to
the outward normal of this region, the strong maximum principle for minimal
hypersurfaces gives a contradiction. Hence no closed minimal surface
component can be contained entirely in $\{0<r<1/5\}\times S^1$.

The same argument
shows that no closed minimal surface component can be contained entirely in the other cap region
$\{9/5<r<2\}\times S^1$. Therefore every connected component of $\Sigma$ must
intersect $\mathcal K$. In particular, $\Sigma\cap\mathcal K\neq\emptyset$ as claimed.

By construction, the metrics $g_i$ have uniformly bounded geometry on the
compact region
\begin{equation}
\mathcal K'
:=
\{\frac{1}{10}\le r\le \frac{19}{10}\}\times \Sph^1_\theta\times \Sph^1_\xi
\supset \mathcal K .
\end{equation}
More precisely, there exist constants $\kappa<\infty$ and $\iota_0>0$,
independent of $i$, such that on $\mathcal K'$,
\begin{equation}
|\operatorname{Sec}_{g_i}|\le \kappa,
\qquad
\operatorname{inj}_{g_i}\ge \iota_0 .
\end{equation}
Indeed, on $\mathcal K'$ one has $\varphi_i\equiv1$, and the functions
$u_i$ have uniform positive lower bounds and uniform $C^2$-bounds.  The latter
can be arranged in the extension step above, since the endpoint data
$\beta_i=u_i(1/5)$ and $\alpha_i=u_i'(1/5)$ satisfy
\begin{equation}
0<\beta_i\le 1/5,
\qquad
\frac12\le \alpha_i\le1,
\end{equation}
so the interpolating functions $p_i$ on $[1/5,2]$ may be chosen with uniform
$C^1$-bounds, and hence $u_i$ with uniform $C^2$-bounds on $\mathcal K'$.

Take a point
\begin{equation}
x_i\in\Sigma\cap\mathcal K,
\end{equation}
and
\begin{equation}
0<\rho_0<\frac12\operatorname{dist}_{g_i}(\mathcal K,\partial\mathcal K')
\end{equation}
uniformly in $i$, decreasing $\rho_0$ if necessary so that the
monotonicity formula for minimal surfaces applies in
$B_{g_i}(x_i,\rho_0)$. Then there exists a constant $c_0>0$, independent
of \(i\), such that
\begin{equation}
\Area_{g_i}(\Sigma)
\ge
\Area_{g_i}\bigl(\Sigma\cap B_{g_i}(x_i,\rho_0)\bigr)
\ge
c_0\rho_0^2.
\end{equation}
This completes the proof by setting
$ A_0:=c_0\rho_0^2>0.$
\end{proof}

By Lemmas \ref{lem: example-scalar-curvature}, \ref{lem: example-volume-bound} and \ref{lem: example-MinA}, the sequence of metrics in Example \ref{example-appendix} satisfies the hypotheses in the condition \ref{eqn: condition-lambda-intro}. Therefore, Proposition \ref{prop: diameter-uniform-bound} implies
\begin{lemma}
    For the sequnce of metrics construced in Example \ref{example-appendix}, there exists a constant $D>0$ such that the diameters
    \begin{equation}
        \Diam_{g_i}(\Sph^2 \times \Sph^2) \leq D, \quad \forall i \in \mathbb{N}.
    \end{equation}
\end{lemma}

The next estimate shows that traveling radially down to the shortcut region
has essentially the same cost as in the limiting metric, while traveling in the
$\xi$-direction inside the shortcut region has vanishing cost.

\begin{lemma}\label{lem: radial-estimate}
For the functions $\varphi_i$ defined in \eqref{eqn: varphi-construction}, and a fixed $0< r_* < \frac{1}{40}$, we have
\begin{equation}
\int_{\rho_i}^{r_*}\frac{dr}{\varphi_i(r)}
=
r_*+o(1), 
\quad \text{as} \ \ i \to \infty.
\end{equation}
\end{lemma}
\begin{proof}
We split the integral as
\begin{equation}
\int_{\rho_i}^{r_*}\frac{dr}{\varphi_i(r)}
=
\int_{\rho_i}^{2\rho_i}\frac{dr}{\varphi_i(r)}
+
\int_{2\rho_i}^{r_*}\frac{dr}{\varphi_i(r)}.
\end{equation}

First, on $[\rho_i,2\rho_i]$, we have $\varphi_i(r)\ge e^{-A_i}.$
Therefore
\begin{equation}\label{eqn: negligible}
\int_{\rho_i}^{2\rho_i}\frac{dr}{\varphi_i(r)}
\le
e^{A_i}\rho_i
=
\frac{1}{10} e^{A_i-A_i^3}
\to0.
\end{equation}

Next consider the integral on $[2\rho_i,r_*]$. Since \(r_*<1/40<1/20\),
for all sufficiently large \(i\), this interval is contained in the region
where \(\psi_i=1\), except for the already discarded transition interval.
Thus, for \(2\rho_i\le r\le r_*\),
\begin{equation}
\int_0^r\frac{\psi_i(s)}s\,ds
=
\int^{r}_{2\rho_i} \frac{\psi_i(s)}{s}ds + \int^{2\rho_i}_{\rho_i} \frac{\psi_i(s)}{s}ds
=
\ln\frac r{\rho_i}+O(1),
\end{equation}
where the $O(1)$ term comes only from the logarithmic integral over $[\rho_i, 2\rho_i]$. 
By \eqref{eqn: rho-defn},
we have
\begin{equation}
\ln\frac r{\rho_i}
=
\ln (10r)+A_i^3.
\end{equation}
Hence
\begin{equation}
\begin{aligned}
f_i(r)
&=
-A_i+\lambda_i\left(A_i^3+\ln (10r)+O(1)\right)\\
&=
\lambda_i\ln (10r)+O(A_i^{-2}).
\end{aligned}
\end{equation}
Here we used \eqref{eqn: lambda-asymptotic}.
Therefore
\begin{equation}
e^{-f_i(r)}
=
e^{O(A_i^{-2})}
\left(\frac{1}{10r}\right)^{\lambda_i}.
\end{equation}
Consequently,
\begin{equation}
\int_{2\rho_i}^{r_*}\frac{dr}{\varphi_i(r)}
=
e^{O(A_i^{-2})}
\int_{2\rho_i}^{r_*}
\left(\frac{1}{10r}\right)^{\lambda_i}\,dr.
\end{equation}
Since \(\lambda_i\to0\), we compute
\begin{equation}
\int_{2\rho_i}^{r_*}
\left(\frac{1}{10r}\right)^{\lambda_i}\,dr
=
\frac{1}{10^{\lambda_i}(1-\lambda_i)}
\left[
r_*^{1-\lambda_i}
-
(2\rho_i)^{1-\lambda_i}
\right].
\end{equation}
Now
\begin{equation}
10^{\lambda_i}\to1,
\qquad
\frac1{1-\lambda_i}\to1,
\qquad
r_*^{1-\lambda_i}\to r_*.
\end{equation}
Moreover,
\begin{equation}
(2\rho_i)^{1-\lambda_i}
\to0.
\end{equation}

Thus
\begin{equation}
\int_{2\rho_i}^{r_*}\frac{dr}{\varphi_i(r)}
=
r_*+o(1).
\end{equation}
Together with estimate in \eqref{eqn: negligible}, this completes the proof of the lemma.
\end{proof}

\begin{lemma}[Strict shortening of the limit distance]\label{lem: example-compare-distances}
For the sequence of metrics constructed in Example \ref{example-appendix}, the limiting distance $d_*$
is strictly smaller than the intrinsic length distance $d_{g_\infty}$ of
the local tensor limit $g_\infty$. More precisely, there exist two points
$p,q\in \mathring M$ such that
\begin{equation}
d_*(p,q)<d_{g_\infty}(p,q).
\end{equation}
\end{lemma}

\begin{proof}
Recall that, away from the pole $r=0$, the metrics $g_i$ converge locally
smoothly to
\begin{equation}
g_\infty=h_\infty+d\xi^2.
\end{equation}

Choose a fixed number
\begin{equation}
0<r_*<\frac{1}{40} < \frac{\pi}{2}.
\end{equation}
Fix \(\theta_0\in S^1\), and consider the two points
\begin{equation}
p=(r_*,\theta_0,0),
\,
q=(r_*,\theta_0,\pi) \in \mathring{M}.
\end{equation}

We first compute their distance in the limit metric \(g_\infty\). For any
piecewise smooth curve $\gamma(t)=(r(t),\theta(t),\xi(t)) $
joining \(p\) to \(q\), the \(g_\infty\)-length satisfies
\begin{equation}
L_{g_\infty}(\gamma)
\ge
\int |\xi'(t)|\,dt.
\end{equation}
Since the $\xi$-coordinates of $p$ and $q$ differ by $\pi$, we have
\begin{equation}
\int |\xi'(t)|\,dt\ge \pi.
\end{equation}
Therefore
\begin{equation}
d_{g_\infty}(p,q)\ge \pi.
\end{equation}
On the other hand, the vertical fiber curve
\begin{equation}
\gamma_0(s)=(r_*,\theta_0,s),
\qquad s\in[0,\pi],
\end{equation}
has $g_\infty$-length exactly $\pi$. Hence
\begin{equation}
d_{g_\infty}(p,q)=\pi.
\end{equation}

We now estimate $d_{g_i}(p,q)$ from above by an explicit shortcut through
the thin well. Note that the well is placed near
\begin{equation}
\rho_i=\frac{1}{10}e^{-A_i^3},
\end{equation}
and that
\begin{equation}
\varphi_i(\rho_i)=e^{-A_i}\to0.
\end{equation}
Consider the piecewise smooth curve
\begin{equation}
(r_*,\theta_0,0)
\longrightarrow
(\rho_i,\theta_0,0)
\longrightarrow
(\rho_i,\theta_0,\pi)
\longrightarrow
(r_*,\theta_0,\pi).
\end{equation}
The two radial pieces have total length
\begin{equation}
2\int_{\rho_i}^{r_*}\frac{dr}{\varphi_i(r)} = 2 r_* + o(1).
\end{equation}
Here we used Lemma \ref{lem: radial-estimate} to obtain the asymptotic estimate.

The middle fiber segment lies at \(r=\rho_i\), where $\varphi_i(\rho_i)=e^{-A_i}.$
Since  $|\partial_\xi|_{g_i}=\varphi_i,$
the length of this fiber segment is
\begin{equation}
\int_0^\pi \varphi_i(\rho_i)\,d\xi
=
\pi e^{-A_i}
\to0.
\end{equation}
Therefore
\begin{equation}
d_{g_i}(p,q)
\le
2\int_{\rho_i}^{r_*}\frac{dr}{\varphi_i(r)}
+
\pi\varphi_i(\rho_i)
=
2r_*+o(1).
\end{equation}
By taking the limit, after passing to a subsequence, this implies
\begin{equation}
d_*(p,q)
\le
2r_*.
\end{equation}
Since \(2r_*<\pi\), while $d_{g_\infty}(p,q)=\pi$, 
we obtain
\begin{equation}
d_*(p,q)<d_{g_\infty}(p,q).
\end{equation}
This proves the strict inequality.
\end{proof}

\bibliographystyle{plain}
\bibliography{ref.bib}

@book {ColdingMinicozziBook,
    AUTHOR = {Colding, Tobias Holck and Minicozzi, II, William P.},
     TITLE = {A course in minimal surfaces},
    SERIES = {Graduate Studies in Mathematics},
    VOLUME = {121},
 PUBLISHER = {American Mathematical Society, Providence, RI},
      YEAR = {2011},
     PAGES = {xii+313},
      ISBN = {978-0-8218-5323-8},
   MRCLASS = {53A10 (35J93 49Q05)},
  MRNUMBER = {2780140},
MRREVIEWER = {Andrew\ Bucki},
       DOI = {10.1090/gsm/121},
       URL = {https://doi.org/10.1090/gsm/121},
}

@article{BasilioDodziukSormani-sewing,
	author = {Basilio, J. and Dodziuk, J. and Sormani, C.},
	doi = {10.1007/s12220-017-9969-y},
	fjournal = {Journal of Geometric Analysis},
	issn = {1050-6926},
	journal = {J. Geom. Anal.},
	mrclass = {53C23},
	mrnumber = {3881982},
	mrreviewer = {Sylvain Maillot},
	number = {4},
	pages = {3553--3602},
	title = {Sewing {R}iemannian manifolds with positive scalar curvature},
	url = {https://doi-org.ezproxy.gc.cuny.edu/10.1007/s12220-017-9969-y},
	volume = {28},
	year = {2018}}

@incollection{Sormani-conjectures,
	author = {Sormani, Christina},
	booktitle = {Perspectives on Scalar Curvature},
	pages = {645-722},
	publisher = {World Scientific},
	title = {Conjectures on Convergence and Scalar Curvature},
	year = {2023}}

@article{Gromov-Plateau,
	author = {Gromov, Misha},
	doi = {10.2478/s11533-013-0387-5},
	fjournal = {Central European Journal of Mathematics},
	issn = {1895-1074},
	journal = {Cent. Eur. J. Math.},
	mrclass = {53C23 (31C12 32Q28 53C20 53C40 58B20)},
	mrnumber = {3188456},
	mrreviewer = {Harish Seshadri},
	number = {7},
	pages = {923--951},
	title = {Plateau-{S}tein manifolds},
	url = {https://doi-org.ezproxy.gc.cuny.edu/10.2478/s11533-013-0387-5},
	volume = {12},
	year = {2014}}

@article{Tian2024,
	journal = {arXiv:2406.08592},
	author = {Wenchuan Tian},
	eprint = {2406.08592},
	primaryclass = {math.DG},
	title = {More Extreme Limits of Manifolds with Positive Scalar Curvature},
	year = {2024}}

@article{SormaniTianYeung2025,
	journal = {arXiv:2506.12491},
	author = {Christina Sormani and Wenchuan Tian and Wai-Ho Yeung},
	eprint = {2506.12491},
	primaryclass = {math.DG},
	title = {Geometric Convergence to an Extreme Limit Space with nonnegative scalar curvature},
	year = {2025}}

@article {AmbrosioKirchheim2000,
    AUTHOR = {Ambrosio, Luigi and Kirchheim, Bernd},
     TITLE = {Currents in metric spaces},
   JOURNAL = {Acta Math.},
  FJOURNAL = {Acta Mathematica},
    VOLUME = {185},
      YEAR = {2000},
    NUMBER = {1},
     PAGES = {1--80},
      ISSN = {0001-5962,1871-2509},
   MRCLASS = {49Q15 (49Q20)},
  MRNUMBER = {1794185},
MRREVIEWER = {Giovanni\ Bellettini},
       DOI = {10.1007/BF02392711},
       URL = {https://doi.org/10.1007/BF02392711},
}

@book{Hebey1999book,
  author    = {Hebey, Emmanuel},
  title     = {Nonlinear Analysis on Manifolds: Sobolev Spaces and Inequalities},
  series    = {Courant Lecture Notes in Mathematics},
  volume    = {5},
  publisher = {New York University, Courant Institute of Mathematical Sciences;
               American Mathematical Society},
  address   = {New York; Providence, RI},
  year      = {1999}
}

@article{Lions1985,
  author  = {Lions, Pierre-Louis},
  title   = {The concentration-compactness principle in the calculus of variations.
             The limit case, Part I},
  journal = {Revista Matemática Iberoamericana},
  volume  = {1},
  number  = {1},
  year    = {1985},
  pages   = {145--201}
}

@article {LeeNaberNeumayer2023,
    AUTHOR = {Lee, Man-Chun and Naber, Aaron and Neumayer, Robin},
     TITLE = {{$d_p$}-convergence and {$\epsilon$}-regularity theorems for
              entropy and scalar curvature lower bounds},
   JOURNAL = {Geom. Topol.},
  FJOURNAL = {Geometry \& Topology},
    VOLUME = {27},
      YEAR = {2023},
    NUMBER = {1},
     PAGES = {227--350},
      ISSN = {1465-3060,1364-0380},
   MRCLASS = {53C21},
  MRNUMBER = {4584264},
MRREVIEWER = {Ovidiu\ Munteanu},
       DOI = {10.2140/gt.2023.27.227},
       URL = {https://doi.org/10.2140/gt.2023.27.227},
}

@misc{Gromov2017Questions,
	author = {Gromov, Misha},
	note = {Available on the author's webpage},
	title = {101 Questions, Problems and Conjectures around Scalar Curvature},
	year = {2017}}

@article{ParkTianWang2018,
	author = {Park, Jiewon and Tian, Wenchuan and Wang, Changliang},
	journal = {Pure and Applied Mathematics Quarterly},
	number = {3--4},
	pages = {529--561},
	title = {A compactness theorem for rotationally symmetric {R}iemannian manifolds with positive scalar curvature},
	volume = {14},
	year = {2018}}

@article{SormaniTianWang2024,
	author = {Sormani, Christina and Tian, Wenchuan and Wang, Changliang},
	journal = {Nonlinear Analysis},
	pages = {Paper No. 113427},
	title = {An extreme limit with nonnegative scalar curvature},
	volume = {239},
	year = {2024}}

@article{TianWang2024,
	author = {Tian, Wenchuan and Wang, Changliang},
	journal = {Mathematische Annalen},
	number = {2},
	pages = {2767--2823},
	title = {Compactness of sequences of warped product circles over spheres with nonnegative scalar curvature},
	volume = {390},
	year = {2024}}

@article{Deng2021Volumic,
	author = {Deng, Jialong},
	journal = {SIGMA. Symmetry, Integrability and Geometry: Methods and Applications},
	pages = {013},
	title = {Curvature-dimension condition meets {G}romov's $n$-volumic scalar curvature},
	volume = {17},
	year = {2021}}

@article{LeFlochMardare2007,
	author = {LeFloch, Philippe G. and Mardare, Cristinel},
	journal = {Portugaliae Mathematica},
	pages = {535--573},
	title = {Definition and weak stability of spacetimes with distributional curvature},
	volume = {64},
	year = {2007}}

@article{KazarasXu2023,
	author = {Kazaras, Demetre and Xu, Kai},
	journal = {arXiv:2309.03756},
	title = {Drawstrings and flexibility in the Georch conjecture},
	year = {2023}}

@article{McFeronSzekelyhidi2012,
	author = {McFeron, Donovan and Sz{\'e}kelyhidi, G{\'a}bor},
	journal = {Communications in Mathematical Physics},
	pages = {425--443},
	title = {On the positive mass theorem for manifolds with corners},
	volume = {313},
	year = {2012}}

@article{AllenTianWang2026,
	author = {Allen, Brian and Tian, Wenchuan and Wang, Changliang},
	journal = {Calculus of Variations and Partial Differential Equations},
	pages = {Article No. 12},
	title = {On the scalar curvature compactness conjecture in the conformal case},
	volume = {65},
	year = {2026}}

@article{BurkhardtGuim2019,
	author = {Burkhardt-Guim, Paula},
	journal = {Geometric and Functional Analysis},
	number = {6},
	pages = {1703--1772},
	title = {Pointwise lower scalar curvature bounds for $C^0$ metrics via regularizing {R}icci flow},
	volume = {29},
	year = {2019}}

@article{Miao2002,
	author = {Miao, Pengzi},
	journal = {Advances in Theoretical and Mathematical Physics},
	number = {6},
	pages = {1163--1182},
	title = {Positive mass theorem on manifolds admitting corners along a hypersurface},
	volume = {6},
	year = {2002}}

@article{SormaniWenger2011,
	author = {Sormani, Christina and Wenger, Stefan},
	journal = {Journal of Differential Geometry},
	number = {1},
	pages = {117--199},
	title = {The intrinsic flat distance between {R}iemannian manifolds and other integral current spaces},
	volume = {87},
	year = {2011}}

@article{LeeLeFloch2015,
	author = {Lee, Dan A. and LeFloch, Philippe G.},
	journal = {Communications in Mathematical Physics},
	pages = {99--120},
	title = {The positive mass theorem for manifolds with distributional curvature},
	volume = {339},
	year = {2015}}

@article{JiangShengZhang2023,
	author = {Jiang, Wenshuai and Sheng, Weimin and Zhang, Huaiyu},
	journal = {Science China Mathematics},
	pages = {1141--1160},
	title = {Weak scalar curvature lower bounds along {R}icci flow},
	volume = {66},
	year = {2023}}

@book {BBI-book,
    AUTHOR = {Burago, Dmitri and Burago, Yuri and Ivanov, Sergei},
     TITLE = {A course in metric geometry},
    SERIES = {Graduate Studies in Mathematics},
    VOLUME = {33},
 PUBLISHER = {American Mathematical Society, Providence, RI},
      YEAR = {2001},
     PAGES = {xiv+415},
      ISBN = {0-8218-2129-6},
   MRCLASS = {53C23},
  MRNUMBER = {1835418},
MRREVIEWER = {Mario\ Bonk},
       DOI = {10.1090/gsm/033},
       URL = {https://doi.org/10.1090/gsm/033},
}

@article {LakzianSormani2013,
    AUTHOR = {Lakzian, Sajjad and Sormani, Christina},
     TITLE = {Smooth convergence away from singular sets},
   JOURNAL = {Comm. Anal. Geom.},
  FJOURNAL = {Communications in Analysis and Geometry},
    VOLUME = {21},
      YEAR = {2013},
    NUMBER = {1},
     PAGES = {39--104},
      ISSN = {1019-8385,1944-9992},
   MRCLASS = {53C23},
  MRNUMBER = {3046939},
MRREVIEWER = {Luca\ Granieri},
       DOI = {10.4310/CAG.2013.v21.n1.a2},
       URL = {https://doi.org/10.4310/CAG.2013.v21.n1.a2},
}

@article{gray1974volume,
  title={The volume of a small geodesic ball of a Riemannian manifold.},
  author={Gray, Alfred},
  journal={Michigan Mathematical Journal},
  volume={20},
  number={4},
  pages={329--344},
  year={1974},
  publisher={University of Michigan, Department of Mathematics}
}

\end{document}